\documentclass[leqno]{amsart}
\usepackage{amssymb}
\usepackage{mathrsfs}
\usepackage{amsmath, amsfonts, vmargin, enumerate}
\usepackage{graphicx}
\usepackage{color}
\usepackage{verbatim}
\usepackage{amsthm}
\usepackage{latexsym, bm}
\usepackage[latin1]{inputenc}
\usepackage[T1]{fontenc}
\usepackage[backref]{hyperref}

\usepackage{euscript}
\usepackage{dsfont}

\input xypic

\xyoption {all}

 \makeatletter

\newtheorem{thm}{Theorem}[section]
\newtheorem{prop}[thm]{Proposition}
\newtheorem{cor}[thm]{Corollary}
\newtheorem{lem}[thm]{Lemma}
\newtheorem{defi}[thm]{Definition}
\newtheorem{remark}[thm]{Remark}
\newtheorem{example}[thm]{Example}
\newtheorem{pb}[thm]{Problem}
\newtheorem{conj}[thm]{Conjecture}

\newenvironment{rk}{\begin{remark}\rm}{\end{remark}}
\newenvironment{definition}{\begin{defi}\rm}{\end{defi}}
\newenvironment{ex}{\begin{example}\rm}{\end{example}}
\newenvironment{problem}{\begin{pb}\rm}{\end{pb}}
\newenvironment{conjecture}{\begin{conj}\rm}{\end{conj}}

\numberwithin{equation}{section}

\newcommand{\real}{{\mathbb R}}
\newcommand{\nat}{{\mathbb N}}
\newcommand{\ent}{{\mathbb Z}}
\newcommand{\com}{{\mathbb C}}
\newcommand{\un}{{\mathds {1}}}

\newcommand{\T}{{\mathbb T}}

\newcommand{\A}{{\mathcal A}}

\newcommand{\C}{{\mathcal C}}
\newcommand{\D}{{\mathcal D}}
\newcommand{\E}{{\mathbb E}}
\newcommand{\F}{{\mathcal F}}
\newcommand{\G}{{\mathcal G}}
\renewcommand{\H}{{\mathcal H}}

\renewcommand{\i}{{\rm i}}

\renewcommand{\a}{\alpha}
\renewcommand{\b}{\beta}
\newcommand{\g}{\gamma}
\newcommand{\Ga}{\Gamma}

\renewcommand{\d}{\delta}
\renewcommand{\t}{\theta}
\newcommand{\e}{\varepsilon}
\newcommand{\f}{\varphi}
\newcommand{\p}{\psi}

\renewcommand{\l}{\lambda}
\renewcommand{\O}{\Omega}
\renewcommand{\o}{\omega}
\newcommand{\s}{\sigma}
\newcommand{\Si}{\Sigma}

\newcommand{\ot}{\otimes}

\newcommand{\8}{\infty}
\newcommand{\el}{\ell}

\newcommand{\la}{\langle}
\newcommand{\ra}{\rangle}
\newcommand{\wt}{\widetilde}
\newcommand{\wh}{\widehat}
\newcommand{\n}{\noindent}

\newcommand{\les}{\lesssim}
\newcommand{\ges}{\gtrsim}
\newcommand{\cc}{\mathsf{c}}
\renewcommand{\tt}{\mathsf{t}}

\newcommand{\be}{\begin{align*}}
\newcommand{\ee}{\end{align*}}
\newcommand{\beq}{\begin{equation}}
\newcommand{\eeq}{\end{equation}}
\newcommand{\beqn}{\begin{equation*}}
\newcommand{\eeqn}{\end{equation*}}

\begin{document}

\title[Vector-valued Littlewood-Paley-Stein theory]{Holomorphic functional calculus and vector-valued Littlewood-Paley-Stein theory for semigroups}


\author[Quanhua  Xu]{Quanhua Xu}
\address{Institute for Advanced Study in Mathematics, Harbin Institute of Technology,  Harbin 150001, China; and
Laboratoire de Math{\'e}matiques, Universit{\'e} de Franche-Comt{\'e}, 25030 Besan\c{c}on Cedex, France}
\email{qxu@univ-fcomte.fr}

\date{}
\maketitle

\centerline{\it Dedicated to the memory of Elias M. Stein}

 \begin{abstract}
 We study vector-valued Littlewood-Paley-Stein theory for semigroups of regular contractions $\{T_t\}_{t>0}$ on $L_p(\Omega)$ for a \textit{fixed} $1<p<\infty$. We prove that if a Banach space $X$ is of martingale cotype $q$, then there is a constant $C$ such that
   $$
  \left\|\left(\int_0^\infty\big\|t\frac{\partial}{\partial t}P_t (f)\big\|_X^q\,\frac{dt}t\right)^{\frac1q}\right\|_{L_p(\Omega)}\le C\, \big\|f\big\|_{L_p(\Omega; X)}\,,
  \quad\forall\, f\in L_p(\Omega; X),$$
where $\{P_t\}_{t>0}$  is the Poisson semigroup subordinated to $\{T_t\}_{t>0}$. Let $\mathsf{L}^P_{\cc, q, p}(X)$ be the least constant $C$, and let  $\mathsf{M}_{\cc, q}(X)$ be the martingale cotype $q$ constant of $X$. We show
 $$\mathsf{L}^{P}_{\cc,q, p}(X)\lesssim \max\big(p^{\frac1{q}},\, p'\big) \mathsf{M}_{\cc,q}(X).$$
Moreover, the order $\max\big(p^{\frac1{q}},\, p'\big)$ is optimal as $p\to1$ and $p\to\infty$. If $X$ is of martingale type $q$, the reverse inequality holds. If additionally $\{T_t\}_{t>0}$ is analytic on $L_p(\Omega; X)$, the semigroup $\{P_t\}_{t>0}$ in these results can be replaced by $\{T_t\}_{t>0}$ itself.

Our new approach is built on holomorphic functional calculus.  Compared with all the previous  approaches, ours is more powerful in several aspects: a) it permits us to go much further beyond the setting of symmetric submarkovian semigroups; b) it yields the optimal orders of growth on $p$ for most of the relevant constants;  c) it gives new insights into the scalar case for which our orders of the best constants in the classical Littlewood-Paley-Stein inequalities for symmetric submarkovian semigroups are better than the previous by Stein.

In particular, we resolve a problem of Naor and Young on the optimal order of the best constant in the above inequality when $X$ is of martingale cotype $q$ and $\{P_t\}_{t>0}$ is the classical Poisson and heat semigroups  on $\mathbb{R}^d$.
 \end{abstract}

\bigskip


\section{Introduction}


This article pursues our investigation on the vector-valued Littlewood-Paley-Stein theory that was initiated in \cite{LP0} and further carried out in \cite{LP1, LP3, LP2}. Our research in this domain has been profoundly influenced by Stein's monograph \cite{stein} and developed in two parallel directions. On the one hand, it deals with the Banach space valued case as in the just quoted articles as well as in the present one; and on the other hand, it extends Littlewood-Paley-Stein theory to the noncommutative setting (see \cite{JX, LMX} for maximal function inequalities and \cite{JLMX} for square function inequalities).

Note  that Betancor and coauthors studied this theory for some special semigroups (cf. \cite{betancor1, betancor1b, betancor2, betancor3}); see also \cite{AFST, betancor-1, betancor0, HTV, Hy, OX,TZ} for related results. Recently,  the theory has found applications to Lipschitz embedding of metric spaces into Banach spaces, and to approximation of Lipschitz maps by linear maps, see, for instance,  the papers by Hyt\"onen and Naor  \cite{HN},  Lafforgue and Naor \cite{LaNa}, Naor and Young  \cite{NaYo}.

\medskip

First, we recall the famous Littlewood-Paley-Stein inequality that is the starting point of all our research in the domain. Let $(\O, \A, \mu)$ be  a $\s$-finite measure space and $\{T_t\}_{t>0}$  a symmetric diffusion semigroup on $(\O,\A, \mu)$ in Stein's sense \cite[section~III.1]{stein}. Namely, $\{T_t\}_{t>0}$ satisfies the following conditions
 \begin{itemize}
 \item[$\bullet$] $T_t$ is a contraction on $L_p(\O)$ for every $1\le p\le \8$,
  \item[$\bullet$] $T_tT_s=T_{t+s}$,
 \item[$\bullet$] $\lim_{t\to 0}T_t(f)=f$ in $L_2(\O)$ for every $f\in L_2(\O)$,
 \item[$\bullet$] $T_t$ is positive (i.e. positivity preserving),
 \item[$\bullet$] $T_t$ is  selfadjoint on $L_2(\O)$,
 \item[$\bullet$] $T_t(1)=1$.
\end{itemize}
The last condition is the markovianity; the next to last is the symmetry. Thus such a semigroup is also called a {\it symmetric markovian semigroup}. A semigroup satisfying all the above conditions except markovianity is usually called a {\it symmetric submarkovian semigroup} (the submarkovianity means $T_t(1)\le1$).

It is a classical fact that the orthogonal projection $\mathsf F$ from $L_2(\O)$ onto the  fixed point subspace of $\{T_t\}_{t> 0}$ extends to a contractive projection on $L_p(\Omega)$ for
every $1\le p\le\infty$. Then $\mathsf F$ is also positive and   $\mathsf F\big(L_p(\Omega)\big)$ is the fixed point subspace of $\{T_t\}_{t> 0}$ on $L_p(\Omega)$.

\medskip

Stein's celebrated extension of the classical Littlewood-Paley inequality asserts that for every symmetric diffusion semigroup $\{T_t\}_{t>0}$ and every $1<p<\8$
 \beq\label{LPS}
 \|f-\mathsf F(f)\|_{L_p(\O)}\approx_p \left\|\left(\int_0^\8\Big |t\frac{\partial}{\partial t} T_t (f)\Big|^2\,\frac{dt}t\right)^{\frac12}\right\|_{L_p(\O)}\,,\quad f\in L_p(\O).
 \eeq
The classical inequality corresponds to the case where $\{T_t\}_{t> 0}$ is the Poisson semigroup on the torus $\T$ or the Euclidean space $\real^d$. Stein's inequality above is the core of  \cite{stein} in which Stein developed a beautiful general  theory. Later, Cowling \cite{Cow} presented an elegant alternative approach to Stein's theory for symmetric submarkovian semigroups; Cowling's goal is to show that the negative generator of $\{T_t\}_{t> 0}$ has a bounded holomorphic functional calculus, then to deduce the maximal inequality on $\{T_t\}_{t> 0}$ which is another fundamental result of Stein.

\medskip

In the present article we are concerned  with the vector-valued case. Given a Banach space $X$ let $L_p(\O;X)$ denote the $L_p$-space of strongly measurable functions from $\O$ to $X$. It is a well known elementary fact that if $T$ is a positive bounded operator on $L_p(\O)$ with $1\le p\le\8$, then $T\ot{\rm Id}_X$ is bounded on $L_p(\O; X)$ with the same norm. For notational convenience, throughout this article, we will denote $T\ot{\rm Id}_X$ by $T$ too. Thus $\{T_t\}_{t>0}$ is also a semigroup of contractions on $L_p(\O; X)$ for any Banach space $X$ with  $\mathsf F\big(L_p(\Omega; X)\big)$ as its fixed point subspace.

The vector-valued Littlewood-Paley-Stein theory consists in investigating \eqref{LPS} for $f\in L_p(\O; X)$ (with the absolute value on the right hand side replaced by the norm of $X$). It is
not hard to show that the equivalence \eqref{LPS}  continues to hold in the $X$-valued setting for the Poisson semigroup on $\T$ iff $X$ is isomorphic to a Hilbert space (cf. \cite{HY, LP0}). However, if one requires only the validity of one of the two one-sided inequalities, the corresponding family of Banach spaces is much larger: the upper estimate corresponds to 2-uniformly smooth  spaces while the lower one to  2-uniformly convex spaces (up to a renorming).

These geometrical properties of Banach spaces can be characterized by martingale inequalities. Recall that a Banach space $X$ is of \emph{martingale cotype} $q$ (with $2\le q<\8$) if there exists a positive constant $c$ such that every finite $X$-valued $L_q$-martingale $(f_n)$ satisfies the following inequality
 $$\sum_n\mathbb{E}\big\|f_n-f_{n-1}\big\|_X^q\le c^q\sup_n\mathbb{E}\big\|f_n\big\|_X^q\,,$$
where $\mathbb{E}$ denotes the underlying expectation.  $X$ is of  \emph{martingale type} $q$ (with $1<q\le2$) if the reverse inequality holds (with $c^{-1}$ in place of $c$). The corresponding best constant will be denoted by $\mathsf{M}_{\cc,q}(X)$ for the martingale cotype $q$ and by $\mathsf{M}_{\tt,q}(X)$ for the martingale type $q$. Pisier's famous renorming theorem asserts that  $X$ is of martingale type (resp. cotype) $q$ iff $X$ admits an equivalent norm that is $q$-uniformly smooth (resp. convex). We refer the reader to \cite{pis1,pis2b, pis3} for more information.

\medskip

Note that in the study of one-sided inequalities in the vector-valued case, the index $2$ on the right hand side of  \eqref{LPS} plays no special role and can be replaced by $1<q<\8$, $q\le 2$ for the upper estimate and $q\ge 2$ for the lower. Now we can summarize the main results of \cite{LP0,LP1,LP2} as follows.

\medskip\n{\bf Theorem~A}.  \emph{Let $X$ be a Banach space and $1<q<\8$.
\begin{enumerate}[\rm(i)]
 \item  $X$ is of martingale cotype $q$ iff for every  symmetric diffusion semigroup $\{T_t\}_{t>0}$ and for every $1<p<\8$ $($equivalently, for some $1<p<\8)$ there exists a constant $c$ such that
  $$
  \left\|\left(\int_0^\8\big\|t\frac{\partial}{\partial t}T_t (f)\big\|_X^q\,\frac{dt}t\right)^{\frac1q}\right\|_{L_p(\O)}\le c\, \big\|f\big\|_{L_p(\O; X)}\,,
  \quad  f\in L_p(\O; X).
  $$
 \item  $X$ is of martingale type $q$ iff  for every  symmetric diffusion semigroup $\{T_t\}_{t>0}$  and for every $1<p<\8$ $($equivalently, for some $1<p<\8)$  there exists a constant $c$ such that
   $$
   \big\|f-\mathsf F (f)\big\|_{L_p(\O; X)}\le c \left\|\left(\int_0^\8\big\|t\frac{\partial}{\partial t} T_t (f)\big\|_X^q\,\frac{dt}t\right)^{\frac1q}\right\|_{L_p(\O)}\,,
   \quad f\in L_p(\O; X).
   $$
  \end{enumerate}}

Note that $\mathsf F (f)$ does not contribute to the norm on the left hand side of the inequality in (i) above since $\frac{\partial}{\partial t}T_t (\mathsf F (f))=0$ for any $t>0$; so if this inequality holds, it automatically holds with $f$ replaced by $f-\mathsf F (f)$ on the right hand side. In the sequel, when cotype inequalities will be considered, we will often use simply $f$ instead $f-\mathsf F (f)$ as in (i). However, for the type inequalities as in (ii), we must use $f-\mathsf F (f)$ on the left hand side.

\medskip

Both ``if'' parts in the above theorem are proved in \cite{LP0}; for that purpose we need only the case where $\{T_t\}_{t>0}$ is the usual Poisson semigroup on $\T$ (or $\real^d$ as in \cite{LP1}). This is the easy direction thanks to the classical link between Poisson integral and Brownian motion. The other direction is harder. It is first proved  in \cite{LP0} for the Poisson semigroup on the unit circle, then in \cite{LP1} for the Poisson semigroup subordinated to any symmetric diffusion semigroup $\{T_t\}_{t>0}$. Left as an open problem in \cite{LP1}, the statement for $\{T_t\}_{t>0}$ itself as above was finally settled in \cite{LP2}. Note that like in  \cite{stein}, the key tool in \cite{LP1,LP2} is Rota's martingale dilation of a  symmetric diffusion semigroup that allows us to adapt the scalar Littlewood-Paley-Stein theory developed in \cite{stein}.

\medskip

The use of Rota's dilation prevented us from weakening the assumption on a symmetric diffusion semigroup. Cowling's approach in \cite{Cow} does not use Rota's dilation but it requires the semigroup in consideration to be symmetric and submarkovian. It has been an open problem of establishing the results of \cite{LP1} or \cite{LP2}  in Cowling's setting. In fact, since a long time it has been a desire to extend all previous results to more general semigroups. This was done in some special cases (Hermite, Laguerre and Bessel semigroups) by Betancor and  coauthors (cf. \cite{betancor1, betancor1b, betancor2, betancor3}).

\medskip

The objective of the present article is to resolve the above problems. We will develop a vector-valued  Littlewood-Paley-Stein theory for semigroups of regular operators on $L_p(\O)$ for a \textit{single} $1<p<\8$, thereby going considerably beyond Stein-Cowling's setting.

Recall that an operator $T$ on $L_p(\O)$ ($1\le p\le\8$)  is {\it regular} if there exists a constant $c$ such that
 $$\big\|\sup_k|T(f_k)|\big\|_p\le c\, \big\|\sup_k|f_k|\big\|_p$$
for all finite sequences $\{f_k\}_{k\ge1}$ in $L_p(\O)$. The least constant $c$ is called the {\it regular norm} of $T$. Obviously, any positive operator is regular with regular norm equal to its operator norm. It is well known that, conversely, if $T$ is regular, then there exists a positive operator $S$ on  $L_p(\O)$  such that $|T(f)|\le S(|f|)$
for any $f\in L_p(\O)$ with  $\|S\|$ equal to the regular norm of $T$; such a positive $S$ is unique and called the absolute value of $T$ and denoted by $|T|$ (see \cite[Chapter~1]{MN}). For  presentation simplicity, in this article we will only consider {\it contractively regular} operators, i.e., those with regular norms less than or equal to $1$, and will simply call these operators as regular operators with a light abuse of terminology.

It is well known (and easy to check) that if $T$ is a contraction on $L_p(\O)$ for every $1\le p\le\8$, then $T$ is regular on $L_p(\O)$. Like positive operators, a regular operator $T$  extends to a contraction on $L_p(\O;X)$ for any Banach space $X$. This extension will be denoted by $T$ too.

\smallskip

Now let $\{T_t\}_{t>0}$ be a strongly continuous semigroup of regular operators on $L_p(\O)$ with $1<p<\8$. Extended to $L_p(\O;X)$, $\{T_t\}_{t>0}$ remains to be a strongly continuous semigroup of contractions on $L_p(\O;X)$.
Let again $\mathsf F$ be the projection from $L_p(\O)$ onto the  fixed point subspace of
$\{T_t\}_{t> 0}$. Then $\mathsf F$ is also regular, so extends to a contractive projection on $L_p(\O;X)$. Note that $\mathsf F(L_p(\O;X))$ coincides with the  fixed point subspace of $\{T_t\}_{t> 0}$ on $L_p(\O;X)$.

Let $\{P_t\}_{t>0}$ be the Poisson semigroup subordinated to $\{T_t\}_{t>0}$:
 \beq\label{subordination}
 P_t(f)=\frac{1}{\sqrt\pi}\,\int_0^\infty \frac{e^{-s}}{\sqrt s}\,T_{\frac{t^2}{4s}}(f)ds.
 \eeq
Recall that if $A$ denotes the negative infinitesimal generator of $\{T_t\}_{t>0}$ (i.e., $T_t=e^{-tA}$), then $P_t=e^{-t\sqrt A}$.
Instead of the square root, one can, of course, consider other subordinated semigroups $e^{-tA^\a}$ for $0<\a<1$; but we will not deal with the latter here.

\medskip

To proceed further, we need to introduce some notions. Define
 $$\mathcal{G}^T_q(f)=\Big(\int_0^\8\big\|t\frac{\partial}{\partial t}T_t(f)\big\|^q_X\frac{dt}t\Big)^{\frac1q}$$
for $f$ in the definition domain of  $A$ in $L_p(\O;X)$.  $X$ is said to be of {\it Luzin cotype $q$ relative to} $\{T_t\}_{t> 0}$ if there exists a constant $c$ such that
 $$
 \big\|\mathcal{G}^T_q(f)\big\|_{L_p(\O)}\le c\big\|f-\mathsf{F}(f)\big\|_{L_p(\O;X)}$$
for all $f$ as above.
The smallest $c$ is denoted by $\mathsf{L}^T_{\cc, q,p}(X)$.  Similarly, we define the  {\it Luzin type $q$} of $X$ by reversing the above inequality and changing $c$ to $c^{-1}$, the corresponding type $q$ constant is denoted by $\mathsf{L}^T_{\tt, q,p}(X)$. See section~\ref{Luzin type and cotype} below for more information.

\begin{rk}\label{Poisson vs Heat}
 The subordination formula \eqref{subordination} immediately implies the pointwise inequality $\mathcal{G}^P_q(f)\le C\,\mathcal{G}^T_q(f)$ for any $f$, where $C$ is an absolute positive constant. It then follows that
  $$\mathsf{L}^T_{\cc, q,p}(X)\ge C\, \mathsf{L}^P_{\cc, q,p}(X)\;\text{ and }\; \mathsf{L}^T_{\tt, q,p}(X)\le C\, \mathsf{L}^P_{\tt, q,p}(X).$$
 \end{rk}

In \cite{LP0}, the Luzin type and cotype relative to the Poisson semigroup on the unit circle are shown to be equivalent to the martingale type and cotype, respectively.  Theorem~A above extends this to symmetric diffusion semigroups.

In the sequel, we will use the following convention:  $A\lesssim B$ (resp. $A\lesssim_\e B$) means that $A\le C B$ (resp. $A\le C_\e B$) for some absolute positive constant $C$ (resp. a positive constant $C_\e$ depending only on $\e$).  $A\approx B$ or  $A\approx_\e B$ means that these inequalities as well as their inverses hold. The index $p$ will be assumed to satisfy $1<p<\8$ and  $p'$ will denote its conjugate index.

\medskip

Below is our first principal result.

\begin{thm}\label{Poisson ML}
 Let $X$ be a Banach space and $1<p, q<\8$. Let $\{T_t\}_{t>0}$ be a strongly continuous semigroup of regular operators on $L_p(\O)$ and $\{P_t\}_{t>0}$ its subordinated Poisson semigroup.
 \begin{enumerate}[\rm(i)]
 \item  If $X$ is of martingale cotype $q$, then $X$ is  of  Luzin cotype $q$ relative to $\{P_t\}_{t> 0}$ and
  $$\mathsf{L}^{P}_{\cc,q, p}(X)\les \max\big(p^{\frac1q},\, p'\big) \mathsf{M}_{\cc,q}(X).$$
 \item  If $X$ is of martingale type $q$, then $X$ is  of  Luzin type $q$ relative to $\{P_t\}_{t> 0}$ and
  $$\mathsf{L}^{P}_{\tt,q, p}(X)\les \max\big(p,\, p'^{\frac1{q'}}\big) \mathsf{M}_{\tt,q}(X).$$
   \end{enumerate}
  \end{thm}

The two above inequalities can be reformulated in another (clearer) way, for instance, the first one reads as
  \begin{eqnarray*}
 \mathsf{L}^{P}_{\cc,q, p}(X)\les
  \left\{\begin{array}{lcl}
\displaystyle p^{\frac1{q}}\,\mathsf{M}_{\cc,q}(X) & \textrm{ if}& p\ge q, \\
 \displaystyle p'\,\mathsf{M}_{\cc,q}(X)& \textrm{ if}& p<q.
  \end{array}\right.
  \end{eqnarray*}

\begin{rk}\label{optimal cotye}
 All the growth orders, except the one on $\mathsf{L}^{P}_{\tt,q, p}(X)$ as $p\to 1$, are optimal since they are already so in the scalar case $X=\com$. More precisely,
  \begin{enumerate}[\rm(i)]
  \item  $\mathsf{L}^{P}_{\cc,q, p}(\com)\ges\max\big(p^{\frac1q},\, p'\big)$ for all $1<p<\8$ when $\{P_t\}_{t> 0}$ is the classical Poisson semigroup on $\real$ (see Proposition~\ref{Optimality} below);
  \item  $\mathsf{L}^{P}_{\tt,q, p}(\com)\ges p$ as $p\to\8$ when $\{P_t\}_{t> 0}$ is the Poisson semigroup subordinated to a symmetric diffusion semigroup $\{T_t\}_{t> 0}$, as shown by  Zhendong Xu and Hao Zhang \cite{XZ}; in fact, they proved the stronger inequality $\mathsf{L}^{T}_{\tt,q, p}(\com)\ges p$ as $p\to\8$ for a symmetric diffusion semigroup $\{T_t\}_{t> 0}$.
     \end{enumerate}
  \end{rk}

Part (i) of the above theorem cannot hold for the semigroup $\{T_t\}_{t>0}$ itself without additional assumption (see Remark~\ref{necessity of analyticity} below). It turns out that the missing condition  is the analyticity of $\{T_t\}_{t>0}$ on $L_p(\O; X)$. Recall that $\{T_t\}_{t>0}$ is  analytic on $L_p(\O; X)$ if $\{T_t\}_{t>0}$  extends to  a bounded analytic function from an open sector $\Si_{\b_0}=\big\{z\in\com: |{\rm arg}(z)|<\b_0\big\}$ to $B\big(L_p(\O; X)\big)$ for some $0<\b_0\le\frac\pi2$, where $B(Y)$ denotes the space of bounded linear operators on a Banach space $Y$. In this case,
 \beq\label{Ana bound}
 \mathsf{T}_{\b_0}=\sup\big\{\big\|T_z\big\|_{B(L_p(\O;X))}: z\in\Sigma_{\b_0}\big\}<\8.
 \eeq

\begin{thm}\label{Heat ML}
 Let $X$ and $p, q$ be as above.
  \begin{enumerate}[\rm(i)]
  \item  If $X$ is of martingale type $q$, then $X$ is  of  Luzin type $q$ relative to $\{T_t\}_{t> 0}$ and
  $$\mathsf{L}^{T}_{\tt,q, p}(X)\les \max\big(p,\,( p')^{\frac1{q'}}\big) \mathsf{M}_{t,q}(X).$$
 \item  Assume additionally that $\{T_t\}_{t>0}$ satisfies \eqref{Ana bound}.  Let $\b_q=\b_0\min(\frac{p}{q}\,,\frac{p'}{q'})$.
 If $X$ is of martingale cotype $q$, then $X$ is  of  Luzin cotype $q$ relative to $\{T_t\}_{t> 0}$ and
 $$\mathsf{L}^{T}_{\cc,q, p}(X)\les\b_q^{-3}\,\mathsf{T}_{\b_0}^{\min(\frac{p}{q},\,\frac{p'}{q'})}\,\max\big(p^{\frac2q},\,(p')^{1+\frac1{q'}}\big) \mathsf{M}_{c,q}(X).$$
   \end{enumerate}
 \end{thm}

\medskip

The two previous theorems considerably improve Theorem~A. Firstly, the semigroup $\{T_t\}_{t>0}$ now acts on $L_p(\O)$ for a single $p$. Secondly, the markovianity or submarkovianity is not assumed (in fact, $T_t(1)$ is even not defined if the measure on $\O$ is infinite). Thirdly, the symmetry is not needed either since the semigroup does not act on $L_2(\O)$ if $p\neq2$.

Another improvement concerns the precise estimates on the best constants, the present estimates are much better than all previously known ones, even in the scalar case (see section~\ref{The scalar case revisited} below for historical comments). Moreover, except one case, they  give the optimal orders of growth as $p\to1$ and as $p\to\8$, as already pointed out in Remark~\ref{optimal cotye}. This is  perhaps a major novelty of our method.

\medskip

The aforementioned optimality  allows us to answer a question raised by Naor and Young about the optimal orders of $\mathsf{L}^{P}_{\cc,q, p}(X)$ and $\mathsf{L}^{T}_{\cc,q, p}(X)$ when $\{T_t\}_{t>0}$ is the heat semigroup on $\real^d$ (see the appendix of  \cite{NaYo1}). In fact, we will show a much stronger result. Let $\f:\real^d\to\com$ be an integrable  function satisfying
\beq\label{Holder}
 \left\{\begin{array}{lcl}
 \displaystyle |\f(x)|\le\frac{1}{(1+|x|)^{d+\e}}, & x\in\real^d\\
 \displaystyle  |\f(x)-\f(y)|\le \frac{|x-y|^\d}{(1+|x|)^{d+\e+\d}}+\frac{|x-y|^\d}{(1+|y|)^{d+\e+\d}},& x, y\in\real^d\\
  \displaystyle \int_{\real^d}\f(x)dx=0
  \end{array}\right.
  \eeq
 for some positive constants $\e$ and $\d$.

 We will also need $\f$  to be {\it nondegenerate}  in the sense that there exists another function $\p$ satisfying \eqref{Holder} such that
 \beq\label{reproduce}
  \int_0^\8 \wh\f(t\xi)\, \wh\p (t\xi)\, \frac{dt}{t} = 1,\quad \forall\xi\in\real^d\setminus\{0\}.
   \eeq
This  nondegeneracy allows us to use  the Calder\'on reproducing formula. There exist plenty of functions  satisfying these conditions, for instance, the kernel of $t\frac{\partial}{\partial t} T_t$, where $\{T_t\}_{t>0}$ is either the heat or Poisson semigroup on $\real^d$, as well as any Schwartz function $\f$ verifying that for any $\xi\in\real^d\setminus\{0\}$ there is $t>0$ such that $\wh\f(t\xi)\neq0$.

\medskip

Let $\f_t(x)=\frac1{t^d}\f(\frac{x}{t})$ for $x\in\real^d$ and $t>0$. We define
 \beq\label{G-function}
 G_{q,\f}(f)(x)=\Big(\int_0^\8\|\f_t*f(x)\|_X^q\,\frac{dt}t\Big)^{\frac1q}, \quad x\in\real^d
 \eeq
for any (reasonable) function $f:\real^d\to X$. Let $\mathsf{L}^{\f}_{\cc,q, p}(X)$ be the best constant $c$ such that
 $$\big\|G_{q,\f}(f)\big\|_{L_p(\real^d)}\le c\,\big\|f\big\|_{L_p(\real^d;X)},\quad f\in L_p(\real^d; X).$$
Similarly, we introduce $\mathsf{L}^{\f}_{\tt,q, p}(X)$ for the reverse inequality (with $c^{-1}$ instead of $c$).

\begin{thm}\label{fML}
 Let $X$ be a Banach space and $1<p, q<\8$. Assume that $\f$ satisfies \eqref{Holder}.
 \begin{enumerate}[\rm(i)]
 \item  If $X$ is of martingale cotype $q$, then
  $$\mathsf{L}^{\f}_{\cc,q, p}(X)\les_{d, \e, \d}\max\big(p^{\frac1q},\, p'\big)\mathsf{M}_{\cc,q}(X).$$
 \item Assume additionally that $\f$ is nondegenerate.
  If $X$ is of martingale type $q$, then
  $$\mathsf{L}^{\f}_{\tt,q, p}(X)\les_{d, \e, \d}\max\big(p,\, (p')^{\frac1{q'}}\big)\mathsf{M}_{\tt,q}(X).$$
  \end{enumerate}
  \end{thm}

Let  $\{\mathbb{H}_t\}_{t>0}$ be the classical heat semigroup on $\real^d$ whose convolution kernel  is given by
 $$\mathbb{H}_t(x)=(4\pi t)^{-\frac d2}e^{-\frac{|x|^2}{4t}}\,.$$
Its subordinated Poisson semigroup is the usual Poisson semigroup $\{\mathbb{P}_t\}_{t>0}$  with the convolution kernel
 $$\mathbb{P}_t(x)=\frac{c_d\,t}{(|x|^2+t^2)^{(d+1)/2}}\,.$$
The above theorem implies the following corollary that resolves Naor and Young's problem.

\begin{cor}\label{NY}
We have
 \begin{align*}
 \mathsf{L}^{\mathbb{P}}_{\cc,q, p}(X)
 \les\max\big(p^{\frac1q},\, p'\big) \mathsf{M}_{\cc,q}(X)\;\text{ and }\;
 \mathsf{L}^{\mathbb{H}}_{\cc,q, p}(X)
 \les_d\max\big(p^{\frac1q},\, p'\big) \mathsf{M}_{\cc,q}(X).
  \end{align*}
 Moreover,
  $$\mathsf{L}^{\mathbb{H}}_{\cc,q, p}(X)\ges\mathsf{L}^{\mathbb{P}}_{\cc,q, p}(\com)\ges\max\big(p^{\frac1q},\, p'\big)\;\text{ and }\;
  \mathsf{L}^{\mathbb{H}}_{\cc,q, q}(X)\ges \mathsf{L}^{\mathbb{P}}_{\cc,q, q}(X)
 \ges \mathsf{M}_{\cc,q}(X).$$
\end{cor}

\begin{rk}
 It is worth to point out that the estimate on $ \mathsf{L}^{\mathbb{P}}_{\cc,q, p}(X)$ is independent of $d$ thanks to Theorem~\ref{Poisson ML} (i). It would be interesting to have a dimension free estimate for the heat semigroup too. This is related to another problem, whether the analyticity constant of $\{\mathbb{H}_t\}_{t>0}$ on $L_p(\real^d; X)$ relative to an appropriate angle can be controlled by a dimension free constant (see Example~\ref{Laplacian} below).
\end{rk}

\begin{problem} (i) Does the second inequality in the first part of Corollary~\ref{NY} hold with a constant independent of the dimension $d$? 

(ii) It would be also interesting to determine the optimal orders of $ \mathsf{L}^{\mathbb{P}}_{\tt,q, p}(X)$ and  $\mathsf{L}^{\mathbb{H}}_{\tt,q, p}(X)$ as $p\to 1$ or $p\to \8$.
\end{problem}

Note that (i) above remains open even for $X=\com$ (see Problem~7 in \cite{LP-Optimality}). It is also so for (ii) as $p \to\8$ (see section~\ref{The scalar case revisited} below and  \cite{LP-Optimality} for more information).
\medskip

Apart from the inequality $\big\|\G^P_q(f)\big\|_{L_p(\O)}\le \mathsf{L}^{P}_{c,q, p}(X) \big\|f\big\|_{L_p(\O; X)}$,  the following variant is also useful:
 $$\Big(\int_0^\8\big\|t\frac{\partial}{\partial t}\, P_t(f)\big\|^r_{L_p(\O; X)}\,\frac{dt}t\Big)^{\frac1r}\le c \big\|f\big\|_{L_p(\O; X)}$$
when $X$ is of martingale cotype $q$ (see, for instance, \cite{LaNa}). Inequalities of this type are less delicate than the previous ones. It is well known that if $X$ is of martingale cotype $q$, then $L_p(\O; X)$ is of martingale cotype $\max(p, q)$, so the above inequality can hold only for $r=\max(p, q)$. We can, of course, consider similar variants in the situation of Theorem~\ref{Heat ML} and Theorem~\ref{fML} as well as their reverse inequalities  when $X$ is of martingale type $q$; but we will concentrate on the above inequality and on the case $1<p\le q$ for illustration.

\medskip

Theorem~\ref{Poisson ML} (i) easily implies the following

\begin{cor}\label{Poisson MLbis}
Let $\{T_t\}_{t>0}$ and $\{P_t\}_{t>0}$ be as in Theorem~\ref{Poisson ML}. Assume that $X$ is of martingale cotype $q$ and $1<p\le q$. Then
 $$\Big(\int_0^\8\big\|t\frac{\partial}{\partial t}\, P_t(f)\big\|^q_{L_p(\O; X)}\,\frac{dt}t\Big)^{\frac1q}\les \max\big((p')^{\frac1q},\,  \mathsf{M}_{\cc,q}(X)\big)\big\|f\big\|_{L_p(\O; X)}$$
for all $f\in L_p(\O; X)$. Moreover, the constant on the right hand side is optimal as $p\to1$.
  \end{cor}

Like the martingale type and cotype, the Luzin type and cotype behave well with duality as shown by Theorem~\ref{dual} below.  This duality theorem allows us to deduce the type case from the cotype case. In contrast with the martingale case, the proof of Theorem~\ref{dual} is much harder and depends on a bounded projection on a certain vector-valued radial tent space. Let $\real_+$ be equipped with the measure $\frac{dt}{t}$. The \textit{radial space} is $L_p(\O; L_q(\real_+;X))$ whose elements $h$ are  functions of two variables $\o\in\O$ and $t\in\real_+$, i.e., $h: (\o, t)\mapsto h_t(\o)$.  The desired projection maps $L_p(\O; L_q(\real_+; X))$ onto the subspace of all $h$ of the form $h_t=t \frac{\partial}{\partial t}  T_t(f)$ for some $f\in L_p(\O; X)$; formally,  it is given by
  \beq\label{def projection}
  \mathcal{T}(h)_s=4\int_0^\8  st\,\frac{\partial}{\partial s} T_s \frac{\partial}{\partial t}  T_t(h_t)\,\frac{dt}t\,,\quad s>0.
  \eeq
Here the expression $\frac{\partial}{\partial t}  T_t(h_t)$ is interpreted as $\frac{\partial}{\partial t}  T_t(f)$ with $f=h_t$. Note that $\mathcal{T}(h)$ is well-defined for nice functions $h\in L_p(\O; L_q(\real_+;X))$, for instance, for all compactly supported
continuous functions from $\real_+$ to the definition domain of the generator of $\{T_t\}_{t>0}$ in $L_p(\O; X)$. Similarly, we define $\mathcal{P}$ associated to  the subordinated Poisson semigroup $\{P_t\}_{t>0}$.

\medskip

The following is the key  result for the duality argument.

\begin{thm}\label{proj}
 Let $\{T_t\}_{t>0}$ be a strongly continuous semigroup of regular operators on $L_p(\O)$ and $\{P_t\}_{t>0}$ its subordinated Poisson semigroup. Let $X$ be a Banach space and $1<p<\8$.
  \begin{enumerate}[\rm(i)]
  \item The map $\mathcal{P}$ extends to a bounded projection on $L_p(\O; L_q(\real_+; X))$ with norm majorized by $C\max((p')^{1-\frac{p}q},\, p^{1-\frac{p'}{q'}})$ for any $1\le q\le\8$.
  \item Assume additionally that $1<q<\8$ and $\{T_t\}_{t>0}$  satisfies \eqref{Ana bound} for some $0<\b_0\le\frac\pi2$. Then $\mathcal{T}$ extends to a bounded projection on $L_p(\O; L_q(\real_+; X))$ with norm majorized by
      $$C \b_q^{-4}\,\mathsf{T}_{\b_0}^{\min(\frac{p}{q},\,\frac{p'}{q'})}\,\max((p')^{1-\frac{p}q},\, p^{1-\frac{p'}{q'}})\;\text{ with }\; \b_q=\b_0\min(\frac{p}{q}\,,\frac{p'}{q'}).$$
  \end{enumerate}
  \end{thm}

It is remarkable that the first part of the theorem above holds for any Banach space $X$ and any subordinated Poisson semigroup. Under the stronger assumption that $\{T_t\}_{t>0}$ be a symmetric diffusion semigroup, assertion  (i) above  is \cite[Theorem~3.2]{LP1}. However, the proof in  \cite{LP1} contains a gap which lies in the reduction of Theorem~3.2 to Lemma~3.3  in \cite{LP1} via Rota's dilation theorem. Recall that if $\{P_t\}_{t>0}$ is the Poisson semigroup on the torus, assertion (i) above was proved in \cite{LP0} by using Calder\'on-Zygmund singular integral theory.

\begin{rk}
The analyticity assumption in Theorem~\ref{Heat ML} and Theorem~\ref{proj} (ii) is unremovable. Recall that $\{T_t\}_{t>0}$ is  analytic on $L_p(\O;X)$  with $1<p<\8$ in one of the following cases:
 \begin{itemize}
 \item[$\bullet$] $\{T_t\}_{t>0}$ is a symmetric diffusion semigroup and $X$ is superreflexive (\cite{pis2});
 \item[$\bullet$] $\{T_t\}_{t>0}$ is a convolution semigroup induced by symmetric probability measures on a locally compact abelian group  and $X$ is K-convex (\cite{pis2});
 \item[$\bullet$] $\{T_t\}_{t>0}$ is an analytic  semigroup of regular operators on $L_p(\O)$ and $X$ is $\theta$-Hilbertian, i.e., a complex interpolation space of a Hilbert space and another Banach space (\cite{LP3}).
 \end{itemize}
 Many classical semigroups are analytic on $L_p(\O;X)$ for any $X$ (see Appendix~\ref{Examples} below). On the other hand, $\{T_t\}_{t>0}$ is analytic on $L_p(\O;X)$ iff its adjoint semigroup $\{T^*_t\}_{t>0}$ is analytic on $L_{p'}(\O;X^*)$. Thus the class of Banach spaces $X$ such that $\{T_t\}_{t>0}$ is analytic on $L_p(\O;X)$ is stable under the passage to duals, subspaces and quotient spaces.

On the other hand, it is well known that  the subordinated Poisson semigroup $\{P_t\}_{t>0}$ is always analytic on  $L_p(\O;X)$ for any Banach space $X$ since its  negative generator is a sectorial operator of type $\frac\pi4$ (see section~\ref{boundedness} for more information).
 \end{rk}

A summary of the main techniques and the contents seems to be in order. Our approach is different from all the previous ones. It is based on holomorphic functional calculus, which constitutes perhaps one of the major ideas of this article. In this regard, it shares some common points with Cowling's approach that deals with the bounded $H^\8$ functional calculus of the generator of a symmetric submarkovian semigroup and the related maximal inequality. We need, however, to adapt McIntosh's $H^\8$ functional calculus for our purpose. This is done in the preparatory section~\ref{boundedness} in which we introduce a key notion of the article: the $\el_q$-boundedness of a family of operators on $L_p(\O; X)$; it gives rise to the definitions of $\el_q$-sectorial operators and $\el_q$-analytic semigroups. We transfer to this setting some well known results about sectorial operators and analytic semigroups. After this preparation, we prove Theorem~\ref{proj} in section~\ref{Proof of Theorem proj}. This projection theorem is a crucial ingredient for the duality studied in section~\ref{Luzin type and cotype}.  Theorem~\ref{dual} establishes our duality result between the Luzin cotype of $X$ and the Luzin type of the dual space $X^*$; this result is as nice as the corresponding one in the martingale case, except the links between the involved constants (compare the constants in Theorem~\ref{dual} and those in \eqref{duality type-cotype} below). This section also contains some general properties of the Luzin type and cotype, in particular, a characterization by lacunary discrete differences (Theorem~\ref{LPS-diff}).

\medskip

As Rota's dilation is no longer available in the present situation, we use instead Fendler's dilation for semigroups of regular operators. Fendler's theorem transfers Theorem~\ref{Poisson ML} (i) to the special case where $\{T_t\}_{t>0}$ is the translation group of $\real$. This allows us to exploit techniques from harmonic analysis.  Our strategy is built, in a crucial way, on Calder\'on-Zygmund singular integral theory and modern real-variable Littlewood-Paley theory. We present all this in the preparatory section~\ref{Dyadic martingales and singular integrals} that will be   needed for the proofs of Theorem~\ref{Poisson ML} and Theorem~\ref{fML}. These proofs constitute the most heavy and technical part of the article. The proofs of Theorem~\ref{fML} and Corollary~\ref{NY} are done in section~~\ref{Proof of Theorem fML}. We then use transference to show Theorem~\ref{Poisson ML} in section~~\ref{Proof of Theorem Poisson ML}. To that end, we first need to represent the $g$-function associated to the Poisson semigroup subordinated to the translation group of $\real$ as a singular integral operator. Theorem~\ref{Heat ML} will follow from Theorem~\ref{Poisson ML} by functional calculus; Corollary~\ref{Poisson MLbis} is an easy consequence of Theorem~\ref{Poisson ML} (i).

\medskip

An additional major significant aspect of the new approach is the fact that it improves the growth orders on $p$ of the relevant best constants even in the scalar Littlewood-Paley-Stein inequalities, see section~\ref{The scalar case revisited}; moreover, except the case of the Luzin type constant as $p\to1$,  it yields  the optimal orders, which is not  the case by the previous methods of Stein and Cowling (see the historical comments in section~\ref{The scalar case revisited}). This shows, to a certain extent, that our method is optimal. Section~\ref{The scalar case revisited} also contains the optimality of the best constants in Corollary~\ref{NY}. We end the article by an appendix that gives some examples of semigroups.

\medskip

The techniques developed in this article allow one to simplify and extend many recent results in the scalar Littlewood-Paley-Stein theory, in particular, those on positive operators on $L_2$ with kernels satisfying Gaussian upper estimates. On the other hand, they are also applicable to  the noncommutative setting. We will carry out all this elsewhere.

\medskip

Throughout the article, $X$ will be a Banach space, $1<p<\8$ and $1\le q\le\8$ (but $1<q<\8$ most of time). Unless explicitly stated otherwise, $\{T_t\}_{t>0}$ will be a strongly continuous semigroup of regular operators on $L_p(\O)$ and $\{P_t\}_{t>0}$ its subordinated Poisson semigroup. These semigroups are extended to  $L_p(\O;X)$. $A$ will denote the negative generator of $\{T_t\}_{t>0}$, so $T_t=e^{-tA}$ and $P_t=e^{-t\sqrt A}$.


\section{The $\el_q$-boundedness}\label{boundedness}


This section is the preparatory part of the article. We will introduce the notion of  $\el_q$-boundedness that is the direct extension to the vector-valued setting of the $R_q$-boundedness introduced by Weis \cite{Weis01}. In fact, though not explicitly stated, this notion appeared before in harmonic analysis with regard to vector-valued inequalities of classical operators. Most results below are  the $\el_q$-boundedness analogues of well known results or those of  Kunstmann and Ullmann \cite{KU} in the scalar case. I learnt the existence of \cite{Weis01, KU} after the submission of this article for publication, and I thank Emiel Lorist for pointing out these references to me.

\medskip

We start with a brief introduction to holomorphic functional calculus in order to fix notation (see \cite{CDMY} for more information). Recall that a densely defined closed operator $B$ on a Banach space $Y$ is called a \emph{sectorial operator of type} $\a$ with $0\le \a<\pi$ if $\com\setminus \overline{\Sigma_\g}$ is contained in the resolvent set of $B$ for any $\g>\a$ and
 $$\sup\big\{\big\|z(z-B)^{-1}\big\|_{B(Y)}\,: z\notin \overline{\Sigma_\g}\big\}<\8,$$
where $\Sigma_\g$ is the open sector $\big\{z\in\com: |{\rm arg}(z)|<\g\big\}$ in the complex plane.  Let $\b>\g$ and $H^\8(\Si_\b)$ denote the space of bounded analytic functions in $\Si_\b$ and $H^\8_0(\Si_\b)$ its subspace consisting of all $\f$ satisfying
   $$|\f(z)|\le \frac{c|z|^\d}{1+|z|^{2\d}}\;\text{ for some }\; c>0\text{ and } \d>0.$$
Let $\Ga$ be the boundary of  $\Si_\b$, oriented in the positive sense. Then for any $\f\in H^\8_0(\Si_\b)$ the integral
 $$\f(B)=\frac1{2\pi{\rm i}}\int_\Ga \f(z)(z-B)^{-1}dz$$
 defines a bounded operator on $Y$, where the integral absolutely converges in $B(Y)$.

The following resolution of the identity will be useful later.  Let $\psi\in H^\8_0(\Si_\b)$ such that
  $$\int_0^\8\psi(t)\frac{dt}t=1.$$
 Then the following integral
  \beq\label{resolution}
  y=\int_0^\8\psi(tB)(y)\frac{dt}t=\lim_{\e\to0}\lim_{C\to\8}\int_\e^C\psi(tB)(y)\frac{dt}t
  \eeq
 exists for every $y\in \overline{{\rm im}B}$. 
This is \cite[Proposition~10.2.5]{HVVW}. Let us include its easy verification for completeness. For $\f\in H^\8_0(\Si_\b)$, we have
  $$\f(z)=\int_0^\8\psi(tz)\f(z)\frac{dt}t\,,\quad z\in \Si_\b.$$
  Thus for any $y\in Y$
  $$\f(B)(y)=\int_0^\8\psi(tB)\f(B)(y)\frac{dt}t\,.$$
 Choose
  $$
  \f(z)=\frac{n^2z}{(n+z)(1+nz)}\,.
  $$
Then $\f(B)(y) \to y$ in $Y$ as $n\to\8$ for any $y\in \overline{{\rm im}B}$ (see \cite[Theorem~3.8]{CDMY}), whence \eqref{resolution} by virtue of the convergence lemma (\cite[Lemma~2.1]{CDMY}).

\begin{definition}
 A family $\F\subset B(L_p(\O;X))$ is said to be $\el_q$-{\it bounded} if there exists a constant $c$ such that
  $$\Big\|\Big(\sum_k \|A_k(f_k)\|^q\Big)^{\frac1q}\Big\|_{L_p(\O)}\le c \Big\|\Big(\sum_k \|f_k\|^q\Big)^{\frac1q}\Big\|_{L_p(\O)}$$
for all finite sequences $\{A_k\}\subset\F$ and  $\{f_k\}\subset L_p(\O;X)$, with the usual modification for $q=\8$ in the above inequality.
\end{definition}

\begin{rk}\label{conv}
It is clear that the sums in the above definition can be replaced by integrals without changing the constant $c$. On the other hand, it is easy to show that the absolutely convex hull of an $\el_q$-bounded family is again $\el_q$-bounded with the same constant.
\end{rk}

Accordingly, we introduce the $\el_q$-boundedness versions of sectoriality of operators and analyticity of semigroups. Recall that a semigroup $\{S_t\}_{t>0}$ on a Banach space $Y$ is said to be analytic if it extends to an analytic function from $\Si_\b$ to $B(Y)$ for some $0<\b\le\frac\pi2$ and bounded in any smaller sector. In this case, we call $\{S_t\}_{t>0}$ an \emph{analytic semigroup of type} $\b$.

\begin{definition}
 \begin{enumerate}[\rm(i)]
 \item A densely defined closed operator $B$ on $L_p(\O;X)$ is called an $\el_q$-{\it sectorial operator of type} $\a$ with $0\le \a<\pi$ if $\com\setminus \overline{\Sigma_\g}$ is contained in the resolvent set of $B$ for any $\g>\a$ and the family
  $\big\{z(z-B)^{-1}: z\notin \overline{\Sigma_\g}\big\}$
 is $\el_q$-bounded on $L_p(\O;X)$.
 \item A semigroup $\{S_t\}_{t>0}$ on $L_p(\O;X)$ is called an $\el_q$-{\it analytic semigroup of type} $\b$ with $0< \b\le\frac\pi2$ if $\{S_t\}_{t>0}$  extends to an analytic function from $\Si_\b$ to $B(L_p(\O;X))$ and for any $\nu<\b$ the family
  $\big\{S_z: z\in\Sigma_\nu\big\}$
 is $\el_q$-bounded on $L_p(\O;X)$.
 \end{enumerate}
 \end{definition}

The following is the $\el_q$-boundedness analogue of a classical characterization of analytic semigroups.

\begin{prop}\label{lq-analyticity}
Let $\{S_t\}_{t>0}$ be a strongly continuous bounded semigroup on $L_p(\O;X)$ and $B$ its negative generator. Then the following statements are equivalent:
 \begin{enumerate}[\rm(i)]
  \item  $\{S_t\}_{t>0}$ is $\el_q$-analytic of type $\b$ for some $0<\b\le\frac\pi2$;
  \item $B$ is  $\el_q$-sectorial of type $\a$ for some $\a<\frac\pi2$;
  \item $\{S_t, tBS_t\}_{t>0}$ is $\el_q$-bounded on $L_p(\O;X)$.
  \end{enumerate}
 \end{prop}

 \begin{proof}
  The proof is a straightforward adaptation of the classical argument (cf. e.g. the proof of \cite[Theorem~5.2]{pazy}). As we want to track the links between the different constants involved, we give below an outline.

  (i) $\Rightarrow$ (ii). Let $\a=\frac\pi2-\b$. For $\a<\g<\frac\pi2$ choose $0<\nu<\b$ such that $\g+\nu>\frac\pi2$, for instance, we can take $\nu=\b-\frac{\g-\a}2$ so that $\g+\nu=\frac\pi2+\frac{\g-\a}2$.  Then for any $z=re^{{\rm i}\theta}\notin \overline{\Sigma_\g}$ and $z\neq0$, we have
 $$z(z-B)^{-1}=-z\int_0^\8e^{tz}S_tdt=-ze^{{\rm i}{\rm sgn}(\theta)\nu} \int_0^\8e^{tre^{{\rm i}(\theta+{\rm sgn}(\theta)\nu)}}S_{te^{{\rm i}{\rm sgn}(\theta)\nu}}dt.$$
 If $C_\nu$ denotes the $\el_q$-boundedness constant of the family $\big\{S_\zeta: \zeta\in\Si_\nu\big\}$, then by Remark~\ref{conv} we deduce that $\big\{z(z-B)^{-1}: z\notin\Si_\g,\, z\neq0\big\}$ is  $\el_q$-bounded with constant $C_\g$ given by
 $$C_\g\le C_\nu\sup_{z\notin\Si_\g}|z|\int_0^\8\big|e^{tre^{{\rm i}(\theta+{\rm sgn}(\theta)\nu)}}\big|dt\le\frac{C_\nu}{|\cos(\theta+{\rm sgn}(\theta)\nu)|}\le \frac{C_\nu}{|\cos(\g+\nu)|}\,.$$

(ii) $\Rightarrow$ (iii). Let $\a<\g<\frac\pi2$ and  $\Ga$ be the boundary of $\Si_{\g}$ oriented in the positive sense. Then
  $$tBS_t=\frac1{2\pi{\rm i}}\int_\Ga t\l e^{-t\l}(\l-B)^{-1}d\l, \quad t>0.$$
 Thus $\{tBS_t\}_{t>0}$ is $\el_q$-bounded with constant
  $$C'_d\le \frac{C_\g}{\pi}\sup_{t>0}\int_0^\8 t e^{-tr\cos\g}dr=\frac{C_\g}{\pi}\int_0^\8 e^{-r\cos\g}dr= \frac{C_\g}{\pi\cos\g}.$$
 To show the  $\el_q$-boundedness of $\{S_t\}_{t>0}$ we need to slightly modify the contour $\Ga$. Let $\Ga'$ be the union of the part of $\Ga$ with $|\l|\ge1$ and the arc in $\com\setminus\Si_{\g}$ of the circle with the origin as center and radius $1$. Then we  have
  $$S_t=\frac1{2\pi{\rm i}}\int_{\Ga'}  e^{-t\l}(\l-B)^{-1}d\l,\quad t>0.$$
The change of variables $\zeta=t\l$ yields
  $$S_t=\frac1{2\pi{\rm i}}\int_{t\Ga'}  e^{-\zeta}\,\frac1t\,(\frac{\zeta}t-B)^{-1}d\zeta,$$
 where $t\Ga'$ is the union of the part of $\Ga$ with $|\l|\ge t$ and the arc in $\com\setminus\Si_{\g}$ of the circle with the origin as center  and radius $t$. However, the Cauchy formula insures that we can go back to $\Ga'$:
   $$S_t=\frac1{2\pi{\rm i}}\int_{\Ga'}  e^{-\zeta}\,\frac1t\,(\frac{\zeta}t-B)^{-1}d\zeta.$$
This implies that  $\{S_t\}_{t>0}$ is $\el_q$-bounded with constant
  $$C''_d\le \frac{C_\g}{\pi} \int_1^\8 e^{-r\cos\g}\frac{dr}r+  \frac{C_\g}{2\pi} \int_\g^{2\pi-\g}e^{-\cos\theta}d\theta\le\frac{CC_\g}{\cos\g}\,.$$

  (iii) $\Rightarrow$ (i). Let $C_d$ denote the $\el_q$-boundedness constant of $\{S_t, tBS_t\}_{t>0}$. The function $t\mapsto S_t$ is infinitely derivable in $\real_+$ and for any positive integer $n$
  $$(S_t)^{(n)}=\big(S'_{\frac{t}n}\big)^n\,.$$
 This shows that $\{t^n(S_t)^{(n)}\}_{t>0}$ is $\el_q$-bounded with constant $(C_d)^n$. Let $\b=\arctan\frac1{eC_d}$. Then $\{S_t\}_{t>0}$ becomes an $\el_q$-analytic semigroup of type $\b$ thanks to the following formula
  $$S_z=\sum_{n=0}^\8\frac{(S_t)^{(n)}}{n!}\, (z-t)^n\,,\quad z\in\Si_\b$$
 and for any $\nu<\b$ the family $\big\{S_z: z\in\overline{\Sigma_\nu}\big\}$ is $\el_q$-bounded with constant
  $$C_\nu\le\frac1{1-eC_d\tan\nu}\,.$$
 The proof is thus complete.
  \end{proof}

 The following is again the $\el_q$-boundedness version of an elementary result on sectorial operators. The case used later  concerns only $\sqrt B$.

 \begin{prop}\label{sqrtA}
  Let $B$ be an $\el_q$-sectorial operator of type $\a$ on $L_p(\O;X)$ with $\a<\pi$. Let $\theta>0$ such that $\theta\a<\pi$. Then $B^\theta$ is an  $\el_q$-sectorial operator of type $\theta\a$ on $L_p(\O;X)$.
 \end{prop}

 \begin{proof}
 Let $\g>\g'>\theta\a$. Given $z\notin\overline{\Si_\g}$, writing 
  $$z=[z^{\theta^{-1}}(z-\l^\theta)+ (z^{\theta^{-1}}\l^\theta-z\l)](z^{\theta^{-1}}-\l)^{-1}\,,$$
 we have
  $$z(z-\l^\theta\,)^{-1}= z^{\theta^{-1}}(z^{\theta^{-1}}-\l)^{-1}+\f(\l),$$
  where $\f(\l)=(z^{\theta^{-1}}\l^\theta -z\l)(z-\l^{\theta})^{-1}(z^{\theta^{-1}}-\l)^{-1}$. Thus
  $$
  z(z-B^\theta\,)^{-1}
  = z^{\theta^{-1}}(z^{\theta^{-1}}-B)^{-1}+\f(B).$$
Note that $\f$ is analytic in $\Si_{\g\theta^{-1}}$. Let $\Ga$ be the boundary of $\Si_{\g'\theta^{-1}}$. Then
  $$\f(B)=\frac1{2\pi{\rm i}}\int_\Ga\f(\l)(\l-B)^{-1}d\l=\frac1{2\pi{\rm i}}\int_\Ga\f(\l)[\l(\l-B)^{-1}]\,\frac{d\l}\l.$$
 The change of variables $\zeta=z^{-\theta^{-1}}\l$ yields
  $$\int_\Ga|\f(\l)|\,\frac{|d\l|}{|\l|}
  =\int_{z^{-\theta^{-1}}\Ga}\frac{|\zeta^\theta-\zeta|}{|1-\zeta^\theta|\,|1-\zeta|}\,\frac{|d\zeta|}{|\zeta|}\,.$$
 Decomposing the last integral into three parts corresponding to $|\zeta|$ close to $0$, $1$ and $\8$, respectively, we get
 $$\int_\Ga|\f(\l)|\,\frac{|d\l|}{|\l|}\les\frac1{\theta}+\frac1{\theta(\g-\g')^2}\,.  $$
Hence by Remark~\ref{conv}, we deduce the desired assertion.
  \end{proof}

\medskip

Now we return back to our distinguished semigroup $\{T_t\}_{t>0}$ of regular operators on $L_p(\O)$. We have extended $\{T_t\}_{t>0}$ to $L_p(\O;X)$. Recall our convention that the regular operators considered in this article are all assumed contractively regular. Also recall the fact that $T$ is regular on $L_p(\O)$ iff
 $$\big\|\sum_k|T(f_k)|\big\|_p\le \big\|\sum _k|f_k|\big\|_p$$
for all finite sequences $\{f_k\}$ in $L_p(\O)$ (see \cite{Marc}). Consequently, $T$ is regular on $L_p(\O)$ iff  its adjoint $T^*$ is regular on $L_{p'}(\O)$. In particular, $\{T^*_t\}_{t>0}$ is a strongly continuous semigroup of regular operators on $L_{p'}(\O)$.

\begin{lem}\label{lq-bdd ergodic}
  Let
  $$M_t=\frac1t\int_0^tT_s ds,\quad t>0.$$
The family  $\{M_t\}_{t>0}$ is $\el_q$-bounded on $L_p(\O;X)$ with constant $\max((p')^{1-\frac{p}{q}},\, p^{1-\frac{p'}{q'}})$.
\end{lem}

 \begin{proof}
 The celebrated theorem of Akcoglu \cite{Ak} asserts  that $\{M_t\}_{t>0}$ satisfies the following maximal ergodic inequality
  $$\big\|\sup_{t>0}|M_t(f)|\,\big\|_{L_p(\O)}\le p' \big\|f\big\|_{L_p(\O)}\,,\quad f\in L_p(\O);$$
 see also \cite[Theorem~5.2.5]{Krengel}.
The regularity of $\{M_t\}_{t>0}$ insures that this inequality remains valid for any $f\in L_p(\O;X)$. Thus for any finite sequences $\{t_k\}\subset\real_+$ and  $\{f_k\}\subset L_p(\O;X)$, we have
   \begin{align*}
   \big\|\sup_{k}\|M_{t_k}(f_k)\|_X\big\|_{L_p(\O)}
   &\le \big\|\sup_{k}|M_{t_k}|(\|f_k\|_X)\big\|_{L_p(\O)}\\
  &\le  \big\|\sup_{k}|M_{t_k}|(\sup_j\|f_j\|_X)\big\|_{L_p(\O)} \\
  &\le  p'\big\|\sup_j\|f_j\|_X\,\big\|_{L_p(\O)}\,.
  \end{align*}
This means that $\{M_t\}_{t>0}$ is $\el_\8$-bounded on $L_p(\O;X)$ with constant $p'$.  On the other hand, $\{M_t\}_{t>0}$ is bounded, so $\el_p$-bounded on $L_p(\O;X)$ with constant  $1$. Thus by complex interpolation, $\{M_t\}_{t>0}$ is $\el_q$-bounded on $L_p(\O;X)$ with constant $(p')^{1-\frac{p}{q}}$ for $q>p$.

The case of $q<p$ is treated by duality. Applying the previous discussion to the adjoint semigroup $\{T^*_t\}_{t>0}$,  we deduce that $\{M^*_t\}_{t>0}$ is $\el_\8$-bounded on $L_{p'}(\O;X^*)$ with constant $p$, so $\{M_t\}_{t>0}$ is $\el_1$-bounded on $L_{p}(\O;X)$ with constant $p$. The assertion for $q<p$ then follows by complex interpolation once more.
  \end{proof}

\begin{rk}\label{lq-bdd poisson}
The above lemma and the subordination formula  \eqref{subordination} imply that  the Poisson subordinated semigroup $\{P_t\}_{t>0}$ is $\el_q$-bounded on $L_p(\O;X)$ with constant $C\max((p')^{1-\frac{p}{q}},\, p^{1-\frac{p'}{q'}})$, where $C$ is an absolute constant coming from \eqref{subordination}.
\end{rk}

It is a classical result that the negative generator $A$ of $\{T_t\}_{t>0}$ on $L_p(\O;X)$ is a sectorial operator of type $\frac\pi2$. The following shows that it is moreover $\el_q$-sectorial.

\begin{prop}\label{A}
  The negative generator $A$ of $\{T_t\}_{t>0}$ is an $\el_q$-sectorial operator of type $\frac\pi2$ on $L_p(\O;X)$. More precisely, the family
  $$\big\{z(z-A)^{-1}: z\notin \overline{\Sigma_\a}\,\big\}$$
 is $\el_q$-bounded on $L_p(\O;X)$ with constant $C_\a\max((p')^{1-\frac{p}{q}},\, p^{1-\frac{p'}{q'}})$ for any $\frac\pi2<\a<\pi$.

 Consequently, $\sqrt A$ is an $\el_q$-sectorial operator of type $\frac\pi4$ on $L_p(\O;X)$. Moreover,  the $\el_q$-boundedness constant of
  $\{z(z-\sqrt A\,)^{-1}: z\notin \overline{\Sigma_\a}\}$
 is majorized by $C_\a\max((p')^{1-\frac{p}{q}},\, p^{1-\frac{p'}{q'}})$ for any $\frac\pi4<\a<\pi$.
 \end{prop}

 \begin{proof}
    Let $z\in\com$ with ${\rm Re}z<0$. Then
    $$(z-A)^{-1}=\int_0^\8e^{tz}T_tdt=- z\int_0^\8te^{tz}M_tdt.$$
Thus by Lemma~\ref{lq-bdd ergodic} and Remark~\ref{conv}, we deduce that
$$\Big\{\frac{({\rm Re}z)^2}{|z|}(z-A)^{-1}: {\rm Re}z<0\Big\}$$
 is $\el_q$-bounded on $L_p(\O;X)$ with constant $\max((p')^{1-\frac{p}{q}},\, p^{1-\frac{p'}{q'}})$. This implies the assertion on $A$ with  the constant  $C_\a$ given by 
  $$C_\a=\sup_{z\notin \overline{\Sigma_\a}} \frac{({\rm Re}z)^2}{|z|^2}.$$ The second part on $\sqrt A$ then follows from Proposition~\ref{sqrtA}.
 \end{proof}

\begin{prop}\label{T-analyticity}
  Assume that  $\{T_t\}_{t>0}$ satisfies \eqref{Ana bound} for some $0<\b_0\le\frac\pi2$. Let $1<q<\8$ and $\b_q=\b_0\min(\frac{p}{q}\,,\frac{p'}{q'})$. Then $\{T_t\}_{t>0}$ is an $\el_q$-analytic semigroup of type $\b_q$ on $L_p(\O;X)$. More precisely, for any $0<\b<\b_q$ the family
    $\big\{T_z: z\in\Sigma_\b\big\}$
 is $\el_q$-bounded on $L_p(\O;X)$ with constant majorized by
  $$C(\b_q-\b)^{-1}\,\mathsf{T}_{\b_0}^{\min(\frac{p}{q},\,\frac{p'}{q'})}\max((p')^{1-\frac{p}{q}},\, p^{1-\frac{p'}{q'}}).$$
 Consequently, $A$ is $\el_q$-sectorial of type $\a_q=\frac\pi2-\b_q$ on $L_p(\O;X)$. More precisely, for any $\a_q<\a<\frac\pi2$ the family $\big\{z(z-A)^{-1}: z\notin \overline{\Sigma_\a}\big\}$ is $\el_q$-bounded on $L_p(\O;X)$  with the relevant constant majorized by
  $$C(\b_q-\b)^{-2}\,\mathsf{T}_{\b_0}^{\min(\frac{p}{q},\,\frac{p'}{q'})}\max((p')^{1-\frac{p}{q}},\, p^{1-\frac{p'}{q'}})\;\text {with }\; \b=\frac\pi2-\a.$$
  \end{prop}

\begin{proof}
  Define
  $$M_z=\frac1z\int_0^zT_\l d\l,\quad z\in\Si_{\b_0},$$
 where the integral is taken along the segment $[0,\,z]$. Clearly, $M$ is analytic in $\Si_{\b_0}$. By  Lemma~\ref{lq-bdd ergodic}, $\{M_t: t>0\}$ is $\el_\8$-bounded (resp. $\el_1$-bounded) on $L_p(\O;X)$ with constant $p'$ (resp. $p$). On the other hand, \eqref{Ana bound} means that $\{M_z: z\in\Si_{\b_0}\}$ is $\el_p$-bounded on $L_p(\O;X)$ with constant $\mathsf{T}_{\b_0}$. Then by complex interpolation, we see that $\{M_z: z\in\Si_{\b_q}\}$ is $\el_q$-bounded on $L_p(\O;X)$ with constant $\mathsf{T}_{\b_0}^{\frac{p}{q}}\,(p')^{1-\frac{p}{q}}$ for $p<q$ and $\mathsf{T}_{\b_0}^{\frac{p'}{q'}}\,p^{1-\frac{p'}{q'}}$ for $p>q$.

We use the identity $T_z=M_z+zM_z'$ to pass from $M_z$ to $T_z$, so it remains to show that $\{zM_z': z\in\Si_{\b}\}$ is $\el_q$-bounded. To this end, let $\d=\frac12(\b+\b_q)$. For any $z=re^{{\rm i}\theta}\in \Si_{\b}$ let $\C$ be the circle with center $z$ and radius $r\sin(\d-|\theta|)$. Note that one of the two rays limiting $\Si_{\d}$ is a tangent of $\C$. By the Cauchy integral formula, we have
 $$zM_z'=\frac{z}{2\pi{\rm i}}\int_\C\frac{M_\l\, d\l}{(\l-z)^2}\,.$$
Since
 $$\frac{|z|}{2\pi}\int_\C\frac{|d\l|}{|\l-z|^2}=\frac{1}{\sin(\d-|\theta|)}\le \frac{1}{\sin\frac12(\b_q-\b)}\,,$$
The $\el_q$-boundedness of $\{M_z: z\in\Si_{\d}\}$ and Remark~\ref{conv} imply that  $\{zM_z': z\in\Si_{\b}\}$ is $\el_q$-bounded on $L_p(\O;X)$ with constant majorized by
  $$C(\b_q-\b)^{-1}\,\mathsf{T}_{\b_0}^{\min(\frac{p}{q},\,\frac{p'}{q'})}\max((p')^{1-\frac{p}{q}},\, p^{1-\frac{p'}{q'}}).$$
The last part on the $\el_q$-sectoriality of $A$  follows from the proof of the implication (i)$\,\Rightarrow\,$(ii) of Proposition~\ref{lq-analyticity}, $\b$ and $\nu$ there being respectively $\b_q$ and $\b$ now.
  \end{proof}


\section{Proof of Theorem~\ref{proj}}\label{Proof of Theorem proj}


Armed with the preparation in section~\ref{boundedness}, we will follow the proof of \cite[Theorem~4.14]{JLMX}. In the sequel, we will use the abbreviation that $\partial=\frac{\partial}{\partial t}$. Recall that $\real_+$ is equipped with the measure $\frac{dt}t$. Also recall our convention that $\{T_t\}_{t>0}$ is a strongly continuous semigroup of regular operators on $L_p(\O)$ and $\{P_t\}_{t>0}$ its subordinated Poisson semigroup.

We first show part (i) concerning the subordinated Poisson semigroup $\{P_t\}_{t>0}$.  Fix $\frac\pi4<\a<\b<\frac\pi2$.
 Let $\Ga$ be the boundary of $\Si_\a$. Define $F(z)=-ze^{-z}$. Then $F\in H^\8_0(\Si_\b)$. For any $t>0$ we have
  \beq\label{proj-fc}
  t\partial P_t=F(t\sqrt A\,)=\frac1{2\pi\textrm{i}}\int_\Ga F(tz)R(z)dz,
 \eeq
 where $R(z)=(z-\sqrt A\,)^{-1}$. Recall that the map $\mathcal{P}$ is defined by \eqref{def projection} (with $\{P_t\}_{t>0}$ instead of $\{T_t\}_{t>0}$ there). It can be rewritten as
  \beq\label{proj-int}
  \mathcal{P}(h)_s
 =\frac2{\pi\textrm{i}}\int_\Ga\int_0^\8 F(sz)F(tz)zR(z)(h_t)\,\frac{dt}t \,\frac{dz}z\,,\quad s>0.
  \eeq
Let $\Ga$ be equipped with the measure $\frac{|dz|}{|z|}$. We define three maps as follows:
\begin{itemize}
\item[$\bullet$] $\Phi_1: L_p(\O;L_q(\real_+; X))\to L_p(\O;L_q(\Ga; X))$ by
 $$\Phi_1(h)_z=\int_0^\8 F(tz) h_t\,\frac{dt}t\,, \quad z\in\Ga,\; h\in  L_p(\O;L_q(\real_+; X)),$$
 \item[$\bullet$] $\Phi_2: L_p(\O;L_q(\Ga; X))\to L_p(\O;L_q(\real_+; X))$ by
 $$\Phi_2(g)_s=\int_\Ga F(sz) g_z\,\frac{dz}z\,,  \quad s>0, \;g\in  L_p(\O;L_q(\Ga; X))$$
\item[$\bullet$] $\Phi: L_p(\O;L_q(\Ga; X))\to L_p(\O;L_q(\Ga; X))$ by
 $$\Phi(g)_z= \frac2{\pi\textrm{i}}\,zR(z)(g_z),  \quad z\in\Ga, \;g\in  L_p(\O;L_q(\Ga; X)).$$
 \end{itemize}
Then $\mathcal{P}=\Phi_2\Phi\Phi_1$. Thus it remains to show that the three newly defined maps are all bounded. Consider first  the case $q<\8$. By the H\"older inequality, we have
 $$\big\|\Phi_1(h)_z\big\|_X^q\le\Big(\int_0^\8 |F(tz)|\,\frac{dt}{t}\Big)^{q-1} \int_0^\8|F(tz)|\,\big\|h_t\big\|^q_X\,\frac{dt}{t}\,.$$
Note that for $z=re^{\pm{\rm i}\a}\in\Ga$
 $$\int_0^\8 |F(tz)|\,\frac{dt}{t} =\int_0^\8 |F(te^{\pm{\rm i}\a})|\,\frac{dt}{t}=\frac1{\cos\a}\,.$$
On the other hand, for any $t>0$
 $$\int_\Ga |F(tz)|\frac{|dz|}{|z|}=\frac2{\cos\a}\,.$$
We then deduce that
  $$\big\|\Phi_1(h)\big\|_{L_p(\O;L_q(\Ga; X))}\le \frac{2^{\frac1q}}{\cos\a} \,\big\|h\big\|_{L_p(\O;L_q(\real_+; X))}\,.$$
 Thus
  $$\big\|\Phi_1\big\|\le  \frac{2^{\frac1q}}{\cos\a}\,.$$
 The same upper estimate holds for $\big\|\Phi_2\big\|$ too.
 Finally, the boundedness of $\Phi$ is just a reformulation of the $\el_q$-boundedness of $\big\{\frac2{\pi\textrm{i}}\,zR(z): z\in\Ga\setminus\{0\}\big\}$. Thus by Proposition~\ref{A},
  $$\big\|\mathcal{P}\big\|\les\frac1{\cos^2\a}\max((p')^{1-\frac{p}{q}},\, p^{1-\frac{p'}{q'}}).$$
  This finishes the proof of the first assertion for $q<\8$ (choosing $\a$ close to $\frac\pi4$). The boundedness of  $\mathcal{P}$ for $q=\8$ is obtained from that for $q=1$ by duality.

  Let us show that $\mathcal{P}$ is a projection. Let $h\in L_p(\O;L_q(\real_+; X))$ be given by $h_t=t \partial P_t(f)$ for some $f\in L_p(\O; X)$. Then by \eqref{proj-int}
    \begin{align*}
     \mathcal{P}(h)_s
 &=\frac2{\pi\textrm{i}}\int_\Ga\int_0^\8 F(sz)F(tz)F(tz)R(z)(f)\,\frac{dt}t \,dz\\
  &=\frac2{\pi\textrm{i}}\int_\Ga F(sz)R(z)(f)dz \int_0^\8(F(t))^2\frac{dt}t\\
   &=\frac1{2\pi\textrm{i}}\int_\Ga F(sz)R(z)(f)dz\\
   &=F(s\sqrt A\,)(f)=s\partial P_s(f).
   \end{align*}
 Thus $\mathcal{P}(h)=h$, so $\mathcal{P}$ is a projection. This shows the first assertion (i).

\medskip

Assertion (ii) on the semigroup $\{T_t\}_{t>0}$ itself  is proved exactly in the same way. Indeed, letting $\a_q=\frac\pi2-\b_q$, by Proposition~\ref{T-analyticity}, $A$ is $\el_q$-sectorial of type $\a_q$. Let $\b=\frac{\b_q}2$ and $\a=\frac\pi2-\b$. Then $\a_q<\a<\frac\pi2$ and
 $$\frac1{\cos^2\a}\approx \frac1{\b_q^2},\quad   \frac1{(\b_q-\b)^2}\approx  \frac1{\b_q^2}.$$
 Thus using the estimate on the $\el_q$-sectoriality constant of $A$ and repeating the above argument, we show that  $\mathcal{T}$ is bounded with the announced norm estimate.


\section{Luzin type and cotype}\label{Luzin type and cotype}


In this section we study Banach spaces that are of Luzin coype or type. Before proceeding we briefly discuss the projection $\mathsf F$ onto the  fixed point subspace of
$\{T_t\}_{t> 0}$ (equivalently, of $\{P_t\}_{t> 0}$). By the mean ergodic theorem, $\mathsf F$ is given by
$$\mathsf F(f)=\lim_{t\to\8}\frac1t\int_0^tT_s(f)ds,\quad f\in L_p(\O).$$
Thus $\mathsf F$ is also regular, so extends to a contractive projection on $L_p(\O;X)$. Then the above formula remains valid for $f\in L_p(\O;X)$  and $\mathsf F(L_p(\O;X))$ coincides with the  fixed point subspace of $\{T_t\}_{t> 0}$ on $L_p(\O;X)$. It follows that $L_p(\O;X)$ admits the following direct sum decomposition:
  \beq\label{decomposition}
 L_p(\O;X)= \mathsf F(L_p(\O;X))\oplus \ker\mathsf F.
  \eeq
 On the other hand, $\ker\mathsf F$ is the closure of $\big\{({\rm Id}-T_t)(L_p(\O;X)): t>0\big\}$. Moreover,
 \beq\label{kernel decomposition}
 \mathsf F(L_p(\O;X))=\ker A=\ker\sqrt A\;\text{ and }\;  \ker\mathsf F=\overline{{\rm im}A}=\overline{{\rm im}\sqrt A}.
 \eeq

By the paragraph before Lemma~\ref{lq-bdd ergodic}, we know that the adjoint semigroup $\{T^*_t\}_{t> 0}$ is regular on $L_{p'}(\O)$ too. Thus the above discussion also applies to the semigroup $\{T^*_t\}_{t> 0}$ that is extended to $L_{p'}(\O; X^*)$ again. Consequently,  \eqref{decomposition} and \eqref{kernel decomposition} transfer to this dual setting. We should draw the reader's attention to the fact that $L_{p'}(\O; X^*)$ is in general not the dual of $L_p(\O;X)$ but an isometric subspace. With this in mind, we have
 $${\rm Id}_{X^*}\otimes T_t^*=\big({\rm Id}_{X}\otimes T_t)^*\Big|_{L_{p'}(\O; X^*)}\,.$$
A similar formula holds for the negative generator $A$ of $\{T_t\}_{t> 0}$ on $L_p(\O;X)$ and the negative generator $A^*$ of $\{T^*_t\}_{t> 0}$ on $L_{p'}(\O; X^*)$, that is, the restriction to $L_{p'}(\O; X^*)$ of the adjoint of the former coincides with the latter. Moreover, $\mathsf F^*\big|_{L_{p'}(\O; X^*)}$ is the fixed point projection associated to  $\{T^*_t\}_{t> 0}$ on $L_{p'}(\O; X^*)$. All this allows us to use duality arguments without any problem as when $L_p(\O;X)^*=L_{p'}(\O; X^*)$ (which is the case for reflexive $X$).

\medskip

According to \cite{LP0}, we introduce the following definition already mentioned before Theorem~\ref{Poisson ML}.

\begin{definition}
 Let $1\le q\le \8$.
 \begin{enumerate}[(i)]
 \item Define
  \beq\label{LPS-funct}
  \mathcal{G}^T_q(f)=\Big(\int_0^\8\|t\partial T_t(f)\|^q_X\frac{dt}t\Big)^{\frac1q}\,,\quad f\in L_p(\O;X).
  \eeq
 \item $X$ is said to be of {\it Luzin cotype $q$ relative to} $\{T_t\}_{t> 0}$ if there exists a constant $c$ such that
 \beq\label{LPS-cotype}
 \big\|\mathcal{G}^T_q(f)\big\|_{L_p(\O)}\le c\big\|f\big\|_{L_p(\O;X)}\,,\quad f\in L_p(\O;X).
 \eeq
The smallest $c$ is denoted by $\mathsf{L}^T_{\cc, q, p}(X)$.
\item $X$ is said to be of {\it Luzin type $q$ relative to} $\{T_t\}_{t> 0}$ if there exists a constant $c$ such that
  \beq\label{LPS-type}
\big\|f-\mathsf{F}(f)\big\|_{L_{p}(\O;X)} \le c\big\|\mathcal{G}^{T}_{q}(f)\big\|_{L_{p}(\O)} \,,\quad f\in L_{p}(\O;X).
 \eeq
 The smallest  $c$ is denoted by $\mathsf{L}^T_{\tt, q, p}(X)$.
  \end{enumerate}
  \end{definition}

In \eqref{LPS-funct}, $f$ is implicitly assumed to belong to the definition domain of $A$ in order to guarantee the derivability of $T_t(f)$ in $t$.
Note that  if $\{T_t\}_{t> 0}$ is analytic on $L_p(\O;X)$, $\mathcal{G}^T_q(f)$ is defined for any $f\in L_p(\O;X)$. When it is defined, $\mathcal{G}^T_q(f)$ is a positive measurable function on $\O$ but may not belong to $L_{p}(\O)$ in which case $\|\mathcal{G}^{T}_{q}(f)\|_{L_{p}(\O)}$ is interpreted as $\8$ (then \eqref{LPS-type} is trivially verified for such $f$). On the other hand, the above definition implicitly depends on $p$,  but this dependence is not essential thanks to the fact that in most cases, if  \eqref{LPS-cotype} or  \eqref{LPS-type} holds for one $p$, so it does for any allowed $p$. Thus to lighten the terminology, we have decided to not explicitly mention $p$ in the above notions; anyway, this dependence on $p$ is reflected in the constants $\mathsf{L}^T_{\cc, q, p}(X)$ and $\mathsf{L}^T_{\tt, q,p}(X)$.

 \begin{rk}
 Without additional assumption on $\{T_t\}_{t> 0}$, the definition may be insignificant. For instance, if  $\{T_t\}_{t> 0}$ is the translation group of $\real$, it is easy to check that
 $$\Big\|\Big(\int_0^\8|t\partial T_t(f)|^q\frac{dt}t\Big)^{\frac1q}\Big\|^p_{L_p(\real)}
 =\int_{\real}\Big(\int_0^\8|t f'(s+t)|^q\frac{dt}t\Big)^{\frac{p}q} ds=\8$$
for any $1\le q\le\8$ and for any $f\in L_p(\real)$ with $f'$ not identically zero. Thus $\com$ is not of Luzin cotype $q$ for any $q$ relative to the translation group of $\real$.
 \end{rk}

 This remark shows that to have a meaningful theory of Luzin type and cotype some minimal condition should be imposed to $\{T_t\}_{t> 0}$. It turns out that this minimal condition is the analyticity of $\{T_t\}_{t> 0}$ on $L_{p}(\O;X)$. As shown in section~\ref{boundedness},  the subordinated Poisson semigroup $\{P_t\}_{t> 0}$ always satisfies this condition.

 \medskip

 It is sometimes convenient to have a discrete version of  $\mathcal{G}^T_q(f)$. Recall that if $\{T_t\}_{t> 0}$ is analytic on $L_p(\O)$, we have the following maximal inequality
   \beq\label{Tmax}
\big\|\sup_{t>0}|T_t(f)|\big\|_{L_{p}(\O)}\le \mathsf{T}_{\max}\|f\|_{L_{p}(\O)}\,,\quad f\in L_p(\O)
 \eeq
for some constant $\mathsf{T}_{\max}$ (see \cite{LMX}).  A similar inequality holds for the adjoint semigroup  $\{T^*_t\}_{t> 0}$, the relevant constant being denoted by  $\mathsf{T}^*_{\max}$.

\begin{prop}
Assume that $\{T_t\}_{t> 0}$ is analytic on $L_p(\O)$. Let $1\le q\le\8$ and  $a>1$. Then  for any $f\in L_{p}(\O;X)$
  $$
  c_{T, q,a}^{-1}\big\|\mathcal{G}^T_q(f)\big\|_{L_p(\O)}
 \le\Big\|\Big(\sum_{k\in\ent}\big\|a^k\partial T_{a^k}(f)\big\|_X^q\Big)^{\frac1q}\Big\|_{L_p(\O)} \le  C_{T, q,a}\big\|\mathcal{G}^T_q(f)\big\|_{L_p(\O)}\,,
  $$
 where
    \begin{align*}
    &c_{T, q,a}= q^{-\frac1q}(a^q-1)^{\frac1q}\max(\mathsf{T}_{\max}^{1-\frac{p}{q}},\, {\mathsf{T}^*}_{\max}^{1-\frac{p'}{q'}}),\\
   &C_{T, q,a}=q^{\frac1q}(1-a^{-q})^{-\frac1q}\max(\mathsf{T}_{\max}^{1-\frac{p}{q}},\, {\mathsf{T}^*}_{\max}^{1-\frac{p'}{q'}})\,.
    \end{align*}
 Similar inequalities hold for $\{P_t\}_{t> 0}$ in place of $\{T_t\}_{t> 0}$ without any additional assumption on $\{T_t\}_{t> 0}$, the corresponding constants being given by
    \begin{align*}
    &c_{P, q,a}=C^{-1} q^{-\frac1q}(a^q-1)^{\frac1q}\max((p')^{1-\frac{p}{q}},\, p^{1-\frac{p'}{q'}}),\\
    &C_{P, q,a}=Cq^{\frac1q}(1-a^{-q})^{-\frac1q}\max((p')^{1-\frac{p}{q}},\, p^{1-\frac{p'}{q'}})\,.
      \end{align*}
 \end{prop}

\begin{proof}
 Using \eqref{Tmax}, its adjoint version and repeating the proof of Lemma~\ref{lq-bdd ergodic}, we show that  $\{T_t\}_{t> 0}$ is $\el_q$-bounded on $L_p(\O;X)$ with constant $\max(\mathsf{T}_{\max}^{1-\frac{p}{q}},\, {\mathsf{T}^*}_{\max}^{1-\frac{p'}{q'}})$.
Write
  \begin{align*}
 \mathcal{G}^T_q(f)^q
 =\sum_{k\in\ent}\int_{a^k}^{a^{k+1}}\big\|t\partial T_t(f)\big\|_X^q\frac{dt}t
=\sum_{k\in\ent}\int_{1}^{a}\big\|a^kt\partial T_{a^kt}(f)\big\|_X^q\frac{dt}t\,.
 \end{align*}
 Using $\partial T_{t+s}=T_s\partial T_{t}$, we have $\partial T_{a^kt}(f)=T_{a^k(t-1)}\partial T_{a^k}(f)$. Then the $\el_q$-boundedness of  $\{T_t\}_{t> 0}$ on $L_p(\O;X)$ yields
  \begin{align*}
 \big\|\mathcal{G}^T_q(f)\big\|_{L_p(\O)}
  &\le \max(\mathsf{T}_{\max}^{1-\frac{p}{q}},\, {\mathsf{T}^*}_{\max}^{1-\frac{p'}{q'}})
  \Big\|\Big(\sum_{k\in\ent}\int_{1}^{a}\big\|a^kt\partial T_{a^k}(f)\big\|_X^q\frac{dt}t\Big)^{\frac1q}\Big\|_{L_p(\O)}\\
  &= q^{-\frac1q}(a^q-1)^{\frac1q}\max(\mathsf{T}_{\max}^{1-\frac{p}{q}},\, {\mathsf{T}^*}_{\max}^{1-\frac{p'}{q'}})
   \Big\|\Big(\sum_{k\in\ent}\big\|a^k\partial T_{a^k}(f)\big\|_X^q\Big)^{\frac1q}\Big\|_{L_p(\O)}\,.
   \end{align*}
  For the converse inequality, we write
   $$\big\|a^k\partial T_{a^k}(f)\big\|_X^q=q(1-a^{-q})^{-1}\int_{a^{-1}}^1\big\|a^ktT_{a^k(1-t)}\partial T_{a^kt}(f)\big\|_X^q\frac{dt}t\,.$$
  As above, we then deduce
   $$\Big\|\Big(\sum_{k\in\ent}\big\|a^k\partial T_{a^k}(f)\big\|_X^q\Big)^{\frac1q}\Big\|_{L_p(\O)}
   \le q^{\frac1q}(1-a^{-q})^{-\frac1q}\max(\mathsf{T}_{\max}^{1-\frac{p}{q}},\, {\mathsf{T}^*}_{\max}^{1-\frac{p'}{q'}})\big\|\mathcal{G}^T_q(f)\big\|_{L_p(\O)}\,.$$
 The second part on $\{P_t\}_{t> 0}$ is just a particular case with $\mathsf{P}_{\max}=Cp'$ and $\mathsf{P}^*_{\max}=Cp$  by virtue of  Remark~\ref{lq-bdd poisson}.
 \end{proof}

 Recall the classical fact that $\{T_t\}_{t> 0}$ is analytic on $L_p(\O; X)$ iff $\{t\partial T_t:t> 0\}$ is uniformly bounded on $L_p(\O; X)$. Thus the following remark immediately follows from the above result; it shows in particular that the analyticity of $\{T_t\}_{t> 0}$ on $L_p(\O; X)$ is necessary  for $X$ to be of Luzin cotype $q$ relative to $\{T_t\}_{t> 0}$ for some $q$.

 \begin{rk}\label{necessity of analyticity}
  Assume that $\{T_t\}_{t> 0}$ is analytic on $L_p(\O)$.
  If $X$ is of Luzin cotype $($resp. type$)$ $q$  relative to $\{T_t\}_{t> 0}$, then $X$ is of Luzin cotype $($resp. type$)$ $r$ relative to $\{T_t\}_{t> 0}$ for any $r>q$ $($resp. $r<q)$. Moreover, if  $X$ is of Luzin cotype $\8$  relative to $\{T_t\}_{t> 0}$, then $\{T_t\}_{t> 0}$ must be analytic on $L_p(\O; X)$.
 \end{rk}

 The following is one of the main results of this section.

\begin{thm}\label{dual}
 Let $X$ be a Banach space and $1\le q\le\8$.
 \begin{enumerate}[\rm(i)]
 \item $X$ is of Luzin cotype $q$ relative to $\{P_t\}_{t> 0}$ iff  $X^*$ is of Luzin type $q'$ relative to $\{P^*_t\}_{t> 0}$.\\
  Moreover, the relevant constants satisfy
  $$\mathsf{L}^{P^*}_{\tt, q', p'}(X^*)\les \mathsf{L}^P_{\cc, q, p}(X)\les\max((p')^{1-\frac{p}{q}},\, p^{1-\frac{p'}{q'}})\mathsf{L}^{P^*}_{\tt, q', p'}(X^*)\,.$$
 \item Assume additionally that $1<q<\8$ and $\{T_t\}_{t>0}$ satisfies \eqref{Ana bound} for some $0<\b_0\le\frac\pi2$. Then $X$ is of Luzin cotype $q$ relative to $\{T_t\}_{t> 0}$ iff  $X^*$ is of Luzin type $q'$ relative to $\{T^*_t\}_{t> 0}$.  Moreover, the relevant constants satisfy
  $$\mathsf{L}^{T^*}_{\tt, q', p'}(X^*)\les\mathsf{L}^T_{\cc, q, p}(X)\les
  \b_q^{-4}\,\mathsf{T}_{\b_0}^{\min(\frac{p}{q},\,\frac{p'}{q'})}\,\max((p')^{1-\frac{p}{q}},\, p^{1-\frac{p'}{q'}})\,\mathsf{L}^{T^*}_{\tt, q', p'}(X^*)$$
   with $\b_q=\b_0\min(\frac{p}{q},\,\frac{p'}{q'})$.
 \end{enumerate}
 \end{thm}

 \begin{proof}
    (i) Assume that $X$ is of Luzin cotype $q$. Let $g\in L_{p'}(\O;X^*)$ with $\mathsf{F^*}(g)=0$. Let $ f\in L_p(\O;X)$. We want to estimate
    $\la f,\, g\ra$, where the duality bracket is that between $L_p(\O;X)$ and $L_{p'}(\O;X^*)$. By \eqref{decomposition} and its dual version, we can assume $\mathsf{F}(f)=0$, which, together with \eqref{kernel decomposition}, implies that $f\in  \ker\mathsf F=\overline{{\rm im}\sqrt A}$. With $F(z)=-ze^{-z}$ and by \eqref{resolution} we have
   $$f=4 \int_0^\8 F(t\sqrt A\,) F(t\sqrt A\,)(f)\,\frac{dt}t\,.$$
Thus by the H\"older inequality and  the Luzin cotype $q$ of $X$
  \begin{align*}
  \big|\la f,\, g\ra\big|
   &=4 \left|\int_0^\8   \la F(t\sqrt A\,)(f),\, F(t\sqrt {A^*})(g)\ra\,\frac{dt}t\right|\\
   &\le 4\big\|\mathcal{G}^P_q(f)\big\|_{L_p(\O)}\,\big\|\mathcal{G}^{P^*}_{q'}(g)\big\|_{L_{p'}(\O)}\\
   &\le 4\,\mathsf{L}^P_{\cc, q, p}(X)\big\|f\big\|_{L_p(\O;X)}\,\big\|\mathcal{G}^{P^*}_{q'}(g)\big\|_{L_{p'}(\O)}\,.
    \end{align*}
Taking the supremum over $f$ with norm 1, we show that $X^*$ is of Luzin type $q'$ with
 $$\mathsf{L}^{P^*}_{\tt, q', p'}(X^*)\le4\,\mathsf{L}^P_{\cc, q, p}(X).$$
To show the converse implication, let $f\in L_p(\O; X)$ and $h\in L_{p'}(\O;L_{q'}(\real_+;X^*))$ (recalling that $\real_+$ is equipped with $\frac{dt}t$). We have
   $$\int_0^\8\la t\partial P_t(f),\, h_t\ra\,\frac{dt}t= \int_0^\8\la f,\, t\partial P^*_t(h_t)\ra\,\frac{dt}t
   =\la f,\, g\ra,$$
  where
   $$g=\int_0^\8t\partial P^*_t(h_t)\,\frac{dt}t.$$
   Applying Theorem~\ref{proj} (i) to $\{P^*_t\}_{t> 0}$ on $L_{p'}(\O;X^*)$, we have
    $$\big\|\mathcal{G}^{P^*}_{q'}(g)\big\|_{L_{p'}(\O)}\les\max(p'^{1-\frac{p}{q}},\, p^{1-\frac{p'}{q'}}) \big\|h\big\|_{L_{p'}(\O;L_{q'}(\real_+;X^*))}\,.$$
   Combining the above inequalities,  we get
   $$\left|\int_0^\8\la t\partial P_t(f),\, h_t\ra\,\frac{dt}t\right|
   \les\max(p'^{1-\frac{p}{q}},\, p^{1-\frac{p'}{q'}})\mathsf{L}^{P^*}_{\tt, q', p'}(X^*)\big\|f\big\|_{L_p(\O;X)}\,\big\|h\big\|_{L_{p'}(\O;L_{q'}(\real_+;X^*))}\,,$$
  which implies the Luzin cotype $q$ of $X$ with
   $$\mathsf{L}^P_{\cc, q, p}(X)\les\max(p'^{1-\frac{p}{q}},\, p^{1-\frac{p'}{q'}})\mathsf{L}^{P^*}_{t, q', p'}(X^*)\,.$$

   (ii) The proof of this part is  similar by using Theorem~\ref{proj} (ii).
   \end{proof}

 \begin{cor}
 Any Banach space $X$ is of Luzin type $1$ relative to $\{P_t\}_{t>0}$, so relative to $\{T_t\}_{t>0}$ too. If $\{T_t\}_{t>0}$ is analytic on $L_p(\O; X)$, then $X$ is of Luzin cotype $\8$ relative to $\{T_t\}_{t>0}$, so $X$ is always of Luzin cotype $\8$ relative to the subordinated Poisson semigroup $\{P_t\}_{t>0}$.
 \end{cor}

\begin{proof}
 Indeed, let  $f\in L_p(\O; X)$ such that $\mathsf F(f)=0$. Then by \eqref{kernel decomposition} and \eqref{resolution} we have
   $$f=-\int_0^\8 t\partial P_t(f)\,\frac{dt}t\,,$$
 whence
   $$\big\|f\big\|_{L_p(\O;X)}\le\big\|\mathcal{G}^P_1(f)\big\|_{L_p(\O)}\,.$$
   Thus $X$ is of Luzin type $1$ relative to $\{P_t\}_{t>0}$, hence also relative to $\{T_t\}_{t>0}$ by virtue of Remark~\ref{Poisson vs Heat}. Passing to duality by means of Theorem~\ref{dual}, we see that  $X$ is of Luzin cotype $\8$ relative to $\{T_t\}_{t>0}$ under the analyticity assumption of $\{T_t\}_{t>0}$ on $L_p(\O; X)$.
\end{proof}

The following formulation of the Littlewood-Paley function $\mathcal{G}^P_q(f)$ by  discrete lacunary differences is of interest in its own right.

 \begin{thm}\label{LPS-diff}
 Let $X$ be a Banach space, $1\le q\le\8$ and $a>1$.
 \begin{enumerate}[\rm(i)]
  \item $X$ is of Luzin cotype $q$ relative to $\{P_t\}_{t> 0}$ iff  there exists a constant $c$ such that
  \beq\label{difference-cotype}
  \Big\|\Big(\sum_{k\in\ent}\big\|(P_{a^kt}- P_{a^{k+1}t})(f)\big\|_X^q\Big)^{\frac1q}\Big\|_{L_p(\O)}\le c\big\|f\big\|_{L_p(\O;X)}
  \eeq
 for all $1\le t\le a$ and $f\in L_p(\O;X)$. Moreover, the best $c$ and $\mathsf{L}^P_{\cc, q,p}(X)$ are linked by
  $$(\log a)^{-\frac1{q'}} c\le \mathsf{L}^P_{\cc, q, p}(X) \les (\log a)^{\frac1q} \,\frac{a+1}{a-1}\,\max((p')^{1-\frac{p}{q}},\, p^{1-\frac{p'}{q'}})c.$$
 \item $X$ is of Luzin type $q$ relative to $\{P_t\}_{t> 0}$ iff  there exists a constant $c$ such that
  \beq\label{difference-type}
  \big\|f-\mathsf F(f)\big\|_{L_p(\O;X)}\le c\Big\|\Big(\int_1^a\sum_{k\in\ent}\big\|(P_{a^kt}- P_{a^{k+1}t})(f)\big\|_X^qdt\Big)^{\frac1q}\Big\|_{L_p(\O)}
  \eeq
   for all $f\in L_p(\O;X)$. Moreover, the best $c$ and $\mathsf{L}^P_{\tt, q, p}(X)$ are linked by
  $$ \frac{a-1}{a+1}\big(\max((p')^{1-\frac{p}{q}},\, p^{1-\frac{p'}{q'}})\big)^{-1} c\
  \les \mathsf{L}^P_{\tt, q, p}(X) \les (a-1)^{\frac1q}(\log a)^{\frac1{q'}} c.$$
 \item Similar statements hold for $\{T_t\}_{t> 0}$ under the additional assumption that $\{T_t\}_{t> 0}$ be analytic on $L_p(\O;X)$ and $1<q<\8$.
  \end{enumerate}
 \end{thm}

 \begin{proof}
  (i) We have
   $$\big\|(P_{a^kt} - P_{a^{k+1}t})(f)\big\|_X^q=\Big\|\int_{a^kt}^{a^{k+1}t}\partial P_s(f) ds\Big\|_X^q
   \le(\log a)^{\frac{q}{q'}} \int_{a^kt}^{a^{k+1}t}\big\|s\partial P_s(f)\big\|_X^q\frac{ds}s\,,$$
which implies
  $$
   \sum_{k\in\ent}\big\|(P_{a^kt}- P_{a^{k+1}t})(f)\big\|_X^q\le (\log a)^{\frac{q}{q'}} \mathcal{G}^P_q(f)^q\,.
   $$
 Thus if $X$ is of Luzin cotype $q$ relative to $\{P_t\}_{t> 0}$, then \eqref{difference-cotype} holds with
  $c\le (\log a)^{\frac{1}{q'}}\mathsf{L}^P_{c, q, p}(X)\,.$

 To show the converse implication,  let $b=\frac12(1+a)$.  We use an idea from \cite{HN} (see also \cite{LP2}) to write
 $$
 \partial P_t
 =\sum_{k=0}^\8\big(\partial P_{b^{k}t}- \partial P_{b^{k+1}t}\big)
 =\sum_{k=0}^\8\partial P_{b^{k}2^{-1}t}\big(P_{b^{k}2^{-1}t}- P_{ab^{k}2^{-1}t}\big) .
 $$
Hence
  \begin{align*}
  \big\|\mathcal{G}^P_q(f)\big\|_{L_p(\O)}
  &\le \sum_{k=0}^\8  \Big\|\big\{t\partial P_{b^{k}2^{-1}t}\big(P_{b^{k}2^{-1}t}- P_{a b^{k}2^{-1}t}\big)(f)\big\}_{t>0}\Big\|_{L_p(\O;L_q(\real_+;X))}\\
  &= 2\sum_{k=0}^\8b^{-k} \Big\|\big\{t\partial P_t\big(P_{t}-P_{at}\big)(f)\big\}_{t>0}\Big\|_{L_p(\O;L_q(\real_+;X))}\\
  &=\frac{2(a+1)}{a-1}\Big\|\big\{t\partial P_t\big(P_{t}-P_{at}\big)(f)\big\}_{t>0}\Big\|_{L_p(\O;L_q(\real_+;X))}\,.
  \end{align*}
By Propositions~\ref{lq-analyticity} and \ref{A}, the family $\{t\partial P_t\}_{t>0}$ is $\el_q$-bounded with constant $C\max((p')^{1-\frac{p}{q}},\, p^{1-\frac{p'}{q'}})$. Therefore,
 $$\Big\|\big\{t\partial P_t\big(P_{t}-P_{at}\big)(f)\big\}_{t>0}\Big\|_{L_p(\O;L_q(\real_+;X))}\les\max((p')^{1-\frac{p}{q}},\, p^{1-\frac{p'}{q'}})\Big\|\big\{\big(P_{t}-P_{at}\big)(f)\big\}_{t>0}\Big\|_{L_p(\O;L_q(\real_+;X))}\,.$$
To estimate the norm on the right hand side, we write
 \begin{align*}
 \int_0^\8\big\|\big(P_{t}-P_{at}\big)(f)\big\|^q_X\,\frac{dt}t
 &=\sum_{k\in\ent} \int_{a^k}^{a^{k+1}}\big\|\big(P_{t}-P_{at}\big)(f)\big\|^q_X\,\frac{dt}t\\
 &=\int_1^{a}\sum_{k\in\ent} \big\|\big(P_{a^kt}-P_{a^{k+1}t}\big)(f)\big\|^q_X\,\frac{dt}t\,.
   \end{align*}
 Note that the function
  $$t\mapsto \Big(\sum_{k\in\ent} \big\|\big(P_{a^kt}-P_{a^{k+1}t}\big)(f)\big\|^q_X\Big)^{\frac1q}$$
is continuous from $\real_+$ to $L_p(\O)$, so there exists  $t_0\in[1,\, a]$ such that
  $$\int_1^{a}\sum_{k\in\ent} \big\|\big(P_{a^kt}-P_{a^{k+1}t}\big)(f)\big\|^q_X\,\frac{dt}t\le (\log a)\sum_{k\in\ent} \big\|\big(P_{a^kt_0}-P_{a^{k+1}t_0}\big)(f)\big\|^q_X\,.$$
We then deduce the Luzin cotype $q$ of $X$ from \eqref{difference-cotype}  with
 $$\mathsf{L}^P_{\cc, q, p}(X)\les(\log a)^{\frac1q}\,\frac{a+1}{a-1}\,\max((p')^{1-\frac{p}{q}},\, p^{1-\frac{p'}{q'}})c.$$

(ii)  The above argument yields the following discretization of $\mathcal{G}^P_q(f)$:
  \begin{align*}
  \frac{a-1}{a+1}&\big(\max((p')^{1-\frac{p}{q}},\, p^{1-\frac{p'}{q'}})\big)^{-1}\big\|\mathcal{G}^P_q(f)\big\|_{L_p(\O)}\\
  \;\;\les\Big\|&\Big(\int_1^{a}\sum_{k\in\ent}\big\|\big(P_{a^kt}-P_{a^{k+1}t}\big)(f)\big\|^q_X\,dt\Big)^{\frac1q}\Big\|_{L_p(\O)}\\
  \;\;&\le (a-1)^{\frac1q}(\log a)^{\frac1{q'}} \big\|\mathcal{G}^P_q(f)\big\|_{L_p(\O)}\,.
   \end{align*}
This immediately implies assertion (ii).

(iii) is proved similarly by virtue of Propositions~\ref{lq-analyticity} and \ref{T-analyticity}.
  \end{proof}

We have seen that the proofs of Theorem~\ref{dual} and Theorem~\ref{proj} are based on functional calculus for the special function $F(z)=-ze^{-z}$. It is known that functional calculus allows us to use more general functions.

\begin{definition}\label{g-function for B}
 Let $B$ be an $\el_q$-sectorial operator of type $\a$ on $L_p(\O;X)$ with $\a<\pi$. Let $\b>\a$ and $\f\in H^\8_0(\Sigma_\b)$ be a nonzero function. Define
  $$\mathcal{G}^B_{q, \f}(f)=\left(\int_0^\8\big\|\f(tB)(f)\big\|_X^q\,\frac{dt}t\right)^{\frac1q}\,, \quad  f\in L_p(\O; X).$$
\end{definition}

The following result is a variant of  \cite[Theorem~5]{MY} (see also \cite[Lemma~20]{LP2}).

\begin{prop}\label{mac}
Let $\f$ and $\psi$ be two  nonzero functions in $H^\8_0(\Sigma_\b)$. Then
 $$\big\|\mathcal{G}^B_{q,\f}(f)\big\|_{L_p(\O)}\le C_{B,q, \f,\psi} \big\|\mathcal{G}^B_{q,\psi}(f)\big\|_{L_p(\O)}\,,  \quad  f\in L_p(\O; X).$$
 \end{prop}

\begin{proof}
Let
 $$a=\int_0^\8\psi(t)^2\,\frac{dt}t\,.$$
Then
 $$1=\frac1a\int_0^\8\psi(tz)^2\,\frac{dt}t\,,\quad z\in\Si_\b.$$
Combined with  \eqref{resolution}, this  implies
 $$f=\frac1a\int_0^\8\psi(tB)^2(f)\,\frac{dt}t\,,\quad f\in\overline{{\rm im} B},$$
 whence
  $$\f(sB)(f)=\frac1a\int_0^\8\f(sB)\psi(tB)\big(\psi(tB)(f)\big)\,\frac{dt}t\,,\quad s>0.$$
Let $\a<\g<\b$ and $\Ga$ be the boundary of $\Si_\g$. We then deduce
 $$\f(sB)(f)=\frac1{2a\pi{\rm i}}\int_\Ga\int_0^\8\f(sz)\psi(tz)z(z-B)^{-1}\big(\psi(tB)(f)\big)\,\frac{dt}t\,\frac{dz}z\,.$$
These equalities are the analogues of \eqref{proj-fc} and \eqref{proj-int} with $h_t=\psi(tB)(f)$. It remains to repeat the proof of Theorem~\ref{proj} to conclude
 $$\big\|\mathcal{G}^B_{q,\f}(f)\big\|_{L_p(\O)}
 \le \frac1{2|a|\pi} \,C_{B,\g}C_{\f,\g}C_{\psi,\g}\big\|\mathcal{G}^B_{q,\psi}(f)\big\|_{L_p(\O)}\,,$$
where $C_{B,\g}$ is the $\el_q$-boundedness constant of $\big\{z(z-B)^{-1}: z\in\Ga\setminus\{0\}\big\}$,
 $$C_{\f,\g} =\max_{\e=\pm1}\int_0^\8|\f(te^{{\rm i}\e\g})|\,\frac{dt}t$$
 and $C_{\psi,\g}$ is similarly defined.
\end{proof}

In particular, combining the previous proposition with the results in section~\ref{boundedness}, we obtain the following

\begin{cor}\label{heat LPS-f}
  Let $X$ be a Banach space and $1\le q\le\8$.
  \begin{enumerate}[\rm(i)]
  \item $X$ is of Luzin cotype $q$ relative to $\{P_t\}_{t>0}$ iff for every nonzero $\f\in H_0^\8(\Si_\b)$ with $\b>\pi/4$ $($equivalently, for some nonzero $\f\in H_0^\8(\Si_\b))$ there exists a constant $c$ such that
   $$\big\|\mathcal{G}^{\sqrt A}_{q,\f}(f)\big\|_{L_p(\O)}\le c\,\big\|f\big\|_{L_p(\O; X)}\,,\quad f\in L_p(\O; X).$$
   \item $X$ is of Luzin type $q$ relative to $\{P_t\}_{t>0}$ iff for every nonzero $\f\in H_0^\8(\Si_\b)$ with $\b>\pi/4$ $($equivalently, for some nonzero $\f\in H_0^\8(\Si_\b))$ there exists a constant $c$ such that
   $$\big\|f-\mathsf{F}(f)\big\|_{L_p(\O; X)}\le c\,\big\|\mathcal{G}^{\sqrt A}_{q,\f}(f)\big\|_{L_p(\O)}\,,\quad f\in L_p(\O; X).$$
 \item Similar statements hold for $\{T_t\}_{t>0}$ when $\{T_t\}_{t>0}$ is analytic on $L_p(\O; X)$ and $1<q<\8$.
   \end{enumerate}
 \end{cor}

We conclude this section by some remarks on general $\el_q$-sectorial operators on $L_p(\O;X)$ for which we  have defined the $g$-function in Definition~\ref{g-function for B}. In fact, what we have done so far for semigroups can be developed for these operators too.

\begin{definition}
 Let $B$ be an $\el_q$-sectorial operator of type $\a$ on $L_p(\O;X)$ with $\a<\pi$.
   \begin{enumerate}[\rm (i)]
   \item $X$ is said to be of {\it Luzin cotype $q$ relative to} $B$ if there exists a constant $c$ such that
 $$
 \big\|\mathcal{G}^B_{q,\f}(f)\big\|_{L_p(\O)}\le c\big\|f\big\|_{L_p(\O;X)}
 $$
for every $f\in \overline{{\rm im} B}$ and some nonzero $\f\in H_0^\8(\Si_\b)$ with $\b>\a$.
\item $X$ is said to be of {\it Luzin type $q$ relative to} $B$ if there exists a constant $c$ such that
 $$
\big\|f\big\|_{L_{p}(\O;X)} \le c\big\|\mathcal{G}^B_{q,\f}(f)\big\|_{L_{p}(\O)} 
 $$
for every $f\in \overline{{\rm im} B}$ and some nonzero $\f\in H_0^\8(\Si_\b)$ with $\b>\a$.
  \end{enumerate}
\end{definition}

Proposition~\ref{mac} shows that the above definition is independent of the choice of $\f$.  Assume additionally that $B$ admits a dual operator $B'$ on $L_{p'}(\O; X^*)$ in the sense of \cite{CDMY}, namely,
 $$\la B(f),\,g\ra=\la f,\,B'(g)\ra, \quad f\in{\rm Dom} B, \;g\in{\rm Dom} B'.$$
Assume further that  $B'$ is also $\el_{q'}$-sectorial of type $\a$ on $L_{p'}(\O;X^*)$.

\begin{prop}
 Under the above assumption, $X$ is of Luzin cotype $q$ relative to $B$ iff $X^*$ is of Luzin type $q'$ relative to $B'$.
\end{prop}

\begin{proof}
 Noting that Theorem~\ref{proj} transfers to the present setting, we can repeat the proof of Theorem~\ref{dual}, and so omit the details.
\end{proof}

It would be interesting to investigate the Luzin type and cotype relative to $B$ as above. Guided by Theorems~\ref{Poisson ML} and \ref{Heat ML}, one would like to know those operators $B$ such that the Luzin type or cotype relative to $B$ is implied by the martingale type or cotype. It seems that more structure should be imposed to $B$ in order to get significant results. We have seen that this is indeed the case if $B$ is the negative generator of $\{P_t\}_{t>0}$ or $\{T_t\}_{t>0}$ on $L_p(\O; X)$.  On the other hand, we have the following proposition that is contained (essentially)  in \cite{VW}. Note that \cite{VW} can be viewed as the particular case of our discussion  where $\O$ is a singleton (the $\el_q$-boundedness then  simply becomes the usual boundedness).

The notion of type and cotype referred in the next proposition is the usual Rademacher type and cotype.

\begin{prop}
 Assume that $B$ has a bounded $H^\8$ functional calculus. If $X$ is of  cotype $($resp. type$)$ $q$, then $X$ is of Luzin cotype $($resp. Luzin type $)$ $q$ relative to $B$.
\end{prop}

\begin{proof}
 As in \cite{VW}, this is a simple consequence of Kalton-Weis'  theorem on the unconditionality of bounded $H^\8$ functional calculus (see Theorem~10.4.6 or its discrete version, Theorem~10.4.4. in  \cite{HVVW}).
\end{proof}

\begin{rk}
  The assumption on the $H^\8$ functional calculus of $B$ seems too strong since it implies that $X$ is a UMD space in many cases, for instance, if $B$ is the negative generator of $\{T_t\}_{t>0}$ on $L_p(\real; X)$ when $\{T_t\}_{t>0}$ is the heat or Poisson semigroup on $\real$.
\end{rk}


\section{Dyadic martingales and singular integrals}\label{Dyadic martingales and singular integrals}


The proofs of Theorem~\ref{Poisson ML} (i) and Theorem~\ref{fML} heavily rely on tools from harmonic analysis, notably from modern real-variable Littlewood-Paley theory. This section is a preparation for using these tools.  We  will mainly follow Wilson's beautiful treatment  in \cite{Wilson08} (see also \cite{Wilson89, Wilson07})  for the part on  Littlewood-Paley theory.

\subsection{Dyadic martingales}\label{Dyadic martingales}

 All cubes in $\real^d$ considered in the sequel are bounded and with sides parallel to the axes. $|Q|$ and $\el(Q)$ denote respectively the volume and side length of a cube $Q$; $tQ$ stands for the cube with the same center as $Q$ and $\el(tQ)=t\el(Q)$ for $t>0$.

 Let $\D$ be the family of all dyadic cubes of $\real^d$, $\D_k\subset\D$ the subfamily of all cubes with side length $2^{-k}$ for $k\in\ent$. Let $\A_k$ be the $\s$-algebra generated by $\D_k$ and $\E_k$  the associated conditional expectation. For $f\in L_p(\real^d; X)$,
  $$\E_k(f)=\sum_{Q\in\D_k}\Big(\frac1{|Q|}\,\int_Qf\Big)\un_{Q}.$$
Let $d_k(f)=\E_k(f)-\E_{k-1}(f)$ and
 $$S_q(f)=\Big(\sum_{k\in\ent}\big\|d_k(f)\big\|_X^q\Big)^{\frac1q}.$$
$S_q(f)$ is the $q$-variant of the usual martingale square function of $f$. It is useful to note that $d_k(f)$ is of vanishing mean on every $Q\in\D_{k-1}$ and constant on every $R\in\D_k$.

Thus if $X$ is of martingale type $q$, then
 $$\big\|f\big\|_{L_q(\real^d; X)}\le \mathsf{M}_{\tt,q}(X) \big\|S_q(f)\big\|_{L_q(\real^d)}\,,\quad f\in L_q(\real^d; X).$$

We will need {\it dyadic-like families of cubes} that Wilson calls good families. $\F$ is such a family if
\begin{itemize}
\item[a)] for $Q\in\F$, all its  $2^d$  immediate dyadic subcubes belong to $\F$ too;
\item[b)] every $Q\in\F$ is one of the $2^d$  immediate dyadic subcubes of another one in $\F$;
\item[c)] for all $Q, R\in\F$,  $Q\subset R$, or $R\subset Q$, or $Q\cap R=\emptyset$.
\end{itemize}
For $\F$ a  dyadic-like family of cubes, we define the associated $S_{q,\F}$:
$$S_{q,\F}(f)=\Big(\sum_{Q\in\F}\,\sum_{R\in\F,\, R\subset Q,\,\el(R)=\frac{\el(Q)}2}\big\|\frac1{|R|}\,\int_Rf-\frac1{|Q|}\,\int_Qf\big\|^q_X\,\un_R\Big)^{\frac1q}.$$
 It is easy to see that given a finite number of cubes in $\F$,  we can bring the subfamily consisting of those cubes in $\F$ that are contained in one of the given cubes  to a subfamily of $\D$ after appropriate translation and rescaling. Thus we have the following.

 \begin{lem}\label{Haar-like type}
 Let $\F$ be a dyadic-like family of cubes. If $X$ is of martingale type $q$, then
 $$\big\|f\big\|_{L_q(\real^d;X)}\le \mathsf{M}_{\tt,q}(X)\big\|S_{q,\F}(f)\big\|_{L_q(\real^d)}$$
for all $f\in L_q(\real^d; X)$ supported by cubes from $\F$.
 \end{lem}

 An important case needed later concerns the family $\{3Q\}_{Q\in\D}$. The following  is due to Wilson \cite{Wilson89}.

\begin{lem}\label{3Q}
The family $\{3Q\}_{Q\in\D}$ is a disjoint union of $3^d$ dyadic-like families.
\end{lem}

It suffices to consider the case $d=1$. Then every $3Q$ can be written in the form $[\frac{3j+s}{2^k},\, \frac{3(j+1)+s}{2^k})$ with $j, k\in\ent$ and $s\in\{0, 1,2\}$. Let $\F^k_{s}$ be the collection of all such intervals for given $k$ and $s$. Then the desired union is
 $\cup_{s=0}^2\cup_{k\in\ent }\F^k_{2^{|k|}s\; {\rm mod}\;3}$.

\medskip

Let  $\F$ be a dyadic-like family of cubes and $\d>0$. Consider a family $\{a_Q\}_{Q\in\F}$ of $X$-valued functions satisfying the following conditions:
 \beq\label{atom}
 {\rm supp}(a_Q)\subset Q,\quad \int a_Q=0\;\text{ and }\; \|a_Q(x)-a_Q(y)\|_X\le |Q|^{-\frac1q}\,\big(\frac{|x-y|}{\el(Q)}\big)^\d.
 \eeq
These are \textit{smooth atoms}. Let $\{\l_Q\}_{Q\in\F}$ be a finite family of complex numbers and $f=\sum_{Q\in\F} \l_Q \,a_Q$.

The following is the adaptation of a lemma due to Wilson to the present setting. We include its proof for the convenience of the reader.

\begin{lem}\label{atom type}
 Under the above assumption, we have
 $$S_{q,\F}(f)\les_{d,\d}  \Big(\sum_{Q\in\F}\frac{|\l_Q|^q}{|Q|}\,\un_Q\Big)^{\frac1q}.$$
 \end{lem}

\begin{proof}
Without loss of generality, we assume that $\F=\D$, so $S_{q,\F}(f)=S_{q}(f)$. Since $a_Q$ is of vanishing mean,  $d_k(a_Q)=0$ whenever $k\le k(Q)$, where $2^{-k(Q)}=\el(Q)$ . Let $R\in\D_{k-1}$ with $k>k(Q)$ and $R\subset Q$. Then on $R$,
 \begin{align*}
 d_k(a_Q)
 &=\sum_{I\in\D_{k},\, I\subset R} \big(\frac1{|I|}\,\int_Ia_Q-\frac1{|R|}\,\int_Ra_Q\big)\un_I\\
 &=\sum_{I\in\D_{k},\, I\subset R} \frac1{|I|\,|R|}\,\int_{I\times R}\big(a_Q(x)-a_Q(y)\big)dx\,dy\,\un_I.
 \end{align*}
 Thus by the last condition of \eqref{atom},
  $$\|d_k(a_Q)\|_X\les_{d}|Q|^{-\frac1q}\,\big(\frac{\el(R)}{\el(Q)}\big)^\d;$$
 hence on $R$,
  \begin{align*}
  \|d_k(f)\|_X
  &\les_{d}\Big(\sum_{Q: Q\supset R}\frac{|\l_Q|^q}{|Q|}\,\big(\frac{\el(R)}{\el(Q)}\big)^\d\Big)^{\frac1q}
  \Big(\sum_{Q: Q\supset R}\big(\frac{\el(R)}{\el(Q)}\big)^\d\Big)^{\frac1{q'}}\\
  &\les_{d,\d}\Big(\sum_{Q: Q\supset R}\frac{|\l_Q|^q}{|Q|}\,\big(\frac{\el(R)}{\el(Q)}\big)^\d\Big)^{\frac1q}.
   \end{align*}
 It then follows that
 \begin{align*}
 S_{q}(f)^q(x)
 &=\sum_{k\in\ent}\,\sum_{R\in\D_{k-1}} \|d_k(f)\|^q_X\,\un_R(x)\\
 &\les_{d,\d}\sum_{R: x\in R}\, \sum_{Q: Q\supset R}\frac{|\l_Q|^q}{|Q|}\,\big(\frac{\el(R)}{\el(Q)}\big)^\d\,\un_R(x)\\
 &\les_{d,\d}\sum_{Q}\frac{|\l_Q|^q}{|Q|} \;\sum_{R: x\in R\subset Q} \,\big(\frac{\el(R)}{\el(Q)}\big)^\d\,\un_R(x)\\
 &\les_{d,\d}\sum_{Q}\frac{|\l_Q|^q}{|Q|}\,\un_Q(x)\,.
  \end{align*}
This gives  the desired assertion.
\end{proof}

\subsection{Singular integrals}\label{Singular integrals}

Given $\e>0$ and $\d>0$, let $\H_{\e, \d}$ be the class of all integrable functions $\f$ on $\real^d$ satisfying \eqref{Holder}. Let $\f\in \H_{\e, \d}$. We consider  the vector-valued kernel  $K$ defined by $K(x)=\{\f_t(x)\}_{t>0}$ for $x\in\real^d$, that is, $K$ is a function from $\real^d$ to $L_q(\real_+)$.  With a slight abuse of notation, we use $K$ to denote the associated singular integral too:
   $$K(f)=\int_\real K(x-y)f(y)dy.$$
 Then
 $$
G_{q,\f}(f)(x)=\big\|K(f)(x)\big\|_{L_q(\real_+; X)}\,,\quad x\in\real^d.
 $$

\begin{lem}\label{Hormander}
 The kernel $K$ satisfies the  regularity properties:
 $$
 \big\|K(x)\big\|_{L_q(\real_+)} \les_{\e} \frac{1}{|x|^d}\;\text{ and }\; \big\|K(x+y)-K(x)\big\|_{L_q(\real_+)} \les_{\e} \frac{|y|^\d}{|x|^{d+\d}},\quad x, y\in\real^d,\; |x|>2|y|.
 $$
  \end{lem}

\begin{proof}
 Let $x\in\real\setminus\{0\}$. Then  by \eqref{Holder},
  \begin{align*}
  \big\|K(x)\big\|_{L_q(\real_+)}^q
  &=\int_0^\8|\f_t(x)|^q\,\frac{dt}t
  \le \int_0^{\8}\big[\frac1{t^d}\,\frac{1}{\big(1+\frac{|x|}t\big)^{d+\e}}\big]^q\,\frac{dt}t\\
  &=\frac1{|x|^{dq}}\int_0^{\8}\frac{t^{\e q}}{(1+t)^{(d+\e)q}}\,\frac{dt}t
  \les_\e \frac1{|x|^{dq}}
 \end{align*}
 Similarly,
   \begin{align*}
  \big\|K(x+y)-K(x)\big\|_{L_q(\real_+)}^q
  &\les|y|^{\d q} \int_0^{\8}\big[\frac1{t^{d+\d}}\,\frac{1}{\big(1+\frac{|x|}t\big)^{d+\e+\d}}\big]^q\,\frac{dt}t\\
  &=\frac{|y|^{\d q}}{|x|^{(d+\d)q}}\int_0^{\8}\frac{t^{\e q}}{(1+t)^{(d+\e+\d)q}}\,\frac{dt}t
  \les_{\e} \frac{|y|^{\d q}}{|x|^{(d+\d)q}}.
 \end{align*}
 This lemma is proved. \end{proof}

\subsection{A quasi-orthogonal decomposition}\label{An atomic decomposition}

Beside $\H_{\e, \d}$ introduced in the previous subsection, we will need its subclass of functions supported in the unit ball. More precisely, let $\H^0_{\d}$ be the class of integrable functions  $\f$ on $\real^d$ such that
 \beq\label{Holder0}
   {\rm supp}(\f)\subset B(0, 1),\quad  |\f(x)-\f(y)|\le |x-y|^\d,\quad \int_{\real^d}\f(x)dx=0.
  \eeq
Here $B(x, t)$ denotes the ball of $\real^d$ with center $x$ and radius $t$.

\smallskip

Any function in $\H_{\e, \d}$ can be decomposed into a series of functions in  $\H^0_{\d}$ thanks to  the following lemma due to Uchiyama \cite{Uch} (see also \cite{Wilson07}).

\begin{lem}\label{Uchiyama}
 Let $\f\in \H_{\e, \d}$. Then there exist a positive constant $C_{\e,\d}$ and a sequence of functions $\p^{(k)}\in\H^0_{\d}$ such that
 $$\f=C_{\e,\d}\sum_{k=0}^\82^{-\e k}(\p^{(k)})_{2^k}\,.$$
    \end{lem}

\begin{proof} The proof is  elementary. Let $\eta$ be a smooth function supported in $\{x\in\real^d: \frac12<|x|<2\}$ such that
 $$\sum_{k\in\ent}\eta(2^{-k}x)=1,\quad\forall\, x\in\real^d\setminus\{0\}.$$
Define
 $$\rho_0(x)=\sum_{j\le-1}\eta(2^{-j}|x|),\quad \rho_k(x)=\eta(2^{-k+1}|x|)\;\text{ for}\; k\ge1$$
 and
  $$\zeta_k=\frac{\sum_{0\le j\le k}\int_{\real^d}\rho_j\f}{\int_{\real^d}\rho_k}\,\rho_k\,.$$
Then  the desired decomposition is given by
  $$\f=(\f\rho_0-\zeta_0)+\sum_{k=1}^\8(\f\rho_k-\zeta_k+\zeta_{k-1})=C_{\e,\d}\sum_{k=0}^\82^{-\e k}(\p^{(k)})_{2^k}.$$
\end{proof}

By our convention that $\real_+$ is equipped with the measure $\frac{dt}t$, the upper half space $\real^{d+1}_+$ is equipped with the product measure $\frac{dx\,dt}{t}$. Consistent with our convention before, we write a function $h:\real^{d+1}_+\to X$ as $h(x, t)=h_t(x)$ for $x\in\real^d$ and $t\in\real_+$. Let $\f\in\H_{\e, \d}$ and $h\in L_q(\real^{d+1}_+; X)$ with compact support. Consider the following function
 $$g(x)=\int_{\real^{d+1}_+}\f_t(y-x)h_t(y)\,\frac{dy\,dt}{t}.$$
We will decompose $g$ into a series of smooth atoms 
 $$g=\sum_i\l_ia_i,$$
where the $a_i$'s satisfy \eqref{atom} relative to $\{3Q\}_{Q\in\D}$ and the $\l_i$'s are reals such that
 $$\Big(\sum_i|\l_i|^q\big)^{\frac1q}\les_{d, \e, \d} \big\|h\big\|_{L_q(\real^{d+1}_+; X)}.$$
This is the so-called the \textit{quasi-orthogonal decomposition} of $g$. First, using Lemma~\ref{Uchiyama}, we reduce our problem to the case where $\f$ is supported in the unit ball:
  \begin{align*}
  g(x)
  &=C_{\e,\d}\sum_{k=0}^\82^{-\e k}\int_{\real^{d+1}_+}(\p^{(k)})_{2^kt}(y-x)h_t(y)\,\frac{dy\,dt}{t}\\
  &=C_{\e,\d}\sum_{k=0}^\82^{-\e k}\int_{\real^{d+1}_+}(\p^{(k)})_{t}(y-x)h_{2^{-k}t}(y)\,\frac{dy\,dt}{t}.
   \end{align*}
 Note that $h_{2^{-k}\,\cdot}$ has the same norm as $h$ in $L_q(\real^{d+1}_+; X)$. Thus it suffices to do the decomposition for each $\p^{(k)}$ in place of $\f$.

In the following, we will assume that $\f$ itself belongs to $\H^0_{\d}$. The argument below is  classical. For $Q\in\D$, let $T_Q=\{(y, t): y\in Q,\; \frac{\el(Q)}2<t\le \el(Q)\}$. Then $\{T_Q\}_{Q\in\D}$ is a partition of $\real^{d+1}_+$. So
 $$g(x)=\sum_{Q\in\D}  \int_{T_Q} \f_t(y-x)h_t(y)\,\frac{dy\,dt}{t}\;{\mathop=^{\rm def}}\,\sum_{Q\in\D}\l_Q a_Q(x)$$
with
 $$\l_Q=\Big(\int_{T_Q}\|h_t(y)\|_{X}^{q}\,\frac{dy\,dt}{t}\Big)^{\frac1q}.$$
Clearly,
  $$\sum_{Q\in\D}|\l_Q|^q= \int_{\real^{d+1}_+}\|h_t(y)\|_X^q\,\frac{dy\,dt}{t}=\big\|h\big\|^q_{L_q(\real^{d+1}_+; X)}.$$
Since $\f$ is supported in the unit ball and of vanishing mean, we see that $a_Q$ is supported in $3Q$ and of vanishing mean too.
On the other hand, since $\f$ is in the H\"older class  $\H^0_{\d}$, by the H\"older inequality,
    \begin{align*}
   \|a_Q(x)-a_Q(x')\|_{X}
   &\le\Big(\int_{T_Q} |\f_t(y-x)-\f_t(y-x')|^{q'}\,\frac{dy\,dt}{t}\Big)^{\frac1{q'}}\\
   &\le \Big( \int_{T_Q}\big[\frac1{t^d}\,\big(\frac{|x-x'|}t\big)^\d\big]^{q'}\,\frac{dy\,dt}{t}\Big)^{\frac1{q'}}\\
  &\les_{d, \d} |Q|^{-\frac1{q}}\,\big(\frac{|x-x'|}{\el(Q)}\big)^\d.
    \end{align*}
Thus $a_Q$ is a smooth atom. This yields the desired quasi-orthogonal decomposition.

\medskip

Combining the above discussion with Lemmas~\ref{Haar-like type} -- \ref{atom type}, we get the following

\begin{lem}\label{atomic dec}
Keep the above notation and assume that $X$ is of martingale type $q$. Then
 $$\big\|g\big\|_{L_q(\real^d; X)}\les_{d, \e,\d} \mathsf{M}_{\tt, q}(X)\,\big\|h\big\|_{L_q(\real^{d+1}_+; X)}.$$
 \end{lem}


\section{Proofs of Theorem~\ref{fML} and Corollary~\ref{NY}}\label{Proof of Theorem fML}


With the preparation in section~\ref{Dyadic martingales and singular integrals}, we are in a position to show Theorem~\ref{fML}.
 For clarity, we divide the proof  into several steps. $X$ will be assumed of martingale cotype $q$ in the first four steps, and of martingale type $q$ in the last step.

\medskip\n{\it Step~1: A weighted norm inequality}.  Let $\f$ be a function satisfying \eqref{Holder}, i.e., $\f\in\H_{\e, \d}$. Beside the $g$-function defined by \eqref{G-function}, we will need the Luzin integral function
  $$S_{q,\f}(f)(x)=\Big(\int_{|y-x|<t} \big\|\f_t*f(y)\big\|_X^q\,\frac{dydt}{t^{d+1}}\Big)^{\frac1q}$$
 for nice $f:\real^d\to X$. The key of this proof is the following weighted norm inequality:

 For any locally integrable nonnegative function $w$ on $\real^d$ and any $f\in L_q(\real^d; X)$
   \beq\label{weighted Sd}
   \Big(\int_{\real^d}\big(S_{q,\f}(f)(x)\big)^qw(x)dx\Big)^{\frac1q}\les_{d,\e,\d} \mathsf{M}_{\cc, q}(X)\Big(\int_{\real^d}\big\|f(x)\big\|_X^qM(w)(x)dx\Big)^{\frac1q},
   \eeq
 where $M(w)$ denotes the Hardy-Littlewood maximal function of $w$:
  $$M(w)(x)=\sup\Big\{\frac1{|B|}\,\int_B w: x\in B, B\text{ ball}\Big\}.$$

First consider the unweighted case, i.e., $w\equiv1$.
By the Fubini theorem, we have
 $$\int_{\real^d}\big(S_{q,\f}(f)(x)\big)^qdx=c_d\int_{\real^{d+1}_+}\|\f_t*f(x)\|^q_X\,\frac{dx\,dt}{t}\,.$$
Let $h:\real^{d+1}_+\to X^*$ be a compactly supported smooth function such that
 $$\int_{\real^{d+1}_+}\|h_t(x)\|^{q'}_{X^*}\,\frac{dx\,dt}{t}\le1.$$
Then
 \begin{align*}
  \int_{\real^{d+1}_+}\la \f_t*f(y),\, h_t(y)\ra\,\frac{dy\,dt}{t}
  =\int_{\real^d}\la f(x),\, g(x)\ra dx\,,
   \end{align*}
 where
  $$g(x)=\int_{\real^{d+1}_+} \f_t(y-x)h_t(y)\,\frac{dy\,dt}{t}.$$
 Recall that it is well known (and easy to check) that $X$ is of martingale type $q$ iff $X^*$ is of martingale cotype $q'$ with the following relation between the relevant constants:
  \beq\label{duality type-cotype}
  \mathsf{M}_{\tt,q}(X)\le \mathsf{M}_{\cc,q'}(X^*)\le 2 \mathsf{M}_{\tt,q}(X).
  \eeq
Thus applying Lemma~\ref{atomic dec} to $h$ and $g$ with $(X^*, q')$ in place of $(X, q)$ there, we get
  \begin{align*}
 \|g\|_{L_{q'}(\real^d; X^*)}
 \les_{d, \e,\d}\mathbf{M}_{\tt, q'}(X^*)\,\big\|h\big\|_{L_{q'}(\real^{d+1}_+; X^*)}
 \les_{d,\e,\d}\mathbf{M}_{\cc, q}(X).
  \end{align*}
Then taking the supremum over all $h$ in the unit ball of $L_{q'}(\real^{d+1}_+; X^*)$, we deduce
 $$\big\|S_{q,\f}(f)\big\|_{L_q(\real^d)}\les_{d,\e,\d} \mathbf{M}_{\cc, q}(X)\big\|f\big\|_{L_q(\real^d; X)}.$$
Namely, the unweighted version of \eqref{weighted Sd} holds.

\medskip

We will deduce the weighted version by a trick from \cite{CWW}. By Lemma~\ref{Uchiyama}, we can assume that $\f$ is supported in the unit ball of $\real^d$. Given a weight $w$, we write (recalling that $B(y, t)$ denotes the ball of center $y$ and radius $t$)
 $$\int_{\real^d}\big(S_{q,\f}(f)(x)\big)^q w(x)dx
 =c_d\int_{\real^{d+1}_+}\|\f_t*f(y)\|_X^q\,\Big(\frac1{|B(y,t)|}\int_{B(y,t)}w(x)dx\Big)\,\frac{dy\,dt}{t}.$$
Let
 $$F_k=\big\{ (y, t)\in\real^{d+1}_+:  2^k<\frac1{|B(y,t)|}\int_{B(y,t)}w(x)dx\le 2^{k+1}\big\},\quad k\in\ent.$$
Clearly,  $(y, t)\in F_k$ implies $B(y, t)\subset E_k=\{x: M(w)(x)>2^k\}$. Together with the fact that $\f_t(y-\cdot)$ is supported in $B(y, t)$, this implies $\f_t*f(y)=\f_t*(f\un_{E_k})(y)$ whenever $(y, t)\in F_k$. Thus using the unweighted version already proved (applied to $f\un_{E_k}$), we deduce
   \begin{align*}
 \int_{\real^d}\big(S_{q,\f}(f)(x)\big)^q w(x)dx
 &\les_d\sum_{k\in\ent} 2^k \int_{F_k}\|\f_t*f(y)\|_X^q\,\frac{dy\,dt}{t}\\
 &\les_d\sum_{k\in\ent} 2^k \int_{F_k}\|\f_t*(f\un_{E_k})(y)\|_X^q\,\frac{dy\,dt}{t}\\
  &\les_{d,\e,\d}\mathbf{M}_{\cc, q}(X)^{q}\sum_{k\in\ent} 2^k \int_{E_k}\|f(x)\|_X^q\,dx\\
 &\les_{d,\e,\d}\mathbf{M}_{\cc, q}(X)^{q}\ \int_{\real^d}\|f(x)\|_X^qM(w)(x)\,dx.
  \end{align*}
Thus \eqref{weighted Sd} is proved.

\medskip\n{\it Step~2: Another weighted norm inequality}.  We need to show that \eqref{weighted Sd} remains valid for $G_{q,\f}$ instead of $S_{q,\f}$:
 \beq\label{weighted Gd}
   \Big(\int_{\real^d}\big(G_{q,\f}(f)(x)\big)^qw(x)dx\Big)^{\frac1q}\les_{d,\e,\d} \mathbf{M}_{\cc, q}(X)\Big(\int_{\real^d}\big\|f(x)\big\|_X^qM(w)(x)dx\Big)^{\frac1q},
   \eeq
To this end, we have to control the $g$-function by the Luzin area function.  If $\f$ is the Poisson kernel, this is a classical fact thanks to harmonicity.  In the present setting, we need a little bit more efforts.

If additionally all partial derivatives of $\f$ with order up to $d$  belong to $\H_{\e, \d}$,  then we can show
 $$G_{q,\f}(f)(x)\les_d \sum_{|\a|\le d}S_{q,D^\a\f}(f)(x) ,\quad x\in\real^d,$$
where $D^\a=\frac{\partial^{\a_1}}{\partial x_1^{\a_1}}\cdots \frac{\partial^{\a_d}}{\partial x_d^{\a_d}}$ for $\a=(\a_1,\cdots, \a_d)$ and $|\a|=\a_1+\cdots+\a_d$. The proof of this inequality is elementary (see \cite[Lemma~4.3]{XXX}).
Thus \eqref{weighted Gd} holds for such $\f$.

For a general $\f$, we need to adapt the arguments of \cite{Wilson07} to the present setting by introducing the vector-valued $q$-variants of Wilson's intrinsic square functions:
  \begin{align*}
  S_{q,\e, \d}(f)(x)^q
  &=\int_{|y-x|<t}\,\sup_{\f\in\H_{\e, \d}}\big\|\f_t*f(y)\big\|_X^q\,\frac{dy\,dt}{t^{d+1}},\\
 G_{q,\e, \d}(f)(x)^q
 &=\int_{0}^\8\,\sup_{\f\in\H_{\e, \d}}\big\|\f_t*f(x)\big\|_X^q\,\frac{dt}{t}.
   \end{align*}
One can show, quite easily, that $S_{q,\e, \d}(f)(x)\approx_{d,\e,\d} G_{q,\e, \d}(f)(x)$ for every $x\in\real^d$ (see \cite{Wilson07}  for more details).

On the other hand, for a compactly supported smooth function $h:\real^{d+1}_+\to X^*$ choose a family
$\{\f^{(y, t)}\}_{(y, t)\in K}\subset\H_{\e,\d}$ ($K$ being the support of $h$) such that
 $$ \big\|\f^{(y, t)}_t*f(y)\big\|_X\ge \frac12\,\sup_{\f\in\H_{\e,\d}}\big\|\f_t*f(y)\big\|_X\,,\quad (y,t)\in K.$$
Then by adapting the arguments in subsection~\ref{An atomic decomposition} to the present situation and by repeating step~1, one can estimate the integral
 \begin{align*}
  \int_{\real^{d+1}_+}\la \f^{(y, t)}_t*f(y),\, h_t(y)\ra\,\frac{dy\,dt}{t}
  =\int_{\real^d}\la f(x),\, \int_{\real^{d+1}_+} \f^{(y, t)}_t(y-x)h_t(y)\,\frac{dy\,dt}{t}\ra dx\,
   \end{align*}
 to conclude that
  $$\big\|S_{q,\e,\d}(f)\big\|_{L_q(\real^d)}\les_{d,\e,\d} \mathbf{M}_{c, q}(X)\big\|f\big\|_{L_q(\real^d; X)}.$$
This implies \eqref{weighted Sd} with $S_{q,\p}$ replaced by $S_{q,\e,\d}$ by the passage from the unweighted case to the weighted one. Then the pointwise equivalence  $S_{q,\e, \d}(f)\approx_{d,\e, \d} G_{q,\e, \d}(f)$ shows that \eqref{weighted Gd} holds for $G_{q,\e, \d}$ instead of  $G_{q,\f}$, whence \eqref{weighted Gd} for every $\f\in \H_{\e, \d}$.

  \medskip\n{\it Step~3: Proof of Theorem~\ref{fML} (i) for $p\ge q$}. We can now easily prove part (i) of Theorem~\ref{fML} for $p\ge q$. Indeed, the case $p=q$ is just the unweighted version of \eqref{weighted Gd}. For $p>q$, let $w$ be a nonnegative function on $\real^d$ with $L_{r}$-norm  equal to 1, where $r$ is the conjugate index of $\frac{p}q$. Then for $f\in L_p(\real; X)$,
  \begin{align*}
  \int_{\real^d}\big(G_{q,\f}(f)(x)\big)^qw(x)dx
 &\les_{d,\e,\d} \mathbf{M}_{\cc, q}(X)^q\,\int_{\real^d}\big\|f(x)\big\|_X^qM(w)(x)dx\\
 &\les_{d,\e,\d} \mathbf{M}_{\cc, q}(X)^q\,\big\|f\big\|_{L_p(\real^d; X)}^q\,\big\|M(w)\big\|_{L_{r}(\real^d)}\\
 &\les_{d,\e,\d} r'\, \mathbf{M}_{\cc, q}(X)^q\,\big\|f\big\|^q_{L_p(\real^d; X)}.
  \end{align*}
 Taking the supremum over all $w$, we get
  $$\big\|G_{q,\f}(f)\big\|_{L_p(\real^d)} \les_{d,\e,\d} p^{\frac1q}\, \mathbf{M}_{\cc, q}(X)\,\big\|f\big\|_{L_p(\real^d; X)},$$
 whence $ \mathsf{L}^\f_{\cc, q,p}(X) \les_{d,\e,\d} p^{\frac1q}\, \mathbf{M}_{\cc, q}(X)$.

 \medskip\n{\it Step~4: Proof of Theorem~\ref{fML} (i) for $p< q$}.  We deal with the case $p<q$  by singular integrals. Let $K$ be the singular integral associated to $\f$ as in subsection~\ref{Singular integrals}. We reduce to showing that $K$ is bounded from $L_p(\real^d; X)$ to $L_p(\real^d; L_q(\real_+; X))$. The previous step insures this boundedness for $p=q$. On the other hand,  Lemma~\ref{Hormander} shows that $K$ is a regular Calder\'on-Zygmund kernel. Thus  $K$ satisfies the assumption of \cite[Theorem~V.3.4]{GRF}. Note that \cite[Theorem~V.3.4]{GRF} is formulated for kernels satisfying the regularities in Lemma~\ref{Hormander} with $\d=1$; however, it is well known that \cite[Theorem~V.3.4]{GRF} remains valid for any kernel as in Lemma~\ref{Hormander} with the same proof. Therefore, $K$ is of weak type $(1,1)$, so by the  vector-valued Marcinkiewicz interpolation theorem (see \cite[Theorem~1.3.1]{bl} and its proof),  $K$ is  bounded from $L_p(\real^d; X)$ to $L_p(\real^d; L_q(\real_+;X))$ with norm controlled by $C_{d,\e,\d}\,p'\mathsf{M}_{\cc,q}(X)$ for $1<p<q$. This finishes the proof of Theorem~\ref{fML} (i).

 \medskip\n{\it Step~5: Proof of Theorem~\ref{fML} (ii)}. In this last step we show part (ii) by duality. Let $\p\in\H_{\e, \d}$ such that \eqref{reproduce} holds. Let $f\in L_p(\real^d; X)$ and $g\in L_{p'}(\real^d; X^*)$. Then \eqref{reproduce} implies
 $$\int_{\real^d}\la f(x), \, g(x)\ra\,dx=\int_{\real^{d+1}_+}\la \f_t*f(x), \, \p_t*g(x)\ra\,\frac{dxdt}t.$$
In the scalar case, this Calder\'on reproducing formula is proved by taking Fourier transforms of  both sides. Then by linearity, the formula extends to the vector-valued case too when both $f$ and $g$ take values in finite dimensional subspaces, which can be assumed by approximation. Therefore,
 $$\Big|\int_{\real^d}\la f(x), \, g(x)\ra\,dx\Big|
 \le \big\|G_{q, \f}(f)\big\|_{L_p(\real^d)}\,  \big\|G_{q', \p}(g)\big\|_{L_{p'}(\real^d)}.$$
Since we are in part (ii) of Theorem~\ref{fML}, $X$ is of martingale type $q$, so $X^*$ is of martingale cotype $q'$ and $\mathsf{M}_{c, q'}(X^*)\le 2\mathsf{M}_{t, q}(X)$ by \eqref{duality type-cotype}. Thus by part (i) already proved, we have
  \begin{align*}
  \big\|G_{q', \p}(g)\big\|_{L_{p'}(\real^d)}
  &\le\mathsf{L}^\p_{\cc, q', p'}(X^*)\, \big\|g\big\|_{L_{p'}(\real^d; X^*)}\\
 &\les_{d, \e,\d}\max\big((p')^{\frac1{q'}}, \, p\big)\mathsf{M}_{\cc, q'}(X^*)\, \big\|g\big\|_{L_{p'}(\real^d; X^*)}.
   \end{align*}
Hence
 $$\Big|\int_{\real^d}\la f(x), \, g(x)\ra\,dx\Big|\les_{d, \e,\d}\max\big(p,\,(p')^{\frac1{q'}}\big) \mathsf{M}_{\tt, q}(X)\, \big\|G_{q, \f}(f)\big\|_{L_p(\real^d)}\,\big\|g\big\|_{L_{p'}(\real^d; X^*)}.$$
Taking the supremum over all $g$ with $\big\|g\big\|_{L_{p'}(\real^d; X^*)}\le1$, we deduce
 $$\mathsf{L}^\f_{\tt, q, p}(X)\les_{d, \e,\d}\max\big(p,\,(p')^{\frac1{q'}}\big) \mathsf{M}_{\tt, q}(X)$$
 as desired. So the proof of Theorem~\ref{fML} is complete.
 
 \begin{proof}[Proof of Corollary~\ref{NY}] The first part of this corollary immediately follows from Theorem~\ref{fML}. The first two inequalities of the second part are consequences of Remark~\ref{Poisson vs Heat} and Proposition~\ref{Optimality} below. 
Finally, the last inequality $\mathsf{L}^{\mathbb{P}}_{\cc,q,q}(X)\ges\mathsf{M}_{\cc, q}(X)$ is obtained by combining \cite{LP1} and \cite{LP0}.

 \end{proof}


\section{Proofs of Theorem~\ref{Poisson ML}, Theorem~\ref{Heat ML} and Corollary~\ref{Poisson MLbis}}\label{Proof of Theorem Poisson ML}


In this section we will first prove Theorem~\ref{Poisson ML}. Theorem~\ref{Heat ML}  and Corollary~\ref{Poisson MLbis} will then follow quite easily. Our strategy for the proof of Theorem~\ref{Poisson ML} is to reduce part (i), via transference, to the special case of the translation group to which we can apply Theorem~\ref{fML}.

\begin{proof}[Proof of Theorem~\ref{Poisson ML}] Again, we divide this proof into several steps. $X$ will be assumed to be of martingale cotype $q$ in the first two steps, and of martingale type $q$ in the last step.

\medskip\n{\it Step~1: The case of the translation group.} For any $t\in\real$ let $\tau_t$ be the translation by $t$ on $L_p(\real)$, i.e., $\tau_t(f)(s)=f(s+t)$. Then $\{\tau_t\}_{t\in\real}$ is a strongly continuous group of positive isometries on $L_p(\real)$. As usual,  $\{\tau_t\}_{t\in\real}$  extends to a group of isometries on $L_p(\real; X)$ too.  Let $\{P_t^\tau\}_{t>0}$ be the associated  Poisson subordinated semigroup. Our aim in this step is to show
 \beq\label{Poisson translation}
 \mathsf{L}^{P^\tau}_{\cc,q, p}(X)\les\max\big(p^{\frac1q},\,p'\big)\mathsf{M}_{\cc,q}(X).
\eeq

We need to  express $t\partial P^\tau_t$ as a convolution operator:
 $$\sqrt{t}\,\partial P^\tau_{\sqrt{t}}(f)(x)=\int_\real \phi_t(x-y)f(y)dy=\phi_t*f(x),\quad x\in\real, \; t>0.$$
 Then
 \beq\label{g-function via f-function}
 \mathcal{G}^{P^\tau}_q(f)(x)=2^{-\frac1q}\Big(\int_0^\8\|\phi_t*f(x)\|^q_X\frac{dt}t\Big)^{\frac1q}=2^{-\frac1q}G_{q,\phi}(f)(x)\,.
 \eeq
 Elementary computations show that $\phi$ is the function with Fourier transform
  \beq\label{f-kernel}
 \wh \phi(\xi)= -\sqrt{-2\pi{\rm i}\,\xi}\,e^{-\sqrt{-2\pi{\rm i}\,\xi}\,}=-\sqrt{2\pi|\xi|}\, e^{-\frac{{\rm i} \,{\rm sgn}(\xi)\pi}4}\exp\big(-\sqrt{2\pi|\xi|}\, e^{-\frac{{\rm i}\, {\rm sgn}(\xi)\pi}4}\big)\,.
 \eeq

We are going to show that $\phi$ belongs to the class $\H_{\frac12,1}$ introduced in subsection~\ref{Singular integrals}. More precisely, $\phi$ satisfies the following estimates:
 \beq\label{f-kernel estimate}
 |\phi(x)|\les\frac1{(1+|x|)^{\frac32}}\;\text{ and }\; |\phi'(x)|\les\frac1{(1+|x|)^{\frac52}} ,\quad x\in\real.
 \eeq

Since $\xi^k\wh\phi(\xi)$ is integrable on $\real$ for any nonnegative integer $k$, $\phi$ is of class $C^\8$ with bounded derivatives of any order. Thus it suffices to prove the estimates for $|x|\ge 1$.

Let $\eta$ be a $C^\8$ even function on $\real$, supported in $\{\xi: \frac12<|\xi|<2\}$, such that
  $$\sum_{j\in\ent}\eta(2^{-j}\xi)=1,\quad \xi\in\real\setminus\{0\}.$$
 Let
  $m_j(\xi)=\wh \phi(\xi)\eta(2^{-j}\xi)$
 and $\phi^{(j)}$ be defined by $\wh{\phi^{(j)}}=m_j$.
Then
 $$
 \phi=\sum_{j\in\ent} \phi^{(j)}.
 $$
 Using \eqref{f-kernel}, one easily shows
  $$\int_{2^{j-1}\le |\xi|\le2^{j+1}}\Big|\frac{d^k}{d\xi^k}\,\wh\phi(\xi)\Big|d\xi\les 2^{j(1-k)}\sum_{\el=0}^k(\sqrt{2^j}\,)^{\el+1}e^{-\sqrt{\pi2^{j-1}}}\,,\quad 0\le k\le3,$$
 This implies
  \beq\label{estimate1}
  \int_{\real}\Big|\frac{d^k}{d\xi^k}\, m_{j}(\xi)\Big|d\xi\les 2^{j(1-k)}\sum_{\el=0}^k(\sqrt{2^j}\,)^{\el+1}e^{-\sqrt{\pi2^{j-1}}}\,,\quad 0\le k\le3.
  \eeq
Let $x\in\real$ with $|x|>1$. We  consider $j$ according to two cases,  I:  $2^j|x|\le1$ and II:  $2^j|x|>1$.

In Case~I, we must have $j\le-1$. Using \eqref{estimate1} for $k=1$, we then have
 $$|x \phi^{(j)}(x)|\les \sqrt{2^{j}} \,.$$
Thus
   $$ \sum_{j\in {\rm I}}|\phi^{(j)}(x)|\les \frac{1}{|x|^{\frac32}}\,.$$
 On the other hand, if $j\in {\rm II}$, we use \eqref{estimate1} for $k=3$ to get
 $$|x^3 \phi^{(j)}(x)|\les 2^{-\frac32j} \,.$$
We deduce again
   $$ \sum_{j\in {\rm II}}|\phi^{(j)}(x)|\les \frac{1}{|x|^{\frac32}}\,.$$
Hence, the first estimate of \eqref{f-kernel estimate} is proved.

The second is shown in a similar way. Indeed, since $\wh{\phi'}(\xi)=2\pi\i\,\xi\wh\phi(\xi)$, we have
$$\int_{2^{j-1}\le |\xi|\le2^{j+1}}\Big|\frac{d^k}{d\xi^k}\,\wh{\phi'}(\xi)\Big|d\xi\les 2^{j(2-k)}\sum_{\el=0}^k(\sqrt{2^j}\,)^{\el+1}e^{-\sqrt{\pi2^{j-1}}}\,,\quad 0\le k\le3,$$
It then remains to repeat the above argument with $\phi$ replaced by $\phi'$.

 \medskip

 Thus  by \eqref{g-function via f-function}, Theorem~\ref{fML} implies \eqref{Poisson translation}. Let us note that for the kernel $\phi$ here, we can avoid Wilson's intrinsic square functions considered in step~2 of the proof of Theorem~\ref{fML} since $\phi'\in\H_{\frac12, 1}$ too. Indeed, repeating the proof of  \eqref{f-kernel estimate}, we show
  $$ |\phi''(x)|\les\frac1{(1+|x|)^{\frac72}}.$$
 Consequently, as pointed out in   step~2 of the proof of Theorem~\ref{fML}, we have
 $$G_{q,\phi}(f)(x)\les S_{q,\phi}(f)(x) +S_{q,\phi'}(f)(x),\quad x\in\real.$$
The proof of this inequality is very easy. It suffices to consider $x=0$. Let $y\in\real$ such that $|y|<t$. We write
 $$\|\phi_t*f(y)\|^q_X- \|\phi_t*f(0)\|^q_X=\int_0^1\frac{d}{ds}\,\|\phi_t*f(sy)\|^q_Xds.$$
Then
\begin{align*}
\Big|\|\phi_t*f(y)\|^q_X- \|\phi_t*f(0)\|^q_X\Big|
&\le q\int_0^1\|\phi_t*f(sy)\|^{q-1}_X\,\|\phi'_t*f(sy)\|_X\,\frac{|y|}{t}\,ds\\
&\le q\int_0^{|y|}\|\phi_t*f(s\frac{y}{|y|})\|^{q-1}_X\,\|\phi'_t*f(s\frac{y}{|y|})\|_X\frac{ds}t\\
&\le q\int_0^{|y|}\big(\|\phi_t*f(s\frac{y}{|y|})\|^{q}_X+\|\phi'_t*f(s\frac{y}{|y|})\|^q_X\big)\frac{ds}t
\end{align*}
Thus
 \begin{align*}
 \|\phi_t*f(0)\|^q_X
 \le \|\phi_t*f(y)\|^q_X+q\int_0^{|y|}\big(\|\phi_t*f(s \frac{y}{|y|})\|^{q}_X+\|\phi'_t*f(s \frac{y}{|y|})\|^q_X\big) \frac{ds}t.
 \end{align*}
Integrating both sides against $\frac{dydt}{t^2}$ over the cone $\{(y,t)\in\real^2_+: |y|<t\}$, we deduce the desired inequality.

\medskip\n{\it Step~2: Transference.} We now use the transference principle to bring the general case to the special one of the translation group. To that end, we first need to dilate our semigroup $\{T_t\}_{t>0}$ to  a group of isometries. Fendler's dilation theorem  is at our disposal for this purpose. It insures that  there exist another larger measure space $(\wt\O, \wt\mu)$, a strongly continuous group $\{\wt T_t\}_{t\in\real}$ of regular isometries on $L_p(\wt\O)$, a positive isometric embedding $D$ from $L_p(\O)$ into $L_p(\wt\O)$ and a regular projection $P$ from $L_p(\wt\O)$ onto $L_p(\O)$ such that
  $$
   T_t=P\wt T_tD,\quad\forall\; t>0.
 $$
This theorem is proved in \cite{Fe} for positive $T_t$ and then extended to regular $T_t$ in \cite{Fe2}.

To prove part (i) of Theorem~\ref{Poisson ML}, it suffices to show
  \beq\label{translation to Poisson}
  \mathsf{L}^{P}_{\cc,q,p}(X) \le  \mathsf{L}^{P^\tau}_{\cc,q,p}(X).
  \eeq
 By the above dilation, we can assume that $\{T_t\}_{t>0}$ itself is a group of regular isometries on $L_p(\O)$. So its extension to $L_p(\O;X)$ is a group of  isometries too.  Recall that $\{M_t\}_{t>0} $  denote the ergodic averages of $\{T_t\}_{t>0}$  in Lemma~\ref{lq-bdd ergodic}. We use $\{M^\tau_t\}_{t>0} $  to denote the corresponding averages of the translation group $\{\tau_t\}_{t>0}$. By \eqref{subordination}, we have
  $$P_t=\frac1{2\sqrt{\pi}}\,\int_0^\8\frac{t}{s^{\frac32}}\exp\big(-\frac{t^2}{4s}\big)\,T_sds.$$
 Thus
   \begin{align*}
   t\partial P_t
   &=\frac1{2\sqrt{\pi}}\,\int_0^\8\big(\frac{t}{s^{\frac32}}-\frac{t^3}{2s^{\frac52}}\big) \exp\big(-\frac{t^2}{4s}\big)\,T_sds\\
   &=\frac1{2\sqrt{\pi}}\,\int_0^\8\big(\frac{t}{s^{\frac32}}-\frac{t^3}{2s^{\frac52}}\big) \exp\big(-\frac{t^2}{4s}\big)\,(sM_s)'ds\\
   &=\int_0^\8\f(\frac{t}{\sqrt s})M_s\,\frac{ds}s\\
   &=\int_0^\8\f(\frac{1}{\sqrt s})M_{t^2s}\,\frac{ds}s,
   \end{align*}
  where
   $$\f(x)=\frac1{16\sqrt{\pi}}\,(12x-12 x^3+x^5)e^{-\frac{x^2}4}.$$
 Let $f\in L_p(\O;X)$ be an element of norm $1$. Let $a>0$ (large). Then
  $$\int_a^{\8}|\f(\frac{1}{\sqrt s})|\,\frac{ds}s\le \frac{C}{\sqrt a}\,.$$
Thus for any $t>0$
 \begin{align*}
 \Big\|\int_a^{\8}\f(\frac{1}{\sqrt s})M_{t^2s}(f)\,\frac{ds}s\Big\|_X
 \le\int_a^{\8}|\f(\frac{1}{\sqrt s})|\big\|M_{t^2s}(f)\big\|_X\,\frac{ds}s
 \le\frac{C}{\sqrt a}\, M^*(f),
   \end{align*}
where $M^*(f)=\sup_{v>0}\big\|M_{v}(f)\big\|_X$. By Lemma~\ref{lq-bdd ergodic},
  $$\big\|M^*(f)\big\|_{L_p(\O)}\le p'\big\|f\big\|_{L_p(\O; X)}= p'.$$
Let $b$ be another large number. Then
  \begin{align*}
  \left\|\Big(\int_{b^{-1}}^b \big\|t\partial P_t(f)\big\|_X^q\,\frac{dt}t\Big)^{\frac1q}\right\|_{L_p(\O)}
  \le  \left\|\Big(\int_{b^{-1}}^b \Big\|\int_{0}^a\f(\frac{1}{\sqrt s})M_{t^2s}(f)\,\frac{ds}s\Big\|_X^q\,\frac{dt}t\Big)^{\frac1q}\right\|_{L_p(\O)}+\frac{C_{p, q, b}}{\sqrt a},
 \end{align*}
 where $C_{p, q, b}=Cp'q^{-\frac1q}(2\log b)^{\frac1q}$. Denote the first term on the right hand side by I. Using the fact that $\{T_t\}$ is a group of isometries on $L_p(\O; X)$, we introduce an additional variable $u$ in the integrand of I:
 $$ {\rm I}
 \le  \left\|\Big(\int_{b^{-1}}^b \Big\|\int_{0}^a\f(\frac{1}{\sqrt s})M_{t^2s}T_u(f)\,\frac{ds}s\Big\|_X^q\,\frac{dt}t\Big)^{\frac1q}\right\|_{L_p(\O)}\,,\quad u>0.$$
 Let $c>0$. Now define $g: \real\to L_p(\O;X)$ by $g(u)=\un_{(0,\,ab^2+c]}(u)T_u(f)$. We easily verify that
   $$M_{t^2s}T_u(f)=M^\tau_{t^2s}(g)(u)\,, \quad 0<s\le a,\; 0<t\le b,\;0< u\le c.$$
Hence
\begin{align*}
{\rm I}^p
& \le\frac1c\int_0^c\int_\O\Big(\int_{b^{-1}}^b \Big\|\int_{0}^a\f(\frac{1}{\sqrt s})M^\tau_{t^2s}(g)(u)\,\frac{ds}s\Big\|_X^q\,\frac{dt}t\Big)^{\frac pq}d\o\,du\\
& \le\frac1c\left\|\Big(\int_{b^{-1}}^b\Big\|\int_{0}^a\f(\frac{1}{\sqrt s})M^\tau_{t^2s}(g)\,\frac{ds}s\Big\|_X^q\,\frac{dt}t\Big)^{\frac1q}\right\|^p_{L_p(\real\times\O)}\,.
 \end{align*}
 Let  $(M^\tau)^*(g)=\sup_{v>0}\big\|M^\tau_{v}(g)\big\|_X$, so
  $$\big\|(M^\tau)^*(g)\big\|_{L_p(\real\times\O)}\le p' \big\|g\big\|_{L_p(\real\times\O; X)}\le p'(ab^2+c)^{\frac1p}.$$
Reversing the preceding procedure with  $\{P_t\}_{t>0}$ replaced by  $\{P^\tau_t\}_{t>0}$, we have
 \begin{align*}
{\rm I}
&\le c^{-\frac1p}\left\|\Big(\int_{b^{-1}}^b\Big\|\int_{0}^\8\f(\frac{1}{\sqrt s})M^\tau_{t^2s}(g)\,\frac{ds}s\Big\|_X^q\,\frac{dt}t\Big)^{\frac1q}\right\|_{L_p(\real\times\O)}
+\frac{C_{p, q, b}}{\sqrt a}\,\Big(\frac{ab^2+c}{c}\Big)^{\frac1p}\\
&\le c^{-\frac1p}\big\|\mathcal{G}^{P^\tau}_q(g)\big\|_{L_p(\real\times\O; X)} +\frac{C_{p, q, b}}{\sqrt a}\,\Big(\frac{ab^2+c}{c}\Big)^{\frac1p}\,.
\end{align*}
However,
 $$\big\|\mathcal{G}^{P^\tau}_q(g)\big\|_{L_p(\real\times\O)}
 \le  \mathsf{L}^{P^\tau}_{\cc,q, p}(X)  \big\|g\big\|_{L_p(\real\times\O;X)}
 \le \mathsf{L}^{P^\tau}_{\cc,q,p}(X)  (ab^2+c)^{\frac1p}\,.$$
Combining all inequalities obtained so far, we finally deduce
  $$
  \left\|\Big(\int_{b^{-1}}^b \big\|t\partial P_t(f)\big\|_X^q\,\frac{dt}t\Big)^{\frac1q}\right\|_{L_p(\O)}
  \le \mathsf{L}^{P^\tau}_{\cc,q,p}(X)  \Big(\frac{ab^2+c}{c}\Big)^{\frac1p} +\frac{C_{p, q, b}}{\sqrt a}\left(1+\Big(\frac{ab^2+c}{c}\Big)^{\frac1p}\right)\,.$$
Letting successively $c\to\8$,  $a\to\8$ and  $b\to\8$, we get
 $$
  \big\|\mathcal{G}^{P}_q(f)\big\|_{L_p(\O)} \le  \mathsf{L}^{P^\tau}_{\cc,q,p}(X)\,,$$
 whence \eqref{translation to Poisson}.

 \medskip\n{\it Step~3: Duality.} Assertion (ii) follows from  (i), Theorem~\ref{dual} and \eqref{duality type-cotype}.
 \end{proof}

\begin{rk}
 The step~1 of the previous proof of Theorem~\ref{Poisson ML} can be largely shortened for the case $p\ge q$. This alternate proof does not rely on the heavy Littlewood-Paley theory. Its key point is to show the boundedness of $\mathcal{G}^P_q$ on $L_q(\O;X)$, i.e., for $p=q$ (see the following remark). Assuming this boundedness and  showing that  $K$ is  bounded from $L_\8(\real^d; X)$ to $BMO(\real^d; L_q(\real_+;X))$ (the latter boundedness is quite easy to get), we can then use the singular integral as in the proof of the step~4 of Theorem~\ref{fML} to conclude the case $p>q$. Unfortunately, this proof yields $p$ as the order of $\mathsf{L}^{P}_{\cc,q,p}(X)$ instead of the optimal $p^{\frac1q}$.
  \end{rk}

\begin{rk}The boundedness of $\mathcal{G}^P_q$ on $L_q(\O;X)$ can be proved by using Theorem~ \ref{LPS-diff} and the following inequality  from \cite{HM}:  For a Banach space $X$ of martingale cotype $q$, we have
  $$
 \Big\|\Big(\sum_{k\in\ent} \big\|\big(M^\tau_{2^{k}} - M^\tau_{2^{k+1}}\big)(f)\big\|_X^q \Big)^{\frac1q}\Big\|_{L_q(\real)}
 \les \mathsf{M}_{\cc,q}(X)\big\|f\big\|_{L_q(\real;X)}\,, \quad f\in L_q(\real;X).
 $$
By the discussion in the previous remark, the validity of the above inequality characterizes the martingale cotype $q$ of $X$. More generally, let $1<p<\8$ and $2\le q<\8$. Then $X$ is of martingale cotype $q$ iff there exists a constant $c$ such that
  $$
 \Big\|\Big(\sum_{k\in\ent} \big\|\big(M^\tau_{2^{k}} - M^\tau_{2^{k+1}}\big)(f)\big\|_X^q \Big)^{\frac1q}\Big\|_{L_p(\real)}
 \le c\, \big\|f\big\|_{L_p(\real;X)}\,, \quad f\in L_p(\real;X).
 $$
See \cite{HLM} for related results.
\end{rk}

\begin{proof}[Proof of Theorem~\ref{Heat ML}]
 By Remark~\ref{Poisson vs Heat},  $\mathsf{L}^{T}_{\tt,q,p}(X)\les \mathsf{L}^{P}_{\tt,q,p}(X)$. Thus assertion (i) follows from Theorem~\ref{Poisson ML} (ii).

Assertion (ii)  is an easy consequence of Theorem~\ref{Poisson ML} and Corollary~\ref{heat LPS-f}. But we will use Proposition~\ref{mac} in order to explicitly track the relevant constants. Let $\f(z)=-ze^{-z}$ and $\psi(z)=-\sqrt z\,e^{-\sqrt z}$. Then
  $$\mathcal{G}_{q, \f}^A(f)=\mathcal{G}_q^T(f)\quad\text {and }\quad\mathcal{G}_{q, \p}^A(f)=\sqrt2\,\mathcal{G}_q^P(f).$$
By Proposition~\ref{T-analyticity} and the notation there, $A$ is $\el_q$-sectorial of type $\a_q=\frac\pi2-\b_q$. Let $0<\b<\b_q$ and $\a=\frac\pi2-\b$. Then by Proposition~\ref{mac} and the estimates obtained in its proof, we get
 $$\big\|\mathcal{G}_q^T(f)\big\|_{L_p(\O)}
  \les (\cos\a\cos\frac\a2)^{-1}(\b_q-\b)^{-2}\,\mathsf{T}_{\b_0}^{\min(\frac{p}{q},\,\frac{p'}{q'})}\max((p')^{1-\frac{p}{q}},\, p^{1-\frac{p'}{q'}})\big\|\mathcal{G}_q^P(f)\big\|_{L_p(\O)}\,.$$
 Choosing  $\b=\frac{\b_q}2$ yields
 $$\big\|\mathcal{G}_q^T(f)\big\|_{L_p(\O)} \les D\, \big\|\mathcal{G}_q^P(f)\big\|_{L_p(\O)},$$
where
 $$D=\b_q^{-3}\, \mathsf{T}_{\b_0}^{\min(\frac{p}{q}\,,\frac{p'}{q'})}\,\max\big((p')^{1-\frac{p}{q}},\,p^{1-\frac{p'}{q'}}\big).$$
We then deduce  Theorem~\ref{Heat ML} (ii) from Theorem~\ref{Poisson ML} (i).
 \end{proof}

\begin{proof}[Proof of Corollary~\ref{Poisson MLbis}]
 By \cite{Na12, Na14},
  $$\mathsf{M}_{\cc,q}(L_p(\O;X))\les  \max\big((p')^{\frac1q},\, \mathsf{M}_{\cc,q}(X)\big).$$
Fix $f_0\in L_q(\O)$ with norm 1. Given $f\in L_p(\O;X)$ let $\wt f=f\ot f_0$. We view $\wt f$ as a function from $\O$ to $L_p(\O;X)$ by $\o\mapsto f(\o)f_0$. Applying Theorem~\ref{Poisson ML} (i) to this function with $X$ replaced by $L_p(\O;X)$ and $p=q$ (noting then that
 $\max\big(q^{\frac1q},\, q'\big)\approx1$ since $q\ge2$), we get
 \begin{align*}
 \left\|\Big(\int_0^\8 \big\|t\frac{\partial}{\partial t} P_t(f)\ot f_0\big\|^q_{L_p(\O;X)}\frac{dt}t\Big)^{\frac1q}\right\|_{L_q(\O)}
 &\les  \max\big((p')^{\frac1q},\, \mathsf{M}_{\cc,q}(X)\big)\big\|f\ot f_0\big\|_{L_q(\O;L_p(\O;X))}\\
  &=  \max\big((p')^{\frac1q},\, \mathsf{M}_{\cc,q}(X)\big)\big\|f\big\|_{L_p(\O;X)}.
 \end{align*}
The left hand side is equal to
 \begin{align*}
 \Big(\int_0^\8 \big\|t\frac{\partial}{\partial t} P_t(f)\ot f_0\big\|^q_{L_q(\O;L_p(\O;X))}\frac{dt}t\Big)^{\frac1q}
 = \Big(\int_0^\8 \big\|t\frac{\partial}{\partial t} P_t(f)\big\|^q_{L_p(\O;X)}\frac{dt}t\Big)^{\frac1q}.
  \end{align*}
Combining the above estimates we get the desired inequality of the corollary.

To show the optimality of the constant, we consider the special case where $X=\com$ and the classical Poisson semigroup $\{\mathbb{P}_t\}_{t>0}$ on $\real$. Let $f=\mathbb{P}_1$. Then
 $$
  \big\|f\big\|_{L_p(\real)}\approx 1\;\text {and }\;
 \big\|t\frac{\partial}{\partial t} \mathbb{P}_t(f)\big\|_{L_p(\real)}\approx \frac{t}{(t+1)^{1+\frac1{p'}}}.
 $$
Thus
 $$
 \Big(\int_0^\8\big\|t\frac{\partial}{\partial t}\mathbb{P}_t(f)\big\|^q_{L_p(\real)}\,\frac{dt}t\big)^{\frac1q}
 \approx p'^{\frac1q}.
 $$
 We then deduce the announced optimality.
\end{proof}


\section{The scalar case revisited and optimality}\label{The scalar case revisited}


The approach previously presented gives new insights even in the scalar case with regard to the involved best constants. Let
 $$\mathsf{L}^T_{\cc ,p}=\mathsf{L}_{\cc, 2, p}(\com)\;\text{ and }\; \mathsf{L}^T_{\tt, p}=\mathsf{L}_{\tt, 2, p}(\com).$$
Thus $\mathsf{L}^T_{\cc, p}$ and $\mathsf{L}^T_{\tt, p}$ are the best constants in the following inequalities
 $$
   (\mathsf{L}^T_{\tt, p})^{-1}\|f-\mathsf F(f)\|_{L_p(\O)}\le \big\|\mathcal{G}^{T}_2(f)\big\|_{L_p(\O)}\le  \mathsf{L}^T_{\cc, p}\|f-\mathsf F(f)\|_{L_p(\O)},\quad f\in L_p(\O).
  $$

Let us restate Theorems~ \ref{Poisson ML} and \ref{Heat ML} for $X=\com$ and $q=2$.

 \begin{thm}\label{scalar LPS}
 Let $\{T_t\}_{t>0}$ be a semigroup of regular contractions on $L_p(\O)$ with $1<p<\8$ and $\{P_t\}_{t>0}$ its subordinated Poisson semigroup.
  \begin{enumerate}[\rm(i)]
  \item We have
   $$ \mathsf{L}^{P}_{\cc,p}\les\max(\sqrt{p},\, p') \quad\text{ and }\quad  \mathsf{L}^{P}_{\tt,p}\les\max(p,\, \sqrt{p'}).$$
  \item Assume that $\{T_t\}_{t>0}$ satisfies \eqref{Ana bound} for $X=\com$.  Let $\b_p=\b_0\min(p\,,p')$.
   Then
  $$\mathsf{L}^{T}_{\cc,p}\les  \b_p^{-3}\,\mathsf{T}_{\b_0}^{\frac12\,\min(p,\, p')}\max\big(p,\, (p')^{\,\frac32}\big) \quad\text{ and }\quad
   \mathsf{L}^{T}_{\tt,p}\les\max(p,\, \sqrt{p'}).$$
   \end{enumerate}
  \end{thm}

\medskip

For symmetric diffusion semigroups we have the following more precise orders  than those in part (ii) above. We are very grateful to the anonymous referee for pointing out the references \cite{Kriegler,Lis} that allow us to improve our previous estimate on $\mathsf{L}^{T}_{\cc,p}$ based on Stein's classical analyticity angle of $\{T_t\}_{t>0}$ on $L_p(\O)$.

\begin{cor}\label{scalar heat LPS}
 Let $\{T_t\}_{t>0}$ be a semigroup of contractions on $L_p(\O)$ for \underline{every} $1\le p\le\8$. Assume that $\{T_t\}_{t>0}$ is strongly continuous on $L_2(\O)$ and each $T_t$ is a selfadjoint operator on $L_2(\O)$. Then
  $$\mathsf{L}^{T}_{\cc,p}\les\max\big(p^{\frac52},\,  (p')^{\,3}\big) \quad\text{ and }\quad  \mathsf{L}^{T}_{\tt,p}\les\max(p,\, \sqrt{p'}).$$
  \end{cor}

\begin{proof}
First note that $T_t$ is a regular contraction on $L_p(\O)$ for any $1\le p\le\8$ and any $t>0$.  The selfadjointness of $T_t$ implies that $\{T_t\}_{t>0}$ is an analytic semigroup of type $\frac\pi2$ with constant $1$ on $L_2(\O)$. Then by \cite[Corollary~3.2]{Lis} and \cite[Corollary~6.2]{Kriegler}, $\{T_t\}_{t>0}$ is analytic of type $\b_p'=\frac\pi2-\arcsin|1-\frac2p|$ with constant $1$  on $L_p(\O)$ for $1<p<\8$. Note that the angle $\b_p'$ is optimal and better than Stein's classical one  that is $\frac\pi2\big(1- |1-\frac2p|\big)$ (see \cite[section~III.2]{stein}).  It remains to apply Theorem~\ref{scalar LPS} (ii) with $\b_0=\b_p' \approx\min\big(\sqrt{\frac{p}{p'}},\, \sqrt{\frac{p'}{p}}\,\big)$ and $\mathsf{T}_{\b_0}=1$. The corresponding $\b_p$ is now equivalent to $\min\big(\frac{1}{\sqrt p},\, \frac{1}{\sqrt{p'}}\big)$. We then deduce  the desired assertion from Theorem~\ref{scalar LPS}.
  \end{proof}

\medskip\n{\bf Historical comments.}
 (i) Theorem~\ref{scalar LPS} was proved in \cite{LMX} without any explicit estimates on the best constants; in fact, their growth obtained there  is more than exponential.

 (ii) If  $\{P_t\}_{t>0}$ is the Poisson semigroup on a compact Lie group,  Stein's proof in \cite[section~II.3]{stein} yields that  $\mathsf{L}^{P}_{\cc,p}\les\max(p,\, p')$ and $\mathsf{L}^{P}_{\tt,p}\les\max(p,\, p')$.

 (iii) If  $\{T_t\}_{t>0}$ is a symmetric diffusion semigroup, Stein's approach in \cite[section~IV.4]{stein} via Rota's dilation  yields $\mathsf{L}^{P}_{\cc,p}\les\max(p,\, (p')^{\frac32})$.

 (iv)  If  $\{T_t\}_{t>0}$ is a symmetric submarkovian semigroup, Cowling \cite{Cow} proved that the negative generator $A$ of $\{T_t\}_{t>0}$ has a bounded holomorphic functional calculus whose relevant constant is of polynomial growth on $p$ as $p\to1$ and $p\to\8$. Using the equivalence between bounded holomorphic functional calculus and square function inequalities, one can then deduce a polynomial growth of $\mathsf{L}^{T}_{\cc,p}$ and $\mathsf{L}^{T}_{\tt,p}$ too, the resulting orders are less good than those in Corollary~\ref{scalar heat LPS}.

\begin{rk}\label{optimal order}
 The orders on  $\mathsf{L}^{P}_{\cc,p}$ in  Theorem~\ref{scalar LPS} are optimal both as $p\to\8$ and $p\to1$ for they are already optimal for the classical Poisson semigroup on $\real$ (see  Proposition~\ref{Optimality} below). Zhendong Xu and Hao Zhang \cite{XZ}  proved that  $\mathsf{L}^{T}_{\tt,p}\ges p$ as $p\to\8$ when $\{T_t\}_{t>0}$ is a symmetric diffusion semigroup, so $\mathsf{L}^{P}_{t,p}\ges p$ as $p\to\8$ for the subordinated Poisson semigroup $\{P_t\}_{t>0}$ too. This shows that our method is optimal.
 \end{rk}

 However, at the time of this writing, we are unable to determine the optimal orders of $\mathsf{L}^{P}_{\tt,p}$ as $p\to1$ even when $\{T_t\}_{t>0}$ is a symmetric diffusion semigroup.

\begin{problem}
It would be interesting to determine the optimal orders of $\mathsf{L}^{P}_{\tt,p}$ as $p\to1$  when $\{T_t\}_{t>0}$ is a symmetric submarkovian (or markovian) semigroup. In particular, does there exist a constant $C$ (possibly depending on $\{T_t\}_{t>0}$) such that
 $$\|f-\mathsf F(f)\|_{L_1(\O)}\le C\big\|\mathcal{G}^{P}_2(f)\big\|_{L_1(\O)}\,,\; \forall \,f\in L_1(\O)\,?$$
\end{problem}

The dual version of the above inequality is related to the BMO space considered in \cite{F-Mei-S}. It is true when $\{P_t\}_{t>0}$ is the Poisson or heat semigroup on $\real^d$.

\medskip




We conclude this section by the proof of the optimality of the growth orders of the best constants in Corollary~\ref{NY} in the scalar case, i.e., $X=\com$ (see \cite{LP-Optimality} for more related results). We will denote $\mathsf{L}^{\mathbb{P}}_{\cc,q, p}(\com)$ and $\mathsf{L}^{\mathbb{H}}_{\cc,q, p}(\com)$ simply by $\mathsf{L}^{\mathbb{P}}_{\cc,q,p}$ and $\mathsf{L}^{\mathbb{H}}_{\cc,q, p}$, respectively. It suffices to consider $\real$.

\begin{prop}\label{Optimality}
 Let $1<p, q<\8$. Then
 $$
 \mathsf{L}^{\mathbb{P}}_{\cc,q,p}\ges \max\big(p^{\frac1q},\, p'\big)\;\text{ and }\; \mathsf{L}^{\mathbb{H}}_{\cc,q,p}\ges \max\big(p^{\frac1q},\, p'\big).
 $$
 \end{prop}

 \begin{proof}
 By Remark~\ref{Poisson vs Heat},  it suffices to show the assertion on the Poisson semigroup.

Let us first consider the case $p\le q$. Fix $s>0$ and let $f=\mathbb{P}_s$. Then
 $$t\frac{\partial}{\partial t} \mathbb{P}_t(f)(x)=\frac{t}\pi\, \frac{x^2-(t+s)^2}{(x^2+(t+s)^2)^2},\quad x\in\real.$$
For $x\ge 6s$, we have
 \begin{align*}
 \G^{ \mathbb{P}}_q(f)(x)\ge \Big(\int_{\frac{x}3-s}^{\frac{x}2-s}\big|t\frac{\partial}{\partial t} \mathbb{P}_t(f)(x)\big|^q\,\frac{dt}t\Big)^{\frac1q}\ges\frac1x.
\end{align*}
Thus
 $$\big\| \G^{ \mathbb{P}}_q(f)\big\|_{L_p(\real)}\ges  \Big(\int_{6s}^{\8}\frac1{x^p}\,dx\Big)^{\frac1p}\ges \frac{s^{-\frac1{p'}}}{p-1}.$$
On the other hand,
 $$\big\| f\big\|_{L_p(\real)}\approx s^{-\frac1{p'}}.$$
Hence, $\mathsf{L}^{\mathbb{P}}_{c,q, p}\ges p'$.

\medskip

Unfortunately, the above simple argument does not apply to the case $p>q$. Our proof for the latter is much harder. By periodization, it is equivalent to considering the torus $\T$ (equipped with normalized Haar measure). The $g$-function relative to the Poisson semigroup on $\T$ is defined by
 $$G_q^{\mathsf{P}}(f)=\Big(\int_0^1\big|(1-r)\frac{d}{dr} \mathsf{P}_r(f)\big|^q\,\frac{dr}{1-r}\Big)^{\frac1q}\,,$$
where $\mathsf{P}_r$ denotes the corresponding Poisson kernel:
 $$\mathsf{P}_r(\t)=\frac{1-r^2}{1-2r\cos\t+r^2}\,.$$
  It is shown in \cite{LP0} that the inequality
 $$
 \|G_q^{\mathsf{P}}(f)\|_{L_p(\T)}\le \mathsf{L}^{\mathsf{P}}_{\cc, q, p}\|f\|_{L_p(\T)}
 $$
is equivalent to the corresponding dyadic martingale inequality on $\O=\{-1, 1\}^\nat$. It is well known that the relevant constant in the latter martingale inequality for $q=2$ is of order $\sqrt p$ as $p\to\8$. To reduce the determination of  optimal order of $\mathsf{L}^{\mathsf{P}}_{\cc,q, p}$ to  the martingale case, we need to refine an argument in the proof of \cite[Theorem~3.1]{LP0}.

Keeping the notation there, let $M=(M_k)_{0\le k\le K}$ be a finite dyadic martingale and
 $$M_k-M_{k-1}=d_k(\e_1,\cdots, \e_{k-1})\,\e_k,$$
where $(\e_k)$ are the coordinate functions of $\O$.
The transformation $\e_k={\rm sgn}(\cos\t_k)$ establishes a measure preserving embedding of $\O$ into $\T^\nat$. Accordingly, define
 \begin{align*}
  a_k(e^{\i\t_1},\cdots, e^{\i\t_{k-1}})
  &=d_k({\rm sgn}(\cos\t_1),\cdots, {\rm sgn}(\cos\t_{k-1})),\\
  b_k(e^{\i\t_{k}})&={\rm sgn}(\cos\t_k).
  \end{align*}
Given $(n_k)$ a rapidly increasing sequence of positive integers, put
 \begin{align*}
  a_{k, (n)}(e^{\i\t})
  &=a_{k, (n)}(e^{\i\t}; e^{\i\t_1},\cdots, e^{\i\t_{k-1}})
  =a_k(e^{\i(\t_1+n_1\t)},\cdots, e^{\i(\t_{k-1}+n_{k-1}\t)}),\\
  b_{k, (n)}(e^{\i\t})
  &=b_{k, (n)}(e^{\i\t};e^{\i\t_{k}})
  =b_k(e^{\i(\t_{k}+n_k\t)}),\\
  f_{(n)}(e^{\i\t})
  &=f_{(n)}(e^{\i\t}; e^{\i\t_1},\cdots, e^{\i\t_{K}})
  =\sum_{k=1}^Ka_{k, (n)}(e^{\i\t})b_{k, (n)}(e^{\i\t}).
  \end{align*}
 The functions $f_{(n)}$, $a_{k, (n)}$ and $b_{k, (n)}$ are viewed as functions on $\T$ for each $(\t_1, \cdots, \t_K)$ arbitrarily fixed. Furthermore, by approximation, we can assume that all $a_k$ and $b_k$ are polynomials. Then, if the sequence $(n_k)$ rapidly increases, Lemmas~3.4 and 3.5 of \cite{LP0} imply
  $$\frac12\, G_q^{\mathsf{P}}(f_{(n)})\le \Big(\sum_{k=1}^K|a_{k, (n)}|^q\,G_q^{\mathsf{P}}(b_{k, (n)})^q\Big)^{\frac1q}\le 2G_q^{\mathsf{P}}(f_{(n)}).$$
 Therefore,
  \beq\label{inter}
  \Big\|\Big(\sum_{k=1}^K|a_{k, (n)}|^q\,G_q^{\mathsf{P}}(b_{k, (n)})^q\Big)^{\frac1q}\Big\|_{L_p(\T)}
  \le 2 \mathsf{L}^{\mathsf{P}}_{\cc, q,p}\big\|f\big\|_{L_p(\T)}\,.
  \eeq
The discussion so far comes from \cite{LP0}. Now we require a finer analysis of the $g$-function $G_q^{\mathsf{P}}(b_{k, (n)})$.  To this end we write the Fourier series of the function $b={\rm sgn}(\cos\t)$:
  $$b(e^{\i\t})=\frac2\pi\,\sum_{j=0}^\8\frac{(-1)^j}{2j+1}\,\big[e^{\i(2j+1)\t}+e^{-\i(2j+1)\t}\big].$$
Then
 $$\frac{d}{dr}\mathsf{P}_r(b_{k, (n)})(e^{\i\t})=\frac4\pi\,n_kr^{n_k-1}{\rm Re}\Big(\sum_{j=0}^\8(-1)^j r^{2n_kj}e^{\i(2j+1)(\t_k+n_k\t)}\Big).$$
Elementary computations show
 $$\left|\frac{d}{dr}\mathsf{P}_r(b_{k, (n)})(e^{\i\t})\right|^q\ge c^q n^q_kr^{q(n_k-1)}\cos^q(\t_k+n_k\t).$$
Here and below, $c,\, C$ denote absolute positive constants.
Thus
 \begin{align*}
 G_q^{\mathsf{P}}(b_{k, (n)})^q
 &\ge c^q\,\cos^q(\t_k+n_k\t)\,n^q_k\int_0^1(1-r)^{q-1}r^{q(n_k-1)}dr \\
 &\approx c^q\,\big[1+{\rm O}(\frac1{n_k})\big]\cos^q(\t_k+n_k\t).
 \end{align*}
Now lifting both sides of \eqref{inter} to power $p$, then integrating the resulting inequality over $\T^K$ with respect to $(\t_1,\cdots, \t_K)$, we get
 \begin{align*}
 \int_\T&\int_{\T^K}\Big(\sum_{k=1}^K|a_{k,(n)}(e^{\i(\t_1+n_1\t)},\cdots, e^{\i(\t_{k-1}+n_{k-1}\t)})|^q\big[1+{\rm O}(\frac1{n_k})\big]\cos^q(\t_k+n_k\t)\Big)^{\frac{p}q}d\t_1\cdots d\t_Kd\t\\
  &\le \big(C \mathsf{L}^{\mathsf{P}}_{\cc, q,p}\big)^p\int_\T\int_{\T^K}\big|f_{(n)}(e^{\i(\t_1+n_1\t)}, \cdots, e^{\i(\t_K+n_K\t)})\big|^pd\t_1\cdots d\t_Kd\t\,.
 \end{align*}
 For each fixed $\t$, the change of variables $(\t_1,\cdots, \t_K)\mapsto(\t_1-n_1\t,\cdots, \t_K-n_K\t)$ being a measure preserving transformation of $\T^K$, we deduce
  \begin{align*}
\int_{\T^K}&\Big(\sum_{k=1}^K|a_{k,(n)}(e^{\i\t_1},\cdots, e^{\i\t_{k-1}})|^q\,\big[1+{\rm O}(\frac1{n_k})\big]\cos^q\t_k\Big)^{\frac{p}q}\,d\t_1\cdots d\t_K\\
  &\le \big(C\mathsf{L}^{\mathsf{P}}_{\cc, q, p}\big)^p\int_{\T^K}\big|f_{(n)}(e^{\i\t_1}, \cdots, e^{\i\t_K})\big|^p\,d\t_1\cdots d\t_K\,.
 \end{align*}
Letting $n_1\to\8$, we get
  \begin{align*}
\int_{\T^K}&\Big(\sum_{k=1}^K|d_{k}({\rm sgn}(\cos\t_1),\cdots, {\rm sgn}(\cos\t_{k-1}))|^q\,\cos^q\t_k\Big)^{\frac{p}q}\,d\t_1\cdots d\t_K\\
  &\le \big(C\, \mathsf{L}^{\mathsf{P}}_{\cc, q, p}\big)^p\big\|M_K\big\|^p_{L_p(\O)}\,.
 \end{align*}

Now we consider an elementary example where $M$ is simple random walk stopped at $\pm 2$, namely
 $$d_k=\un_{\{\tau\ge k\}}\;\text{ with }\;
 \tau=\inf\big\{k: \big|\sum_{j=1}^k\e_j\big|=2\big\}.$$
 Note that the probability of the event $\{\tau=j\}$ is zero for odd $j$ and $2^{-\frac{j}2}$ for even $j$. On the other hand, recalling $\e_k={\rm sgn}(\cos\t_k)$ and letting
  $$A_j=\big\{\tau=j,\; |\cos\t_k|\ge \frac1{\sqrt 2}\,,\; 1\le k\le j\big\},$$
we easily check that the probability of $A_j$ is $8^{-\frac{j}2}$ for even $j$. Thus for $K\ge j$
 \begin{align*}
 \sum_{k=1}^K|d_{k}(\e_1,\cdots, \e_{k-1})|^q\,\cos^q\t_k
 \ge \un_{A_j}\sum_{k=1}^j\un_{\{\tau\ge k\}}\cos^q\t_k
 \ge 2^{-\frac{q}2}j\,\un_{A_j};
 \end{align*}
consequently, for $K=2J$ with $J\in\nat$
 \begin{align*}
 \int_{\T^K}&\Big(\sum_{k=1}^K|d_{k}({\rm sgn}(\cos\t_1),\cdots, {\rm sgn}(\cos\t_{k-1}))|^q\,\cos^q\t_k\Big)^{\frac{p}q}\,d\t_1\cdots d\t_K
 \ge c^q\sum_{j=1}^J j^{\frac{p}q} 8^{-j}\ge c^p p^{\frac{p}q}.
  \end{align*}
Noting that $|M_K|\le2$ and combining all the previous inequalities together, we finally obtain
 $$\mathsf{L}^{\mathsf{P}}_{\cc, q,p}\ges p^{\frac1q}.$$
This completes the proof.
 \end{proof}


\appendix
\section{Examples}\label{Examples}


There exist plenty of examples of semigroups to which the results of this article apply. Many second order differential operators in analysis generate such semigroups. In the following we will only discuss the cotype case since the type case can be dealt with by duality. Note that it is obvious that if $X$ is of Luzin cotype $q$  relative to $\{T_t\}_{t>0}$, it is so relative to the subordinated Poisson semigroup $\{P_t\}_{t>0}$ of  $\{T_t\}_{t>0}$.

\medskip

A main task in the study of the vector-valued Littlewood-Paley-Stein theory would be the following

\begin{problem} Determine the family of semigroups $\{T_t\}_{t>0}$ such that a Banach space $X$ is of Luzin type (resp. cotype) $q$ relative to $\{T_t\}_{t>0}$ or its subordinated Poisson semigroup $\{P_t\}_{t>0}$ iff $X$ is of martingale type (resp. cotype) $q$.
\end{problem}

\begin{ex}\label{Laplacian}{\bf Laplacian operators}. The classical heat semigroup on $\real^d$ is given by $\mathbb{H}_t=e^{t\Delta}$, where $\Delta$ is the Laplacian operator. It is well  known (and  easy to check) that $\{\mathbb{H}_t\}_{t>0}$ is analytic of angle $\frac{\pi}2$ on $L_p(\real^d; X)$ for any $1\le p<\8$ and any Banach space $X$; the relevant constant as in \eqref{Ana bound} with $\b_0=\frac{\pi}2$ depends only on $d$. By \cite{HN, LP2}, if $X$ is of martingale cotype $q$, then $X$ is of Luzin cotype $q$ relative to $\{\mathbb{H}_t\}_{t>0}$. Conversely, suppose that $X$ is of Luzin cotype $q$ relative to $\{\mathbb{H}_t\}_{t>0}$, then $X$ is also of Luzin cotype $q$ relative to the classical Poisson semigroup $\{\mathbb{P}_t\}_{t>0}$, thereby $X$ is  of martingale cotype $q$ by virtue of \cite{LP1}. Thus the Luzin cotype relative to the classical heat semigroup is equivalent to the martingale cotype.\end{ex}

\begin{ex}{\bf Schr\"odinger operators}. Let $\O$ be a region in $\real^d$ equipped with Lebesgue measure. Let $a(x)=(a_{ij}(x))_{1\le i,j\le d}$ be a positive matrix whose entries are locally integrable real functions on $\O$ such that
 $$\a(x)\le a(x)\le\b(x)$$
 for two positive continuous functions $\a$ and $\b$ on $\O$. We consider the following elliptic operator
 $$L(f)=-\sum_{i, j=1}^d\frac{\partial}{\partial x_i}\Big(a_{ij}\frac{\partial f}{\partial x_j}\Big).$$
Given  $V$ a nonnegative locally integrable function on $\O$, define $A(f)=L(f)+Vf$.  It is well known that $-A$ generates a symmetric submarkovian (markovian for $V=0$) semigroup $\{T_t\}_{t>0}$ on $\O$ (cf. \cite[Theorem~1.8.1]{Davies}).  In particular, $\{T_t\}_{t>0}$ is analytic on $L_p(\O)$ for any $1<p<\8$. Thus if $X$ is of martingale cotype $q$, then it is of Luzin cotype $q$ relative to the Poisson semigroup  $\{P_t\}_{t>0}$  subordinated to $\{T_t\}_{t>0}$ on $L_p(\O; X)$ for any $1<p<\8$.

Assume in addition that $L$ is uniformly elliptic, namely, the above two functions $\a$ and $\b$ are constant. Then the integral kernel $K^0_t(x, y)$ of $e^{-tL}$ satisfies the following Gaussian  upper bound (cf. \cite[Theorem~3.2.8]{Davies})
 $$K^0_t(x,y)\le C_{\d, \a}t^{-\frac{d}2}\exp\left(-\frac{|x-y|^2}{(1+\d)\b t}\right),\quad t>0,\; x, y\in\O,\; 0<\d<1.$$
By the Trotter formula
 $$e^{-tA}(f)=\lim_{n\to\8}\big(e^{-\frac{tL}n}\,e^{-\frac{tV}n}\big)^n(f),$$
we deduce that the integral kernel $K_t(x, y)$ of $e^{-tA}$ is majorized by $K^0_t(x, y)$:
 $$K_t(x, y)\le K^0_t(x, y).$$
Thus $K_t(x, y)$ satisfies the same Gaussian  upper bound as $K^0_t(x, y)$. Let $z\in\com$ with ${\rm Re} z>0$. By \cite[Theorem~3.4.8]{Davies}, the complex time heat kernel of $e^{-zA}$ satisfies
$$|K_z(x,y)|\les_{\d, \a}({\rm Re} z)^{-\frac{d}2}\exp\left(-\frac{|x-y|^2\, {\rm Re} (z^{-1})}{(1+\d)\b}\right).$$
Then we easily show that $\{T_t\}_{t>0}$ extends to an analytic semigroup of type $\frac\pi2$ on $L_p(\O; X)$ for any  $X$ and any $1<p<\8$ ($p$ can be equal to $1$ too).  It then follows that $X$ is of Luzin cotype $q$ relative to $\{T_t\}_{t>0}$ whenever $X$ is of martingale cotype $q$.\end{ex}

As in the case of Laplacian operators, it is likely that the Luzin cotype relative to the semigroups generated by Schr\"odinger operators characterizes the martingale cotype. This is indeed the case if $\O=\real^d$ with $d\ge3$, $L=-\Delta$ and the potential $V$ satisfies a reverse H\"older inequality (see \cite{AFST}). Let us formulate the general case  explicitly as a conjecture.

\begin{conjecture} Let  $\{P_t\}_{t>0}$ be  the Poisson semigroup subordinated to the heat semigroup $\{T_t\}_{t>0}$ generated by a Schr\"odinger operator as above. If a Banach space $X$ is of Luzin cotype $q$ relative to $\{P_t\}_{t>0}$, then $X$ is of martingale cotype $q$. The same is conjectured for $\{T_t\}_{t>0}$ itself when the underlying differential operator $L$ is uniformly elliptic.\end{conjecture}

\begin{ex}{\bf Laplace-Beltrami operators}. The preceding examples can be extended to the setting of Riemannian manifolds. Let $M$ be a complete $d$-dimensional Riemannian  manifold with metric $g$. Let $a(x)=(a_{ij}(x))_{1\le i,j\le d}$ be a positive matrix smoothly depending on $x\in M$. The associated second order elliptic operator $L$ is represented as
  $$L(f)=-g^{-\frac12}\sum_{i, j=1}^d\frac{\partial}{\partial x_i}\Big(g^{\frac12}a_{ij}\frac{\partial f}{\partial x_j}\Big)$$
in local coordinates. Then $T_t=e^{-tL}$ extends to a symmetric diffusion semigroup on $M$ for all $1\le p\le\8$ (cf. \cite{Davies1}). Thus our previous results apply to the associated subordinated Poisson semigroup. The most important case is the one where $-L=\Delta$ is the  Laplace-Beltrami operator. Then the  celebrated theorem of Li and Yau \cite{LY} asserts that the integral kernel of $e^{t\Delta}$ has a Gaussian upper bound under the additional assumption that the Ricci curvature be nonnegative  (see also \cite[Theorem~5.5.6]{Davies}). Thus  as in the Euclidean case, the heat semigroup $\{e^{t\Delta}\}_{t>0}$ extends to an analytic semigroup on $L_p(M; X)$ for any Banach space $X$ and $1<p<\8$.  It would be interesting to determine whether the Luzin cotype of $X$ relative to $\{e^{t\Delta}\}_{t>0}$ characterizes the martingale coptye of $X$.
\end{ex}

\begin{ex}{\bf Hermite operators}. The Hermite operator on $\real^d$ is a particular Schr\"odinger operator: $A=-\Delta+|x|^2$. The associated semigroup $\{T_t\}_{t>0}$ is a symmetric submarkovian semigroup on $\real^d$.   The integral kernel of $T_t$ is given by
 $$K_t(x, y)=\Big(\frac{2}{\pi\sinh(2t)}\Big)^{\frac{d}2}\exp\Big(-\frac14\big[ |x-y|^2\coth t +|x+y|^2\tanh t\big]\Big).$$
Using the Trotter formula, we see  that $K_t(x, y)$ is less than or equal to the heat kernel:
 $$K_t(x, y)\le \Big(\frac{1}{4\pi t}\Big)^{\frac{d}2}\, e^{-\frac{|x-y|^2}{t}}\,.$$
This Gaussian upper bound can be deduced from the above explicit formula of $K_t$. It then follows that $\{T_t\}_{t>0}$ is an analytic semigroup of type $\frac\pi2$  on $L_p(\real^d; X)$ for any Banach space $X$ and $1<p<\8$. Betancor {\it et al}  showed in \cite{betancor1} that $X$ is of Luzin cotype $q$ relative to $\{T_t\}_{t>0}$ iff $X$ is of martingale cotype $q$ (see also \cite{betancor2} for related results).\end{ex}

\begin{ex}{\bf Laguerre operators}.  For simplicity, we only consider the Laguerre semigroup on $\real_+$,  the multi-dimensional case can be treated by tensor product. In this example, $\real_+$ is equipped with Lebesgue measure, contrarily to our usual convention. Let $\a>-1$ and
  $$\A=\frac12\Big(-\frac{d^2}{dx^2} + x^2+\frac{1}{x^2}(\a^2-\frac14)\Big),\quad x>0.$$
 We have
  $$\A(\f^\a_k)=\l^\a_k\f^\a_k,\quad k\in\nat,$$
 where $\l^\a_k=2k+|\a|+1$ and 
  $$\f^\a_k(x)=\left(\frac{2\Ga(k+1)}{\Ga(k+1+\a)}\right)^{\frac12}x^{\a+\frac12}e^{-\frac{x^2}2}L_k^\a(x^2)$$
 with $L_k^\a$ the k-th polynomial of type $\a$ (see \cite[p. 76]{Leb}). $\{\f^\a_k\}_{k\in\nat}$ is an orthonormal basis in $L_2(\real_+)$. 
 
For every $f\in L_2(\real_+)$, setting
 $$c_k(f)=\int_0^\8f(x)\f^\a_k(x)dx,$$
 we consider the operator $A$ formally defined by
  $$A(f)=\sum_{k=0}^\8c_k(f)\l^\a_k\f^\a_k.$$
 Note that $A(f)=\A(f)$ if $f$ is compactly supported and smooth.
Then $-A$ generates a symmetric semigroup $\{T_t\}_{t>0}$ of positive contractions on $L_2(\real_+)$:
  $$T_t(f)=\sum_{k=0}^\8c_k(f)e^{-\l^\a_k\,t}\f^\a_k.$$
 with kernel given by
 $$K_t(x, y)=\frac1{\sinh t}\,\sqrt{xy}\,I_\a\big(\frac{xy}{\sinh t}\big)\exp\big(-\frac12(x^2+y^2)\coth t\big),$$
 where $I_\a$ is the modified Bessel function of the first kind and order $\a$:
  $$I_\a(z)=2^{-\a}z^\a\,\sum_{k=0}^\8\frac{z^{2k}}{2^{2k}\Ga(k+1)\Ga(k+\a+1)}\,.$$

  It is proved in \cite{NS} that $\{T_t\}_{t>0}$ is contractive on $L_p(\real_+)$ for all $1\le p\le\8$ iff $\a=-\frac12$ or $\a\ge\frac12$, and that  $\{T_t\}_{t>0}$ is a bounded semigroup on $L_p(\real_+)$ for all $1\le p\le\8$ if $-\frac12<\a<\frac12$. However, for $-1<\a<-\frac12$, $T_t$ is unbounded on $L_p(\real_+)$ for $p\le p_\a=\frac2{2\a+3}$ and $p\ge p_\a'$.

 On the other hand,  \cite{betancor1} shows that for $\a>-\frac12$, a Banach space $X$ is of Luzin cotype $q$ relative to $\{T_t\}_{t>0}$  iff $X$   is of martingale  cotype $q$; as a byproduct,  \cite{betancor1} also shows that for the same range of $\a$, $\{T_t\}_{t>0}$ is analytic on $L_p(\real_+; X)$ for any $X$.

In the remaining case of $-1<\a<-\frac12$, it is quite easy to show that $\{T_t\}_{t>0}$ is a bounded semigroup on $L_p(\real_+)$ for $p_\a<p<p_\a'$. Let us outline the argument for the convenience of the reader.  By dilation invariance via the change of variables $u={x}/{\sqrt{\sinh t}}$ and $v={y}/{\sqrt{\sinh t}}$, the kernel $K_t$ is brought to
  $$\f(x, y)= \sqrt{xy}\,I_\a(xy)\exp\big(-\frac12(x^2+y^2)\cosh t\big).$$
 Let $\Phi$ be the associated integral operator:
 $$\Phi(f)=\int_0^\8\f(x, y)f(y)dy.$$
To estimate the norm of $\Phi$ in $B(L_p(\real_+; X))$ for $p_\a<p<p_\a'$, we appeal to  the following estimates of the Bessel function (cf. \cite[Chapter~5]{Leb})
 $$I_\a(z)\approx\frac{z^\a}{2^\a\Ga(\a+1)}\;\text{ as }\; z\to0\;\text{ and }\; I_\a(z)\approx\frac{e^z}{\sqrt{\pi z}}\;\text{ as }\; z\to\8.$$
Accordingly,  $I_\a$ is decomposed as
 $$I_\a(xy)=I_\a(xy)\un_{xy\le1} + I_\a(xy)\un_{xy>1}\les (xy)^\a \un_{xy\le1} + (xy)^{-1/2}e^{xy} \un_{xy>1}$$
with the relevant constant depending only on $\a$. Thus
 $$\f(x, y)\les\f_1(x, y)+\f_2(x, y),$$
 where
   \begin{align*}
   &\f_1(x, y)=(xy)^{\a+\frac12} \exp\big(-\frac12(x^2+y^2)\cosh t\big) \un_{xy\le1}\,, \\
   &\f_2(x, y)=e^{xy} \exp\big(-\frac12(x^2+y^2)\cosh t\big)\un_{xy>1}\,.
 \end{align*}
Let $\Phi_i$ be the integral operator corresponding to $\f_i$. Then by the H\"older inequality, we have
  \begin{align*}
  \big\|\Phi_1\big\|_{B(L_p(\real_+; X))}
  &\le \Big(\int_0^\8 x^{p(\a+\frac12)}e^{- p x^2\frac{\cosh t}2}dx\Big)^{\frac1p}
  \Big(\int_0^\8 y^{p'(\a+1/2)}e^{- {p'}y^2\frac{\cosh t}2}dy\Big)^{\frac1{p'}}\\
  &=C_{\a,p} (\cosh t)^{-(\a+1)}\le C_{\a,p}\,,
  \end{align*}
 where we have used the assumption that $p_\a<p<p_\a'$. Noting that $\f_2(x, y)\le e^{-(x-y)^2/2}$, we see that $ \big\|\Phi_2\big\|_{B(L_p(\real_+; X))}\le1$.

 In particular, $\{T_t\}_{t>0}$ is a symmetric semigroup of positive contractions on $L_2(\real_+)$, so analytic. Applying the previous sections to the associated subordinated Poisson semigroup $\{P_t\}_{t>0}$, we recover the result of \cite{betancor2} that $X$ is of Luzin cotype $q$ relative to $\{P_t\}_{t>0}$ on $L_2(\real; X)$ whenever $X$ is of martingale cotype $q$. Moreover, \cite{betancor2} shows that the converse is also true. Note that \cite{betancor2} also extends this result to all $p\in(p_\a,\, p_\a')$.

 We do not know, however,  to determine the analyticity of $\{T_t\}_{t>0}$ on $L_p(\real; X)$ for the range $-1<\a<-\frac12$.\end{ex}

\begin{rk} We would like to point out an interesting phenomenon revealed by this example. It is easy to get a semigroup of contractions on $L_2$ thanks to spectral theory. If the contractions are further positive (or regular), the results of the previous sections apply. In many concrete examples, one can then extrapolate $L_2$  to $L_p$ using tools from harmonic analysis. This is indeed the case for all previous examples.\end{rk}

\begin{ex}{\bf Ornstein-Uhlenbeck semigroup}. Now $\real^d$ is equipped with the canonical Gaussian measure $\g_d$. Let $\{T_t\}_{t>0}$ be the  Ornstein-Uhlenbeck semigroup on $\real^d$ whose negative generator is  given by $A=-\Delta +x\cdot\nabla$. This is again a symmetric diffusion semigroup.  By \cite{pis2}, $\{T_t\}_{t>0}$ is analytic on $L_p(\real^d, \g_d; X)$ iff $X$ is K-convex (a property weaker than the finite martingale cotype). On the other hand, by \cite{LP1}, $X$ is of Luzin cotype $q$ relative to the Poisson semigroup subordinated to $\{T_t\}_{t>0}$  iff $X$ is of martingale cotype $q$. We then deduce that the Luzin cotype $q$ of $X$ relative to $\{T_t\}_{t>0}$ itself  on $L_p(\real^d, \g_d; X)$ characterizes the martingale cotype $q$ of $X$. It is worth noting that in contrast to \cite{LP1}, all estimates obtained by the present method or by \cite{LP2} are dimension free.\end{ex}

\begin{ex}{\bf Walsh semigroup}.  Let $\O=\{-1, 1\}^\nat$ be the dyadic group equipped with normalized Haar measure. The coordinate functions $\{\e_n\}_{n\ge1}$  on $\O$ form an independent sequence of symmetric random variables (Rademacher functions). We introduce the Walsh system $(w_A)$:  for any finite subset
$A\subset \nat$ let
 $$w_A=\prod_{k\in A}\e_k\,.$$
If $A=\emptyset$, $w_\emptyset=1$. All such $w_A$'s form an orthonormal basis of $L_2(\O)$. Any $f\in L_2(\O)$ admits the following Fourier expansion:
 $$f=\sum_{A} \a_Aw_A.$$
Define
 $$T_t(f)=\sum_{A} e^{-t |A|}\a_Aw_A.$$
Then $\{T_t\}_{t>0}$ is a symmetric diffusion semigroup  on $\O$, it can be viewed as a baby model of the Ornstein-Uhlenbeck semigroup. Again, by \cite{pis2}, $\{T_t\}_{t>0}$ is analytic on $L_p(\O; X)$ iff $X$ is K-convex.

\begin{rk}  Let $\{T_t\}_{t>0}$ be the semigroup in the above example. Then the Luzin cotype relative to $\{T_t\}_{t>0}$ characterizes the martingale cotype.\end{rk}

Indeed, assume that  $X$ is of Luzin cotype $q$ relative to $\{T_t\}_{t>0}$. Then by an argument via the central limit theorem as in \cite{Bec}, we can show that  $X$ is of Luzin cotype $q$ relative to the Ornstein-Uhlenbeck semigroup too, so by the previous example, $X$ is of martingale cotype $q$.\end{ex}

It would be interesting to show the analogue of Corollary~\ref{NY} for Ornstein-Uhlenbeck or Walsh semigroup.

\begin{problem}
 Let $\{T_t\}_{t>0}$ be the Ornstein-Uhlenbeck or Walsh semigroup as above and $X$ be of martingale cotype $q$. Does one have
 $$
 \mathsf{L}^{T}_{\cc,q, p}(X)
 \les\max\big(p^{\frac1q},\, p'\big) \mathsf{M}_{\cc,q}(X)?$$
 \end{problem}

It would, of course, suffice to consider the Walsh case. On the other hand, it is likely that in the Ornstein-Uhlenbeck setting one could get a dimension dependent estimate $
 \mathsf{L}^{T}_{\cc,q, p}(X)\les_d\max\big(p^{\frac1q},\, p'\big) \mathsf{M}_{\cc,q}(X)$ by standard techniques on the Ornstein-Uhlenbeck semigroup. However, here a dimension free estimate is more important than the corresponding one for the heat semigroup in $\real^d$ in view of analysis in Wiener space.

\begin{ex}{\bf Translation semigroup}.  We have already used the translation semigroup $\{\tau_t\}_{t>0}$ in in the proof of Theorem~\ref{Poisson ML}. It is not analytic on $L_p(\real)$ for any $p$. By \eqref{translation to Poisson}, if $X$ is of Luzin cotype $q$ relative to  the  subordinated Poisson semigroup $\{P^\tau_t\}_{t>0}$ on $L_p(\real;X)$, so is $X$ relative to the Poisson semigroup subordinated to any semigroup of regular contractions  $\{T_t\}_{t>0}$ on $L_p(\O; X)$. Consequently, the Luzin cotype relative to $\{P^\tau_t\}_{t>0}$ is equivalent to the martingale cotype.\end{ex}

\begin{ex}{\bf $L_2$-theory}. Let $A$ be a positive densely defined operator on $L_2(\O)$. Then $-A$ generates an analytic semigroup $\{T_t\}_{t>0}$ of contractions on $L_2(\O)$. Being positivity preserving can be characterized by means of the Dirichlet form associated to $A$ (cf. \cite[Theorem~1.3.2]{Davies}). Many classical examples are built in this way. Note, however, that it can happen that $\{T_t\}_{t>0}$ extends to a semigroup of bounded operators on $L_p(\O)$ for $p$ only in a  small symmetric interval around $2$ as shown by the Laguerre semigroup. Even worse, it can happen that $\{T_t\}_{t>0}$ does not extrapolate to $L_p(\O)$ for any $p\neq2$.\end{ex}


\centerline{\bf Added in proof}

After the submission of this article for publication, a few related works have appeared. For instance, A.K.~Lerner, E.~Lorist and S.~Ombrosi \cite{Lerner} found a new proof of Theorem~\ref{fML} (i) by the sparse domination principle,  T.P.~Hyt\"onen and S.~Lappas \cite{HL} obtained an estimate close to that appearing in Theorem~\ref{Heat ML} (ii), and G. Hong, Z. Xu and H. Zhang  \cite{HXZ} partially resolved Problem 1.8,  Problem A.1 and Conjecture A.4.

\bigskip \n{\bf Acknowledgements.}  I am extremely grateful to Assaf Naor for many inspiring communications as well as for pointing the problems in his paper \cite{NaYo1} joint with Robert Young; the resolution of their problems has forced me to invent new techniques, thereby has lead to Theorem~\ref{fML}, Corollary~\ref{NY} and a significant improvement on the orders of the constants in a preliminary version of Theorem~\ref{Poisson ML} and Theorem~\ref{Heat ML}; our communications have also  been a special impulse to the work  \cite{LP-Optimality} on the optimal orders of the best constants in the classical Littlewood-Paley inequalities (Assaf also asked himself these questions on the classical case in his own research). I also wish to thank Alexandros Eskenazis, Zhendong Xu and Hao Zhang for useful discussions and careful readings of various versions of the article. I learnt the existence of \cite{Weis01, KU} after the submission of this article for publication, and I thank Emiel Lorist for pointing out these references to me. Finally, I want to thank the three anonymous referees for their numerous helpful comments and suggestions; one of them has pointed out to me the references \cite{Lis, Kriegler} that have lead to an improvement on the estimate of the constant $\mathsf{L}^{T}_{\cc,p}$ in Corollary~\ref{scalar heat LPS}.

This work is partially supported by the French ANR project (No. ANR-19-CE40-0002) and the Natural Science Foundation of China (No.12031004).

\bigskip



\begin{thebibliography}{10}


\bibitem{AFST}
I.~Abu-Falahah, P. R.~Stinga, and J. L.~Torrea. Square functions associated to {S}chr\"{o}dinger operators. \textit{Studia Math.} 203 (2011), 171-194.
\bibitem{Ak}
M. A.~Akcoglu. A pointwise ergodic theorem in $L_p$-spaces. \textit{Can. J. Math.} 27 (1975), 11075-1082.

 \bibitem{Bec}
W.~Beckner. Inequalities in Fourier analysis. \textit{Ann. Math.} 102 (1975), 159-182.

\bibitem{bl}
J.~Bergh, and J.~L{\"o}fstr{\"o}m. \textit{Interpolation spaces.}  Springer-Verlag, Berlin, 1976.

\bibitem{betancor-1}
 J. J.~Betancor,  A. J.~Castro, J.C.~Fari\~{n}a, and  L.~Rodr\'{\i}guez-Mesa. UMD Banach spaces and square functions associated with heat semigroups for Schr\"odinger, Hermite, and Laguerre operators. \textit{Math. Nachr.} 289 (2016), 410-435.

\bibitem{betancor0}
 J. J.~Betancor,  A. J.~Castro, and  L.~Rodr\'{\i}guez-Mesa. Characterization of uniformly convex and smooth {B}anach spaces by using {C}arleson measures in {B}essel settings.
\textit{J. Convex Anal.} 20 (2013), 763-811.

\bibitem{betancor1}
 J. J.~Betancor, J.C.~Fari\~{n}a, V.~Galli, and S.~Molina. Uniformly convex and smooth {B}anach spaces and {$L^p$}-boundedness properties of {L}ittlewood-{P}aley and area functions associated with semigroups. \textit{J. Math. Anal. Appl.} 482 (2020), no.1, 123534, pp 56.

\bibitem{betancor1b}
 J. J.~Betancor,  J. C.~Fari\~{n}a, T.~Mart\'{\i}nez, and J. L.~Torrea. Riesz transform and {$g$}-function associated with {B}essel operators and their appropriate {B}anach spaces.
 \textit{Israel J. Math.} 157 (2007), 259-282.

\bibitem{betancor2}
J. J.~Betancor, J.C.~Fari\~{n}a, L.~Rodr\'{\i}guez-Mesa, A.~Sanabria, and J.L.~Torrea.  Luzin type and cotype for {L}aguerre {$g$}-functions. \textit{Israel J. Math.} 182 (2011), 1-30.

\bibitem{betancor3}
J. J.~Betancor, S.~Molina, and L.~Rodr\'{\i}guez-Mesa.  Area {L}ittlewood-{P}aley functions associated with {H}ermite and {L}aguerre operators. \textit{Potential Anal.} 34 (2011), 345-369.



\bibitem{CWW}
S-Y. A.~Chang, J.~M.~Wilson, and T. Wolff. Some weighted norm inequalities concerning the Schr\"odinger operator. \textit{Comm. Math. Helv.} 60 (1985), 217-246.

\bibitem{Cow}
 M. Cowling. Harmonic analysis on semigroups. \textit{Ann. Math.} 117 (1983), 267-283.

\bibitem{CDMY}
  M.~Cowling, I.~Doust, A.~McIntosh, and A.~Yagi. Banach space operators with a bounded $H^{\infty}$ functional calculus. \textit{J. Aust. Math. Soc.} 60 (1996), 51-89.

\bibitem{Davies1}
E. B.~Davies.  Gaussian upper Bounds for the heat kernels of some second-order operators on Riemannian manifolds. \textit{J. Funct. Anal.} 80 (1988), 16-32.

\bibitem{Davies}
E. B.~Davies. \textit{Heat kernel and spectral theory}. Cambridge University Press, 1989.

\bibitem{GRF}
J.~Garc\'{\i}a-Cuerva, and J.L. Rubio de Francia. \textit{Weighted norm inequalities and related topics}. North-Holland Publishing Co., Amsterdam, 1985.


\bibitem{Fe}
 G. Fendler. Dilations of one parameter semigroups of positive contractions on $L_p$-spaces.  \textit{Canad. J. Math.} 49 (1997), 736-748.

\bibitem{Fe2}
 G. Fendler.  On dilations and transference for continuous one-parameter semigroups of positive contractions on $L^p$-spaces.  \textit{Ann. Univ. Sarav. Ser. Math.} 9 (1998),  number 1, iv+97 pages.

\bibitem{F-Mei-S}
T.~Ferguson, T.~Mei, B.~Simanek. {$H^\infty$}-calculus for semigroup generators on {BMO}.  \textit{Adv. Math.} 347 (2019), 408-441.

\bibitem{HY}
Y.~Han, and Y.~Meyer. A characterization of Hilbert spaces and the vector-valued Littlewood-Paley theorem. \textit{Meth. Appl. Anal.} 3 (1996), 228-234.

\bibitem{HTV}
E.~Harboure, J. L.~Torrea, and B.~Viviani. Vector-valued extensions of operators related to the Ornstein-Uhlcnbeck semigroup.  \textit{J. Anal. Math.} 91 (2003), 1-29.


\bibitem{HM}
 G.~Hong, and T. Ma. Vector valued $q$-variation for differential operators and semigroups I. \textit{Math. Z.} 286 (2017), 89-120.

\bibitem{HLM}
 G.~Hong, W.~Liu, and T. Ma. Vector-valued $q$-variational inequalities for averaging operators and Hilbert transform. \textit{Achiv der Math.} 115 (2020), 423-433.
 
 \bibitem{HXZ}
 G.~Hong, Z.~Xu, and H. Zhang. Best constants in the vector-valued Littlewood-Paley-Stein theory.  \textit{arXiv}, 2024.

 \bibitem{Hy}
 T.P.~Hyt\"onen. Littlewood-Paley-Stein theory for semigroups in UMD spaces. \textit{Rev. Mat. Iberoamericana} 23 (2007), 973-1009.
 
  \bibitem{HL}
 T.P.~Hyt\"onen,  and S.~Lappas. Quantitative estimates for bounded holomorphic semigroups. Preprint, . \textit{arXiv}, 2022.

\bibitem{HN}
 T.P.~Hyt\"onen, and A.~Naor. Heat flow and quantitative differentiation. \textit{J. Euro. Math. Soc.} 21 (2019), 3415-3466.

\bibitem{HVVW}
T~Hyt\"{o}nen, J.~van Neerven, M.~Veraar, and L.~Weis. \textit{Analysis in {B}anach spaces. {V}ol. {II}}. Springer,  2017.

\bibitem{JLMX}
  M.~Junge, C. Le Merdy, and Q.~Xu. $H^{\infty}$ functional calculus and square functions on noncommutative $L^p$-spaces. \textit{Ast\'erisque}. vol. 305, 2006. vi+138 pages.


\bibitem{JX}
 M.~Junge, and Q.~Xu. Noncommutative maximal ergodic inequalities. \textit{J. Amer. Math. Soc.} 20 (2007), 385-439.



\bibitem{Krengel}
 U.~Krengel. \textit{Ergodic Theory.} De Guyter, 1985.
 
 \bibitem{Kriegler}
 C.~Kriegler. Analyticity angle for non-commutative diffusion semigroups. \textit{J. Lond. Math. Soc.}  83 (2011), 168-186.

\bibitem{KU}
 P.~Kunstmann, and A.~Ullmann. $R_s$-sectorial operators and generalized Triebel-Lizorkin spaces.  \textit{J. Fourier Anal. Appl.} 20 (2014), 135-185,

\bibitem{LaNa}
 V.~Lafforgue, and A.~Naor. Vertical versus horizontal {P}oincar\'{e} inequalities on the {H}eisenberg group.  \textit{Israel J. Math.} 203 (2014), 309-339.

 \bibitem{Leb}
 N. N.~Lebedev.  \textit{Special functions and their applications}. Dover, New York, 1972.
 
 \bibitem{Lerner}
  A.K.~Lerner, E.~Lorist, and S.~Ombrosi. Operator-free sparse domination. \textit{Forum Math. Sigma} 10 (2022), e15, 1-28.
 
  \bibitem{Lis}
  V.A.~Liskevich, and M.A.~Perelmuter. Analyticity of sub-Markovian semigroups. \textit{Proc. Amer. Math. Soc.} 123 (1995), 1097-1104.


 \bibitem{LMX}
  C.~Le Merdy, and Q.~Xu. Maximal theorems and square functions for analytic operators on $L_p$-spaces. \textit{J. London Math. Soc.} 86 (2012), 343-365.

 \bibitem{LY}
 P.~Li, and  S. T. Yau. On the parabolic kernel of the Schr\"odinger operator. \textit{Acta Math.} 156 (1986), 153-201.

 \bibitem{Marc}
J. L.~Marcolino Nhani. La structure des sous-espaces de treillis. \textit{Dissertationes Math.} 397 (2001), 50 pp.



\bibitem{LP1}
T. Mart\'{\i}nez, J. L.  Torrea, and Q. Xu. Vector-valued Littlewood-Paley-Stein theory for semigroups. \textit{Adv. Math.} 203 (2006), 430-475.

\bibitem{MY}
  A.~McIntosh, and A.~Yagi. Operators of type $\omega$ without a bounded $H^\8$ functional calculus. \textit{Miniconference on Operators in Analysis, Proceedings of the Centre for Mathematics and its Applications} 24 (Australian National University, Canberra, 1989) 159-172.

\bibitem{MN}
P. Meyer-Nieberg. \textit{Banach lattices.} Springer-Verlag Berlin Heidelberg, New York, 1991.

\bibitem{Na12}
A.~Naor. On the Banach-space-valued Azuma inequality and small-set isoperimetry of Alon-Roichman graphs.  \textit{Comb. Probab. Comput.} 21 (2012), 623--634.

\bibitem{Na14}
A.~Naor. Comparison of metric spectral gaps.  \textit{Anal. Geom. Metr. Spaces} 2 (2014), 1-52.

\bibitem{NaYo}
A.~Naor, and R.~Young. Vertical perimeter versus horizontal perimeter.  \textit{Ann. Math.} 188 (2018), 171-279.

\bibitem{NaYo1}
A.~Naor, and R.~Young. Foliated corona decompositions.  \textit{Acta Math.} 229 (2022), 55-200.

\bibitem{NS}
A.~Nowak, and K.~Stempak. On $L_p$-contractivity of Laguerre semigroups.  \textit{Illinois J. Math.} 56 (2012), 433-452.


\bibitem{OX}
C.~Ouyang, and Q.~Xu. BMO functions and Carleson measures with values in uniformly convex spaces. \textit{Can. J. Math.} 62 (2010), 827-844.

\bibitem{pazy}
  A. Pazy. \textit{Semigroups of linear operators and applications to partial differential equations.} Springer-Verlag Berlin Heidelberg, New York, 1983.

\bibitem{pis1}
G. Pisier. Martingales with values in uniformly convex spaces.   \textit{Israel J. Math.}, 20 (1975), 326-350.

\bibitem{pis2}
 G. Pisier. Holomorphic semigroups and the geometry of Banach spaces. \textit{Ann. Math. } 115 (1982), 375-392.

\bibitem{pis2b}
 G. Pisier.  Probabilistic methods in the geometry of Banach spaces.\textit{Lect. Notes in Math.}  Springer-Verlag, Berlin, 1206 (1986), 167-241.

 \bibitem{pis3}
G. Pisier. \textit{Martingales in Banach spaces}. Cambridge University Press, 2016.



\bibitem{stein}
E.M.~Stein. \textit{Topics in harmonic analysis related to the Littlewood-Paley theory.}  Ann. Math. Studies, Princeton, University Press, 1970.


\bibitem{TZ}
J. L.~Torrea, and C.~Zhang. Fractional vector-valued {L}ittlewood-{P}aley-{S}tein theory for semigroups. \textit{Proc. Roy. Soc. Edinburgh Sect. A}. 144 (2014), 637-667.

\bibitem{Uch}
A.~Uchiyama. A constructive proof of the Fefferman-Stein decomposition of BMO $({\bf R}^{n})$. \textit{Acta Math.} 148 1(1982), 215--241.

\bibitem{VW}
M.~Veraar, and L.~Weis. Estimates for vector-valued holomorphic functions and Littlewood-Paley-Stein theory.  \textit{Studia Math.} 228 (2015), 73-99.

\bibitem{Weis01}
L.~Weis. A new approach to maximal $L_p$-regularity (2001). In: \textit{Evolution Equations and Their Applications in Physical and Life Sciences} (Bad Herrenalb, 1998.) Lecture Notes in Pure and Appl. Math., vol. 215, pp. 195-214. Dekker, New York (2001).

\bibitem{Wilson89}
J.~M.~Wilson. Weighted norm inequalities for the continuous square function. \textit{Trans. Amer. Math. Soc.} 314 (1989, 661--692.

\bibitem{Wilson07}
J.~M.~Wilson. The intrinsic square function. \textit{Rev. Mat. Iberoam.} 23 (2007), 771--791.

\bibitem{Wilson08}
J.~M.~Wilson. \textit{Weighted {L}ittlewood-{P}aley theory and exponential-square  integrability}. Lecture Notes in Math. Springer, Berlin,  1924 (2008), pp. xiv+224.

\bibitem{XXX}
 R.~Xia, X.~Xiong, and Q.~Xu. Characterizations of operator-valued Hardy spaces and applications to harmonic analysis on quantum tori. \textit{Adv. Math.}  291 (2016), 183-227.

\bibitem{LP0}
 Q.~Xu. Littlewood-{P}aley theory for functions with values  in uniformly convex spaces. \textit{J. Reine Angew. Math.} 504 (1998), 195-226.


  \bibitem{LP3}
 Q.~Xu. $H^\8$ functional calculus and maximal  inequalities for semigroups of contractions on  vector-valued $L_p$-spaces. \textit{Int. Math. Res. Not.} 14 (2015), 5715-5732.

 \bibitem{LP2}
 Q.~Xu. Vector-valued Littlewood-Paley-Stein theory for semigroups II. \textit{Int. Math. Res. Not.}  21 (2020), 7769-7791.

\bibitem{LP-Optimality}
Q.~Xu. Optimal orders of the best constants in the Littlewood-Paley $g$-function inequalities. \textit{J. Funct. Anal.} 283 (2022), 109570.

\bibitem{XZ}
Z.~Xu, and H. Zhang. From the Littlewood-Paley-Stein inequality to the Burkholder-Gundy inequality.  \textit{Trans. Amer. Math. Soc.} 376 (2023), 371-389.



\end{thebibliography}
\end{document}